\documentclass[12pt,reqno]{amsart}
\usepackage[cp1251]{inputenc}
\usepackage[T2A]{fontenc}
\usepackage[russian,english]{babel}
\usepackage{amsmath, amsthm, amscd, amsfonts, amssymb, graphicx, color}
\usepackage[bookmarksnumbered, colorlinks, plainpages]{hyperref}

\numberwithin{equation}{section}

\newtheorem{theorem}{Theorem}[section]
\newtheorem{corollary}[theorem]{Corollary}
\newtheorem{lemma}[theorem]{Lemma}
\newtheorem{proposition}[theorem]{Proposition}

\theoremstyle{remark}
\newtheorem{remark}[theorem]{Remark}

\theoremstyle{definition}

\newtheorem*{main-definition}{Main Definition}

\begin{document}

\title{Distribution spaces associated with elliptic operators}


\author[I. Chepurukhina]{Iryna Chepurukhina}

\address{Institute of Mathematics of the National Academy of Sciences of Ukraine, 3 Tereshchen\-kivs'ka, Kyiv, 01024, Ukraine}

\email{Chepuruhina@gmail.com, chepurukhina@imath.kiev.ua}


\author[A. Murach]{Aleksandr Murach}

\address{Institute of Mathematics of the National Academy of Sciences of Ukraine, 3 Tereshchen\-kivs'ka, Kyiv, 01024, Ukraine}

\email{murach@imath.kiev.ua}





\begin{abstract}
We study complex distribution spaces given over a bounded Lipschitz domain $\Omega$ and associated with an elliptic differential operator $A$ with $C^{\infty}$-coefficients on $\overline{\Omega}$. If $X$ and $Y$ are quasi-Banach distribution spaces over $\Omega$, then the space $X(A,Y)$ under study consists of all distributions $u\in X$ such that $Au\in Y$ and is endowed with the graph quasi-norm. Assuming $X$ to be an arbitrary Besov space or Triebel--Lizorkin space over $\Omega$, we find sufficient conditions for $Y$ under which the interpolation between the spaces $X(A,Y)$ preserves their structure, these spaces are separable, and the set $C^{\infty}(\overline{\Omega})$ is dense in them. We then explicitly describe the spaces obtained by the real, complex, and $\pm$ interpolation between the spaces under study. We apply these spaces to general elliptic problems with rough boundary data by proving the Fredholm property for bounded operators induced by these problems and defined on certain spaces $X(A,Y)$. Specifically, we establish the maximal regularity of solutions to some elliptic problems with Gaussian white noise in boundary conditions. Quasi-Banach distribution spaces are involved in the concept of $X(A,Y)$ for the first time. Our results are new even for inner product Sobolev spaces of integer-valued order.
\end{abstract}

\maketitle

\section{Introduction}\label{sec1}

The paper is devoted to distribution spaces of the form $X(A,Y):=\{u\in X:Au\in Y\}$ provided that $X$ and $Y$ are complex Banach or quasi-Banach distribution spaces over a bounded Lipschitz domain $\Omega$ and that $A$ is an elliptic linear partial differential operator (PDO) on $\overline{\Omega}$ with complex-valued coefficients of class $C^{\infty}(\overline{\Omega})$. Naturally, the space $X(A,Y)$ is endowed with the graph quasi-norm $\|u\|_{X}+\|Au\|_{Y}$. Such spaces arise in the theory of elliptic PDOs and elliptic boundary-value problems.

The first historical example of $X(A,Y)$ was given by letting $X=Y=L_{2}(\Omega)$. In this case, $X(A,Y)$ is the domain of the maximal operator induced by the unbounded operator $A$ in the Hilbert space $L_{2}(\Omega)$ \cite{Birman57, Hermander55, Hermander58, Ladyzhenskaya51}. Lions and Magenes \cite{LionsMagenes62V, LionsMagenes63VI} systematically used the space $X(A,Y)$ in the case where $X$ is the normed Sobolev space $H^{s}_{p}(\Omega)$ with smoothness index (or order) $s<\mathrm{ord}\,A=:2l$ and exponent $p\in(1,\infty)$ and where $Y=L_{p}(\Omega)=H^{0}_{p}(\Omega)$. Then $X(A,Y)$ serves as a solution space of regular elliptic problems posed for the equation $Au=f$ and containing boundary data of low regularity (in other wards, rough boundary data). In the case of the Dirichlet boundary-value problem, Lions and Magenes \cite{LionsMagenes61II, LionsMagenes61III} let $Y:=H^{-l}_{p}(\Omega)$ to obtain a wider solution space. Their fundamental monograph \cite[Chapter~2, Section 6 and~7]{LionsMagenes72} uses the Hilbert space $X(A,Y)$ as the solution space in the case where $X=H^{s}_{2}(\Omega)$ with $s<2l$ and when $Y$ is a certain weighted Sobolev space of negative order $s-2l$. Various classes of spaces $X(A,Y)$ were applied to elliptic problems with data of low regularity in \cite{DenkPlossRauSeiler23, KasirenkoMikhailetsMurach19, MikhailetsMurach14, Murach09MFAT2} in the framework of Sobolev spaces, \cite{ChepurukhinaMurach21Dop6, Murach94Dop12, MurachChepurukhina20Dop8} concerning Besov, Nikolskii, and Triebel--Lizorkin spaces, and \cite{AnopDenkMurach21CPAA1, AnopKasirenkoMurach18UMJ3, ChepurukhinaMurach20MFAT2, MikhailetsMurach09OperatorTheory191, MikhailetsMurach12BJMA2, MikhailetsMurach14, MurachChepurukhina15UMJ5} involving distribution spaces of generalized smoothness.

Our purpose is to investigate properties of the spaces $X(A,Y)$ concerning interpolation between these spaces, their separability, and dense sets in them. We give applications of such spaces to general elliptic problems with rough boundary data by proving Fredholm property for bounded operators induced by these problems and defined on certain spaces $X(A,Y)$. A special case where $X$ and $Y$ are inner product Sobolev spaces was studied in \cite{KasirenkoMikhailetsMurach19} for the case where $\partial\Omega$ is of class $C^{\infty}$. Nevertheless, our results are new even for these Sobolev spaces (moreover, with integer-valued smoothness index) as for properties of $X(A,Y)$ and applications to elliptic problems. To our best knowledge, quasi-Banach distribution spaces are involved in the concept of $X(A,Y)$ for the first time.

The paper consists of 7 sections. Section~\ref{sec1} is Introduction. Section~\ref{sec2} studies the spaces $X(T,Y)$ in the general setting where $T$ is an arbitrary linear bounded operator from a quasi-Banach space $X$ to a linear topological space containing another quasi-Banach space $Y$. We indicate sufficient conditions for the interpolation between these spaces to give a space of the same form (Theorem~\ref{th-basic}) and find sufficient conditions for the separability of $X(T,Y)$ and the density of sets in $X(T,Y)$ (Theorem~\ref{th-sep-dense}). These results serve as a base for Section~\ref{sec3}, which studies distribution spaces $X(A,Y)$ associated with the elliptic PDO $A$. We assume that $X$ is a Besov space or Triebel--Lizorkin space over $\Omega$ with arbitrarily smoothness index and exponents and that $Y$ contains one of such spaces. Sobolev, Nikolskii, and H\"older--Zygmund spaces fit this setting. We describe a result of the interpolation between such spaces $X(A,Y)$ by any interpolation functor (Theorem~\ref{main-interp-th}) and specify this result for the functors of the real (Theorems \ref{R-interp} and \ref{R-interp-special}), complex (Theorem~\ref{C-interp}), and $\pm$ interpolation (Theorem~\ref{plus-minus-interp}). We also find sufficient conditions for the space $X(A,Y)$ to be separable and for the set $C^{\infty}(\overline{\Omega})$ to be dense in this space (Theorem~\ref{th-distr-sep-dense}). This implies that $X(A,Y)$ is separable and that $C^{\infty}(\overline{\Omega})$ is dense in $X(A,Y)$ if $X$ and $Y$ are Besov or Triebel--Lizorkin spaces with finite exponents (Theorem~\ref{B-F-sep-dense}).

The next Section~\ref{sec4} gives applications of spaces $X(A,Y)$ to a general elliptic problem with boundary data of low regularity, the boundary of $\Omega$ being assumed to be of class $C^{\infty}$. Using various interpolation functors, we prove that this problem sets Fredholm bounded operators between appropriate Besov spaces or Triebel--Lizorkin spaces with finite exponents (Theorem~\ref{th-Fredholm}) and between Nikolskii spaces (Theorem~\ref{th-Fredholm-Nikolskii}). As a consequence, we find a sufficient condition for $Y$ under which the space $X(A,Y)$ does not depend on the sum-exponent $q$ provided that $X$ is the Triebel--Lizorkin space $F^{s}_{p,q}(\Omega)$ (Theorem~\ref{source-space-F}). We also give an application of Theorem~\ref{th-Fredholm-Nikolskii} to certain elliptic problems with white noise in boundary data by obtaining the maximal regularity of their solutions in terms of negative Nikolskii spaces and a corresponding \textit{a~priori} estimate of the solutions (Theorem~\ref{th-white-noise}).

Section~\ref{sec5} gives historical comments on results obtained in two previous sections. Sections \ref{sec6} and \ref{sec7} are appendices. The first of them discusses the structure of a common range complement of the Fredholm operators induced by the elliptic problem and the structure of the kernel of the adjoint operators. The last Section~\ref{sec7} gives proofs of Lemma~\ref{prop-interp-Fredholm} and Corollary~\ref{cor-interp-Fredholm} about sufficient conditions for the Fredholm property to be preserved under the interpolation of bounded linear operators by an arbitrary interpolation functor.

\section{Basic results}\label{sec2}

In the paper we consider complex linear spaces and Hausdorff topological spaces. Let $X$ and $Y$ be quasi-Banach spaces (specifically, they may be Banach spaces), $\Upsilon$ be a linear topological space, and $T:X\to\Upsilon$ be a continuous linear operator. We assume that $Y$ is continuously embedded in $\Upsilon$; i.e., $Y$ is a linear manifold in~$\Upsilon$, and the identity mapping is a continuous operator from $Y$ to $\Upsilon$. As usual, this property is denoted by $Y\hookrightarrow\Upsilon$.

Consider the linear space
\begin{equation}\label{space-X(T,Y)}
X(T,Y):=(X)(T,Y):=\{u\in X:Tu\in Y\}
\end{equation}
endowed with the graph quasi-norm
\begin{equation}\label{norm-X(T,Y)}
\|u\|_{X(T,Y)}:=\|u\|_{X}+\|Tu\|_{Y}.
\end{equation}
This quasi-normed space is complete and does not depend on $\Upsilon$.
Of course, if $X$ and $Y$ are Banach spaces, then so is $X(T,Y)$. The completeness is demonstrated as follows:

If $(u_k)_{k=1}^{\infty}$ is a Cauchy sequence in $X(T,Y)$, there exist vectors $u\in X$ and $v\in Y$ such that
$u_k\to u$ in $X$ and that $Tu_k\to v$ in $Y$ as $k\rightarrow\infty$. It follows from the first convergence that $Tu_k\to Tu$ in $\Upsilon$; thus, $Tu=v\in Y$, and $u_k\to u$ in $X(T,Y)$.

Let $\mathfrak{F}$ be an arbitrary interpolation functor defined on the category of all interpolation pairs of quasi-Banach spaces or on the category of all interpolation pairs of Banach spaces. Recall \cite[Section~1.2.1]{Triebel95} that an ordered pair $[X_{0},X_{1}]$ of quasi-Banach spaces $X_{0}$ and $X_{1}$ is called an interpolation one if these spaces are continuously embedded in a certain linear topological space. Given such a pair, we obtain the quasi-Banach spaces $X_{0}\cap X_{1}$ and $X_{0}+X_{1}$ endowed with the quasi-norms
\begin{equation*}
\|u\|_{X_0\cap X_1}:=
\max\bigl\{\|u\|_{X_0},\|u\|_{X_1}\bigr\}
\end{equation*}
and
\begin{equation*}
\|u\|_{X_0+X_1}:=
\inf\bigl\{\|u_0\|_{X_0}+\|u_1\|_{X_1}:u_0\in X_0,\,u_1\in X_1,\, u_0+u_1=u\bigr\},
\end{equation*}
respectively. Of course, if $X_{0}$ and $X_{1}$ are Banach spaces, then so are $X_{0}\cap X_{1}$ and $X_{0}+X_{1}$.

Recall also \cite[Definition~1.2.2/1]{Triebel95} that $\mathfrak{F}$ is characterized by the following properties:
\begin{itemize}
\item[(a)] If $[X_{0},X_{1}]$ is an interpolation pair of quasi-Banach spaces (resp., Banach spaces), then $\mathfrak{F}[X_{0},X_{1}]$ is a quasi-Banach space (resp., Banach space) such that the continuous embeddings $X_{0}\cap X_{1}\hookrightarrow\mathfrak{F}[X_{0},X_{1}]\hookrightarrow X_{0}+X_{1}$ hold true.
\item[(b)] If  $[X_{0},X_{1}]$ and $[Y_{0},Y_{1}]$~are interpolation pairs of quasi-Banach spaces (resp., Banach spaces) and if $T:X_{0}+X_{1}\to Y_{0}+Y_{1}$ is a linear operator such that its restrictions to the spaces $X_{0}$ and $X_{1}$ are bounded operators $T:X_{0}\to Y_{0}$ and $T:X_{1}\to Y_{1}$, then the restriction of $T$ to the space $\mathfrak{F}[X_{0},X_{1}]$ is a bounded operator $T:\mathfrak{F}[X_{0},X_{1}]\to \mathfrak{F}[Y_{0},Y_{1}]$.
\end{itemize}
We will say that the space $\mathfrak{F}[X_{0},X_{1}]$ is obtained by the interpolation between the spaces $X_{0}$ and $X_{1}$ (by means of the functor $\mathfrak{F}$). We also say that the bounded operator $T:\mathfrak{F}[X_{0},X_{1}]\to \mathfrak{F}[Y_{0},Y_{1}]$ is obtained by the interpolation of the operators $T:X_{j}\to Y_{j}$, with $j\in\{0,1\}$.

Let us formulate sufficient conditions under which the interpolation between spaces of the form \eqref{space-X(T,Y)}, \eqref{norm-X(T,Y)} gives a space of the same form.

\begin{theorem}\label{th-basic}
Suppose that six quasi-Banach spaces (resp., Banach spaces) $X_{0}$, $Y_{0}$, $Z_{0}$, $X_{1}$, $Y_{1}$, and $Z_{1}$ and three linear mappings $T$, $R$, and $S$ are given and satisfy the following conditions:
\begin{itemize}
\item[(i)] The pairs $[X_0,X_1]$ and $[Y_0,Y_1]$ are interpolation ones.
\item[(ii)] The spaces $Y_{0}$, $Z_{0}$, $Y_{1}$, and $Z_{1}$ are linear manifolds in a certain linear space.
\item[(iii)] The continuous embeddings $Y_{0}\hookrightarrow Z_{0}$ and $Y_{1}\hookrightarrow Z_{1}$ hold true.
\item[(iv)] The mapping $T$ acts from $X_0+X_1$ to $Z_0+Z_1$, and its restriction to $X_j$ is a bounded operator $T:X_j\rightarrow Z_j$ whenever $j\in\{0,1\}$.
\item[(v)] The mapping $R$ acts from $Z_0+Z_1$ to $X_0+X_1$, and its restriction to $Z_j$ is a bounded operator $R:Z_j\rightarrow X_j$ whenever $j\in\{0,1\}$.
\item[(vi)] The mapping $S$ acts on $Z_0+Z_1$, and its restriction to $Z_j$ is a bounded operator $S:\nobreak Z_j\rightarrow Y_j$ whenever $j\in\{0,1\}$.
\item[(vii)] The equality $TRw=w+Sw$ is valid for every $w\in Z_0+Z_1$.
\end{itemize}
Then the equality of spaces
\begin{equation}\label{11f1}
\mathfrak{F}[X_{0}(T,Y_{0}),X_{1}(T,Y_{1})]=
(\mathfrak{F}[X_{0},X_{1}])(T,\mathfrak{F}[Y_{0},Y_{1}])
\end{equation}
holds true up to equivalence of quasi-norms, with these spaces being well defined.
\end{theorem}

This theorem is proved in \cite[Chapter~1, Section~14.3]{LionsMagenes72} for the holomorphic (i.e. complex) method of interpolation between Banach spaces, the proof using the construction of this method. A~version of the theorem for the quadratic interpolation (between Hilbert spaces) is proved in \cite[Section~4]{MikhailetsMurach06UMJ11} (see also \cite[Section~3.3.2]{MikhailetsMurach14}).

\begin{proof}[Proof of Theorem $\ref{th-basic}$.]
The spaces $X_{0}(T,Y_{0})$ and $X_{1}(T,Y_{1})$ are well defined due to conditions (iii) and (iv) and are continuously embedded in $X_0+X_1$. Hence, the space on the left of \eqref{11f1} is well defined. The mapping $T$ is a bounded operator from $\mathfrak{F}[X_{0},X_{1}]$ to $\mathfrak{F}[Z_{0},Z_{1}]$ by condition (iv), and the space $\mathfrak{F}[Y_{0},Y_{1}]$ is continuously embedded in $\mathfrak{F}[Z_{0},Z_{1}]$ by condition~(iii). Thus, the space on the right of \eqref{11f1} is also well defined. Since the quasi-Banach spaces on the left and right of \eqref{11f1} are continuously embedded in $X_0+X_1$, it follows from Closed Graph Theorem that we only need to prove \eqref{11f1} as an equality of linear spaces. (As i known \cite[Section~2]{Kalton03}, this theorem is applicable to quasi-Banach spaces.)

It follows from the continuous embeddings $X_j(T,Y_j)\hookrightarrow X_j$, where $j\in\{0,1\}$, that $\mathfrak{F}[X_{0}(T,Y_{0}),X_{1}(T,Y_{1})]$ is continuously embedded in $\mathfrak{F}[X_{0},X_{1}]$. Since $T$ is a bounded operator from $X_j(T,Y_j)$ to $Y_j$ for each $j\in\{0,1\}$, we conclude that $T$ is a bounded operator from $\mathfrak{F}[X_{0}(T,Y_{0}),X_{1}(T,Y_{1})]$ to $\mathfrak{F}[Y_{0},Y_{1}]$. Hence,
\begin{equation}\label{11f6}
\mathfrak{F}[X_{0}(T,Y_{0}),X_{1}(T,Y_{1})]\subset
(\mathfrak{F}[X_{0},X_{1}])(T,\mathfrak{F}[Y_{0},Y_{1}]).
\end{equation}

To prove the inverse inclusion we introduce the mapping
\begin{equation}\label{11f2}
Pu:=u-RTu,\quad\mbox{with}\quad u\in X_0+X_1.
\end{equation}
This mapping is a bounded operator on $X_j$ for each $j=\{0,1\}$ by
conditions (iv) and (v). It follows from conditions (iv), (vi), and (vii) that, if $u\in X_j$, then
\begin{equation*}
TPu=Tu-TRTu=Tu-(Tu+STu)=-STu\in Y_j;
\end{equation*}
i.e., $u\in X_j$ implies $Pu\in X_j(T,Y_j)$. Moreover,
\begin{equation*}
\|Pu\|_{X_j(T,Y_j)}=\|Pu\|_{X_j}+\|TPu\|_{Y_j}=
\|Pu\|_{X_j}+\|STu\|_{Y_j}\leq c_1\|u\|_{X_j},
\end{equation*}
where the number $c_1>0$ does not depend on $u\in X_j$. Thus, the mapping \eqref{11f2} defines the bounded operators
\begin{equation}\label{11f3}
P:X_j\to X_j(T,Y_j),\quad\mbox{with}\quad j\in\{0,1\}.
\end{equation}

Given $j\in\{0,1\}$, we conclude by conditions (iii) and (v)--(vii) that  $TRw=w+Sw\in Y_j$ for every $w\in Y_j$; i.e., $w\in Y_j$ implies $Rw\in X_j(T,Y_j)$. Moreover,
\begin{align*}
\|Rw\|_{X_j(T,Y_j)}&=\|Rw\|_{X_j}+\|TRw\|_{Y_j}\leq
\|Rw\|_{X_j}+c_2\bigl(\|w\|_{Y_j}+\|Sw\|_{Y_j}\bigr)\\
&\leq c_3\|w\|_{Z_j}+c_2\|w\|_{Y_j}+c_4\|w\|_{Z_j}\leq
c_5\|\omega\|_{Y_j},
\end{align*}
where the positive numbers $c_2,\ldots,c_5$ do not depend on $w\in Y_j$. Thus, the mapping $R$ defines the bounded operators
\begin{equation}\label{11f4}
R:Y_j\to X_j(T,Y_j),\quad\mbox{with}\quad j\in\{0,1\}.
\end{equation}

Applying the interpolation to \eqref{11f3} and \eqref{11f4}, we get the bounded operators
\begin{gather*}
P:\mathfrak{F}[X_{0},X_{1}]\to\mathfrak{F}[X_{0}(T,Y_{0}),X_{1}(T,Y_{1})],\\
R:\mathfrak{F}[Y_{0},Y_{1}]\to\mathfrak{F}[X_{0}(T,Y_{0}),X_{1}(T,Y_{1})].
\end{gather*}
Hence, if $u \in \mathfrak{F}[X_{0},X_{1}]$ and $Tu \in \mathfrak{F}[Y_{0},Y_{1}]$, then
$$
u=Pu+RTu\in\mathfrak{F}[X_{0}(T,Y_{0}),X_{1}(T,Y_{1})]
$$
in view of \eqref{11f2}, which means the inverse of~\eqref{11f6}.
\end{proof}

\begin{theorem}\label{th-sep-dense}
Suppose that quasi-Banach spaces $X$, $Y$, $Z$, and a bounded linear operator $\nobreak{T:X\to Z}$ satisfy the following two conditions: the continuous embedding $Y\hookrightarrow Z$ holds true, and there exist bounded linear operators $R:Z\to X$ and $S:Z\to Y$ such that $TRw=w+Sw$ for every $w\in Z$. Then:
\begin{itemize}
\item[(i)] If $X$ and $Y$ are separable, then $X(T,Y)$ is also separable.
\item[(ii)] If a certain set $M\subset X(T,Y)$ is dense in $X(T,Y)$ in the quasi-norm $\|\cdot\|_{X}$ and $T(M)$ is dense in $Y\cap T(X)$ in the quasi-norm $\|\cdot\|_{Y}$ and, moreover, $R(T(M))\subset M$, then $M$ is dense in the quasi-normed space $X(T,Y)$.
\end{itemize}
\end{theorem}

\begin{proof}
Let us use the proof of Theorem~\ref{th-basic} in the case where $X_{0}=X_{1}=X$, $Y_{0}=Y_{1}=Y$, and $Z_{0}=Z_{1}=Z$. According to \eqref{11f2}--\eqref{11f4} we have the bounded operators
\begin{equation}\label{P-R-operators}
P:=I-RT:X\to X(T,Y)\quad\mbox{and}\quad R:Y\to X(T,Y),
\end{equation}
with $I$ being the identity operator on $X$.

Let us prove assertion (i). Suppose that $X$ and $Y$ are separable; i.e., there exist countable sets $X^\circ\subset X$ and $Y^\circ\subset Y$ such that $X^\circ$ is dense in $X$ and that $Y^\circ$ is dense in $Y$. Using \eqref{P-R-operators}, form the countable set
\begin{equation*}
Q:=\{Px+Ry:x\in X^\circ, y\in Y^\circ\}\subset X(T,Y)
\end{equation*}
and show that $Q$ is dense in the quasi-normed space $X(T,Y)$. For every $u\in X(T,Y)$ there exist sequences of vectors $x_n\in X^\circ$ and $y_n\in Y^\circ$ such that $x_n\to u$ in $X$ and that $y_n\to Tu$ in $Y$, as $n\to\infty$. Hence,
\begin{equation*}
Q\ni Px_n+Ry_n\to Pu+RTu=u\quad\mbox{in}\quad X(T,Y).
\end{equation*}
Thus, $X(T,Y)$ is separable.

Let us prove assertion (ii). Suppose that the hypotheses stated in this assertion are valid. Then for arbitrary $u\in X(T,Y)$ there exist sequences of vectors $u_n\in M$ and $v_n\in M$ such that $u_n\to u$ in $X$ and that $Tv_n\to Tu$ in $Y$, as $n\to\infty$. Using the bounded operators \eqref{P-R-operators}, we conclude that
\begin{equation*}
Pu_n+RTv_n\to Pu+RTu=u\quad\mbox{in}\quad X(T,Y).
\end{equation*}
Here, the vectors $Pu_n=u_n-RTu_n$ and $RTv_n$ belong to $M$ by
the hypothesis $R(T(M))\subset M$. Thus, $M$ is dense in the quasi-normed space $X(T,Y)$.
\end{proof}

\section{Main results}\label{sec3}

Let $1\leq n\in\mathbb{Z}$. As usual, $B^{s}_{p,q}(\mathbb{R}^{n})$ denotes the Besov space over $\mathbb{R}^{n}$ with smoothness index or (differentiation) order $s$, integral-exponent $p$ and sum-exponent $q$, and  $F^{s}_{p,q}(\mathbb{R}^{n})$ denotes the Triebel--Lizorkin space over $\mathbb{R}^{n}$ with the same index and exponents (see, e.g., \cite[Section~2.3.1]{Triebel83} or \cite[Section~1.3]{Triebel06}). We consider complex distribution spaces and suppose that
\begin{equation}\label{indexes-case}
\begin{gathered}
s\in\mathbb{R},\quad 0<p\leq\infty, \quad 0<q\leq\infty,\\
\mbox{with}\;\;p\neq\infty\;\;\mbox{for the Triebel--Lizorkin spaces}.
\end{gathered}
\end{equation}
Note that $B^{s}_{p,p}(\mathbb{R}^{n})=F^{s}_{p,p}(\mathbb{R}^{n})$ with equality of norms. Let $\mathcal{U}(\mathbb{R}^{n})$ stand for the class of all quasi-Banach spaces $B^{s}_{p,q}(\mathbb{R}^{n})$ and $F^{s}_{p,q}(\mathbb{R}^{n})$ subject to~\eqref{indexes-case}. They are Banach spaces if and only if $p\geq1$ and $q\geq1$.

The class $\mathcal{U}(\mathbb{R}^{n})$ contains the Lebesgue spaces $L_{p}(\mathbb{R}^{n})=F^{0}_{p,2}(\mathbb{R}^{n})$ where $1<p<\infty$, the local Hardy spaces $h_{p}(\mathbb{R}^{n})=F^{0}_{p,2}(\mathbb{R}^{n})$ where $0<p<\infty$, the Sobolev spaces $H^{s}_{p}(\mathbb{R}^{n})=F^{s}_{p,2}(\mathbb{R}^{n})$ where $s\in\mathbb{R}$ and $1<p<\infty$, the Nikolskii spaces $B^{s}_{p,\infty}(\mathbb{R}^{n})$ where $s>0$ and $1\leq p<\infty$,  the Zygmund spaces $\mathcal{C}^{s}(\mathbb{R}^{n})=B^{s}_{\infty,\infty}(\mathbb{R}^{n})$ where $s>0$, and the H\"older spaces $C^{s}(\mathbb{R}^{n})=B^{s}_{\infty,\infty}(\mathbb{R}^{n})$ where $0<s\notin\mathbb{Z}$, the last two spaces being built on the base of the space of bounded uniformly continuous functions on $\mathbb{R}^{n}$ (see, e.g., \cite[Sections 2.3.5]{Triebel83} or \cite[Section~1.2]{Triebel06}). Every space from $\mathcal{U}(\mathbb{R}^{n})$ is continuously embedded in the linear topological space $\mathcal{S}'(\mathbb{R}^{n})$ of tempered distributions in~$\mathbb{R}^{n}$.

Let $\Omega$ be a nonempty open subset of $\mathbb{R}^{n}$. If $V(\mathbb{R}^{n})$ is a quasi-Banach (resp., Banach) space from the class $\mathcal{U}(\mathbb{R}^{n})$, we consider another quasi-Banach (resp., Banach) space
\begin{equation*}
V(\Omega):=\bigl\{w\!\upharpoonright\!\Omega:
w\in V(\mathbb{R}^{n})\bigr\}
\end{equation*}
endowed with the quasi-norm (resp., norm)
\begin{equation*}
\|u\|_{V(\Omega)}:=
\inf\bigl\{\|w\|_{V(\mathbb{R}^{n})}:w\in
V(\mathbb{R}^{n}),\;u=w\!\upharpoonright\!\Omega\bigr\}
\end{equation*}
where $u\in V(\Omega)$. As usual, $w\!\upharpoonright\!\Omega$ denotes the restriction of the distribution $w\in\mathcal{S}'(\mathbb{R}^{n})$ to~$\Omega$. We thus obtain the Besov space $B^{s}_{p,q}(\Omega)$ and Triebel--Lizorkin space $F^{s}_{p,q}(\Omega)$ over $\Omega$. Let $\mathcal{U}(\Omega)$ stand for the class of all these spaces subject to \eqref{indexes-case}. They are continuously embedded in the linear topological space $\mathcal{D}'(\Omega)$ of all distributions in~$\Omega$.

We suppose hereafter that $n\geq2$. Consider a linear PDO
$$
A:=\sum_{|\mu|\leq 2l}a_{\mu}(x)\,
\frac{\partial^{|\mu|}}{\partial^{\mu_{1}}_{x_{1}}\cdots
\partial^{\mu_{n}}_{x_{n}}}
$$
of even order $2l\geq2$ with complex-valued coefficients $a_{\mu}\in C^{\infty}(\overline{\Omega})$. As usual, $\mu:=(\mu_{1},\ldots,\mu_{n})$ is a multi-index with integer-valued nonnegative components, $|\mu|:=\mu_{1}+\cdots+\mu_{n}$, $x=(x_1,\ldots,x_n)$ is an arbitrary point in the closure $\overline{\Omega}$ of $\Omega$, and $C^{\infty}(\overline{\Omega})$ consists of restrictions of all functions from $C^{\infty}(\mathbb{R}^{n})$ to $\overline{\Omega}$. We suppose that $A$ is elliptic in $\overline{\Omega}$, i.e.
$$
\sum_{|\mu|=2l}a_{\mu}(x)\,
\xi_{1}^{\mu_{1}}\cdots\xi_{n}^{\mu_{n}}\neq0
$$
whenever $x\in\overline{\Omega}$ and  $\xi=(\xi_{1},\ldots,\xi_{n})\in\mathbb{R}^{n}\setminus\{0\}$. If $n=2$,  we suppose in addition that $A$ is properly elliptic in $\overline{\Omega}$ (see, e.g., \cite[Chapter~2, Section~1.2]{LionsMagenes72} or \cite[Section~5.2.1]{Triebel95}). The operator $A$ acts continuously on $\mathcal{D}'(\Omega)$. Specifically, if $\Omega$ is bounded, then $A$ is a bounded operator from $B^{s}_{p,q}(\Omega)$ to $B^{s-2l}_{p,q}(\Omega)$ and from $F^{s}_{p,q}(\Omega)$ to $F^{s-2l}_{p,q}(\Omega)$ under condition \eqref{indexes-case}; this follows from \cite[Theorems 2.3.8 and 2.8.2]{Triebel83}.

If quasi-Banach spaces $X$ and $Y$ are continuously embedded in $\mathcal{D}'(\Omega)$, the quasi-Banach space $X(A,Y)$ is well defined by formulas \eqref{space-X(T,Y)} and \eqref{norm-X(T,Y)} in which $T:=A$ and $\Upsilon:=\mathcal{D}'(\Omega)$.

We say that a quasi-Banach space $Y\hookrightarrow\mathcal{D}'(\Omega)$ is admissible if $E\hookrightarrow Y$ for a certain space $E$ from the class $\mathcal{U}(\Omega)$, the embeddings being continuous.

As above, $\mathfrak{F}$ is an arbitrary interpolation functor defined on the category of all interpolation pairs of quasi-Banach spaces (resp., Banach spaces). We suppose henceforth in this section that
\begin{equation}\label{Lipschitz-boundary}
\mbox{$\Omega$ is a bounded domain in $\mathbb{R}^{n}$ with Lipschitz boundary.}
\end{equation}

\begin{theorem}\label{main-interp-th}
Let $X_0$ and $X_1$ be quasi-Banach (resp., Banach) spaces from the class $\mathcal{U}(\Omega)$, and let $Y_0$ and $Y_1$ be admissible quasi-Banach (resp., Banach) spaces continuously embedded in $\mathcal{D}'(\Omega)$. Given $j\in\{0,1\}$, we put $Z_j:=B^{s-2l}_{p,q}(\Omega)$ if $X_j:=B^{s}_{p,q}(\Omega)$ and put
$Z_j:=F^{s-2l}_{p,q}(\Omega)$ if $X_j:=F^{s}_{p,q}(\Omega)$ (of course, $s$, $p$, and $q$ depends on $j$). Then
\begin{equation}\label{interp-XY-XY}
\mathfrak{F}[X_0(A,Y_0),X_1(A,Y_1)]=
(\mathfrak{F}[X_0,X_1])(A,\mathfrak{F}[Y_{0}\cap Z_{0},Y_{1}\cap Z_{1}])
\end{equation}
up to equivalence of quasi-norms.
\end{theorem}

\begin{proof}
\emph{Step~$1$.} Without loss of generality we may suppose that the continuous embeddings
\begin{equation}\label{assump-embedd}
Y_{0}\hookrightarrow Z_{0}\quad\mbox{and}\quad Y_{1}\hookrightarrow Z_{1}
\end{equation}
hold true and should prove that
\begin{equation}\label{interp-XY-XY-assump}
\mathfrak{F}[X_0(A,Y_0),X_1(A,Y_1)]=
(\mathfrak{F}[X_0,X_1])(A,\mathfrak{F}[Y_{0},Y_{1}]).
\end{equation}
Indeed, it follows from \cite[Theorems~2.3.8 and~2.8.2]{Triebel83} that $A$ is a bounded operator from $X_j$ to $Z_j$. Hence, the space $X_j(A,Y_j)$ is equal to $X_j(A,Y_j\cap Z_j)$ up to equivalence of quasi-norms for each $j\in\{0,1\}$, and the left-hand side of \eqref{interp-XY-XY} becomes
\begin{equation*}
\mathfrak{F}[X_0(A,Y_0\cap Z_0),X_1(A,Y_1\cap Z_1)],
\end{equation*}
where each space $Y_j\cap Z_j$ is admissible and continuously embedded in $Z_j$. We may therefore take $Y_j\cap Z_j$ instead of $Y_j$ in our considerations. Thus, we suppose in what follows that \eqref{assump-embedd} is valid and will prove \eqref{interp-XY-XY-assump}. Note that we only need to prove \eqref{interp-XY-XY-assump} as an equality of linear spaces because the equivalence of quasi-norms in the corresponding quasi-Banach spaces follows from this equality since they are continuously embedded in $\mathcal{D}'(\Omega)$.

\emph{Step~$2$.} Here, we consider the special case where the boundary $\partial\Omega$ of $\Omega$ is of class $C^{\infty}$. Let $E$ stand for either $B$ or $F$ used in the designations of Besov spaces or Triebel--Lizorkin spaces, and let $p$ and $q$ satisfy \eqref{indexes-case}. Given an integer $r\geq1$, we consider the PDO $A^{r}A^{r\ast}+I$ of order $4lr$; here, $A^{r\ast}$ is the formally adjoint PDO to the $r$-th iteration $A^{r}$ of~$A$, and $I$ is the identity operator. Put
\begin{equation}\label{trace-th-parameter}
\pi(p,n):=\frac{1}{p}+
\max\biggl\{0,(n-1)\biggl(\frac{1}{p}-1\biggl)\biggr\}
\end{equation}
in view of the trace theorem for these spaces \cite[Theorem~3.3.3]{Triebel83}. Choosing a real number
\begin{equation}\label{sigma-cond}
\sigma>2lr-1+\pi(p,n)
\end{equation}
arbitrarily, we let $E^{\sigma}_{p,q}(\Omega,2lr)$ denote the linear space of all distributions $u\in E^{\sigma}_{p,q}(\Omega)$ such that $\partial^{j-1}_{\nu}u=0$ on $\partial\Omega$ for each $j\in\{1,\ldots,2lr\}$. Here, $\partial_{\nu}$ is the operator of differentiation along the inward normal to $\partial\Omega$. The set $E^{\sigma}_{p,q}(\Omega,2lr)$ is well defined and is a (closed) subspace of the Banach space $E^{\sigma}_{p,q}(\Omega)$ according to the above-mentioned trace theorem.

The PDO $A^{r}A^{r\ast}+I$ sets an isomorphism
\begin{equation}\label{isom-special}
A^{r}A^{r\ast}+I:E^{\sigma}_{p,q}(\Omega,2lr)\leftrightarrow
E^{\sigma-4lr}_{p,q}(\Omega)
\end{equation}
whenever $\sigma$ satisfies \eqref{sigma-cond}. Indeed, since the Dirichlet boundary-value problem
\begin{gather*}
A^{r}A^{r\ast}u+u=f\;\;\mbox{in}\;\;\Omega,\\
\partial^{j-1}_{\nu}u=g_{j}\;\;\mbox{on}\;\;\partial\Omega,\quad j=1,\ldots,2lr,
\end{gather*}
is elliptic in $\overline{\Omega}$, the PDO $A^{r}A^{r\ast}+I$ is a Fredholm bounded operator between the spaces indicated in \eqref{isom-special}. Moreover, the kernel and index of this operator do not depend on $\sigma$, $p$, and $q$. This is proved in \cite[Theorem~14]{FrankeRunst95} (and also follows from \cite[Theorem~5.2]{Johnsen96}, the latter article concerning pseudodifferential elliptic problems). As is known \cite[Chapter~2, Section~8.5]{LionsMagenes72}, the index of any elliptic Dirichlet problem is zero. Moreover, the kernel of the above problem is trivial \cite[Lemma~3.1]{MikhailetsMurach14}. This implies the required isomorphism \eqref{isom-special}. The inverse of \eqref{isom-special} sets an isomorphism
\begin{equation}\label{isom-invers}
(A^{r}A^{r\ast}+I)^{-1}:E^{\lambda}_{p,q}(\Omega)\leftrightarrow E^{\lambda+4lr}_{p,q}(\Omega,2lr)
\end{equation}
for every real $\lambda>-2lr-1+\pi(p,n)$.

Given $j\in\{0,1\}$, we let $X_j:=E^{s}_{p,q}(\Omega)$ for some parameters $s$, $p$, and $q$ subject to \eqref{indexes-case}, with the parameters and $E\in\{B,F\}$ depending on $j$. Since the space $Y_j$ is admissible and since the domain $\Omega$ is bounded, it follows from \cite[Theorem~3.3.1]{Triebel83} that the continuous embedding $E^{s-2l+4lr}_{p,q}(\Omega)\hookrightarrow Y_j$ holds true if the integer $r\geq1$ is large enough. We may and do choose $r$ to be independent of $j$ and suppose in addition that $s-2l>-2lr-1+\pi(p,n)$. Owing to \eqref{isom-invers}, the mapping $S:w\mapsto-(A^{r}A^{r\ast}+I)^{-1}w$ sets the bounded operator
\begin{equation*}
S:Z_j=E^{s-2l}_{p,q}(\Omega)\to E^{s-2l+4lr}_{p,q}(\Omega)
\hookrightarrow Y_j.
\end{equation*}
Hence, the mapping $R:w\mapsto A^{r-1}A^{r\ast}Sw$ sets the bounded operator
\begin{equation*}
R:Z_j\to E^{s-2l+4lr-2l(r-1)-2lr}_{p,q}(\Omega)=E^{s}_{p,q}(\Omega)=X_j.
\end{equation*}
Thus, the spaces $X_{0}$, $Y_{0}$, $Z_{0}$, $X_{1}$, $Y_{1}$, and $Z_{1}$ and  mappings $T:=A$, $R$, and $S$ satisfy hypotheses (i)--(vi) of Theorem~\ref{th-basic}. The last hypothesis (vii) is also satisfied because
\begin{equation*}
TRw=ARw=(A^{r}A^{r\ast}+I-I)(A^{r}A^{r\ast}+I)^{-1}w=w+Sw
\end{equation*}
for every $w\in Z$. Hence, the required equality \eqref{interp-XY-XY-assump} holds true (with equivalence of norms) by this theorem.

\emph{Step~$3$.} Using the result obtained on Step~2, we here treat the general case \eqref{Lipschitz-boundary}. By assumption, all coefficients of the PDO $A$ have extensions of class $C^{\infty}(\mathbb{R}^{n})$. Replacing the coefficients with these extensions, we define a certain PDO $A^{\diamond}$ in $\mathbb{R}^{n}$. Since $A$ is elliptic in $\overline{\Omega}$, the PDO $A^{\diamond}$ is elliptic on the set $\Theta_{\varepsilon}:=\{x\in\mathbb{R}^{n}:
\mathrm{dist}(x,\overline{\Omega})\leq\varepsilon\}$ for some $\varepsilon>0$. Then $A^{\diamond}$ is properly elliptic on $\Theta_{\varepsilon}$ as to the $n=2$ case (see \cite[Chapter~2, Remark~1.2]{LionsMagenes72}). Since $\Omega$ has Lipschitz boundary, there exists an (open) domain $\Omega^{\times}$ such that its boundary $\partial\Omega^{\times}$ is of class $C^{\infty}$ and that $\Omega\subset\Omega^{\times}\subset\Theta_{\varepsilon}$ (see, e.g, \cite[Assertion~9.1.5]{Agranovich13}). Let $A^{\times}$ stand for the restriction of the PDO $A^{\diamond}$ to the closure of $\Omega^{\times}$.

Given $j\in\{0,1\}$, we put $X^{\times}_j:=B^{s}_{p,q}(\Omega^{\times})$ and $Z^{\times}_j:=B^{s-2l}_{p,q}(\Omega^{\times})$ if $X_j:=B^{s}_{p,q}(\Omega)$ and put $X^{\times}_j:=F^{s}_{p,q}(\Omega^{\times})$ and
$Z^{\times}_j:=F^{s-2l}_{p,q}(\Omega^{\times})$ if $X_j:=F^{s}_{p,q}(\Omega)$. Thus, the PDO $A^{\times}$ acts continuously from $X^{\times}_j$ to $Z^{\times}_j$. We also introduce the linear space
\begin{equation*}
Y^{\times}_j:=\{v\in Z^{\times}_j:v\!\upharpoonright\!\Omega\in Y_{j}\}
\end{equation*}
endowed with the quasi-norm
\begin{equation*}
\|v\|_{Y^{\times}_j}:=\|v\|_{Z^{\times}_j}+
\|v\!\upharpoonright\!\Omega\|_{Y_{j}}.
\end{equation*}
Note that $Y^{\times}_j=Z^{\times}_{j}(K,Y_{j})$, where  $K$ denotes the operator which maps every distribution $v\in\mathcal{D}'(\Omega^{\times})$ into its restriction $v\!\upharpoonright\!\Omega$. Hence, the quasi-normed space $Y^{\times}_j$ is complete. It is admissible (as a space continuously embedded in $\mathcal{D}'(\Omega^{\times})$). Indeed, since $Y_j$ is admissible, the continuous embeddings $B^{\sigma}_{p,q}(\Omega)\hookrightarrow Y_j$ and $B^{\sigma}_{p,q}(\Omega^{\times})\hookrightarrow Z^{\times}_j$ hold true provided that $\sigma\gg1$; therefore, $B^{\sigma}_{p,q}(\Omega^{\times})\hookrightarrow Y^{\times}_j$.

Owing to Step~2,
\begin{equation}\label{interp-star-spaces}
\mathfrak{F}[X^{\times}_0(A^{\times},Y^{\times}_0),
X^{\times}_1(A^{\times},Y^{\times}_1)]=
(\mathfrak{F}[X^{\times}_0,X^{\times}_1])
(A^{\times},\mathfrak{F}[Y^{\times}_{0},Y^{\times}_{1}])
\end{equation}
up to equivalence of quasi-norms. The result obtained on Step~2 is applicable here because $\partial\Omega^{\times}$ is of class $C^{\infty}$. Let us show that the images of the left and right sides of \eqref{interp-star-spaces} under the mapping $K$ coincide with the corresponding sides of \eqref{interp-XY-XY-assump}.

This mapping sets bounded operators
\begin{equation}\label{restriction-oper-X-Y}
K:X^{\times}_j\to X_j\quad\mbox{and}\quad K:Y^{\times}_j\to Y_j
\end{equation}
and
\begin{equation}\label{restriction-oper}
K:X^{\times}_j(A^{\times},Y^{\times}_j)\to X_j(A,Y_j),
\end{equation}
with $j\in\{0,1\}$. They have a common right-inverse linear operator. Indeed, since $\Omega$ has Lipschitz boundary, there exists a linear mapping
$J:\mathcal{D}'(\Omega)\to\mathcal{S}'(\mathbb{R}^{n})$ that extends every distribution $u\in\mathcal{D}'(\Omega)$ to $\mathbb{R}^{n}$ and that sets a bounded operator from $E^{s}_{p,q}(\Omega)$ to $E^{s}_{p,q}(\mathbb{R}^{n})$ if $E\in\{B,F\}$ and whenever $s$, $p$, and $q$ satisfy \eqref{indexes-case}. This extension operator was built in \cite[Theorem~2.2]{Rychkov99}. Hence, the mapping $J^{\times}:u\mapsto(Ju)\!\upharpoonright\!\Omega^{\times}$, where $u\in\mathcal{D}'(\Omega)$, satisfies the condition $KJ^{\times}u=u$ for every $u\in\mathcal{D}'(\Omega)$ and sets bounded operators
\begin{equation}\label{extension-oper-X-Y}
J^{\times}:X_j\to X^{\times}_j\quad\mbox{and}\quad
J^{\times}:Y_j\to Y^{\times}_j
\end{equation}
and
\begin{equation}\label{extension-oper}
J^{\times}:X_j(A,Y_j)\to X^{\times}_j(A^{\times},Y^{\times}_j)
\end{equation}
for each $j\in\{0,1\}$.

Applying interpolation functor $\mathfrak{F}$ to the operators \eqref{restriction-oper} and \eqref{extension-oper} where $j\in\{0,1\}$, we conclude that the mappings $K$ and $J^{\times}$ set bounded operators
\begin{equation*}
K:\mathfrak{F}[X^{\times}_0(A^{\times},Y^{\times}_0),
X^{\times}_1(A^{\times},Y^{\times}_1)]\to
\mathfrak{F}[X_0(A,Y_0),X_1(A,Y_1)]
\end{equation*}
and
\begin{equation*}
J^{\times}:\mathfrak{F}[X_0(A,Y_0),X_1(A,Y_1)]\to
\mathfrak{F}[X^{\times}_0(A^{\times},Y^{\times}_0),
X^{\times}_1(A^{\times},Y^{\times}_1)].
\end{equation*}
Hence,
\begin{equation}\label{restriction-left-side}
K\bigl(\mathfrak{F}[X^{\times}_0(A^{\times},Y^{\times}_0),
X^{\times}_1(A^{\times},Y^{\times}_1)]\bigr)=
\mathfrak{F}[X_0(A,Y_0),X_1(A,Y_1)].
\end{equation}

Applying the interpolation functor $\mathfrak{F}$ to each of the operators \eqref{restriction-oper-X-Y} and \eqref{extension-oper-X-Y} where $j\in\{0,1\}$, we see that the restriction mapping gives bounded operators
\begin{equation*}
K:\mathfrak{F}[X^{\times}_0,X^{\times}_1]\to\mathfrak{F}[X_0,X_1]
\quad\mbox{and}\quad
K:\mathfrak{F}[Y^{\times}_0,Y^{\times}_1]\to\mathfrak{F}[Y_0,Y_1],
\end{equation*}
whereas the extension mapping sets bounded operators
\begin{equation*}
J^{\times}:\mathfrak{F}[X_0,X_1]\to\mathfrak{F}[X^{\times}_0,X^{\times}_1]
\quad\mbox{and}\quad
J^{\times}:\mathfrak{F}[Y_0,Y_1]\to\mathfrak{F}[Y^{\times}_0,Y^{\times}_1].
\end{equation*}
Thus,
\begin{equation*}
K\bigl(\mathfrak{F}[X^{\times}_0,X^{\times}_1]\bigr)=\mathfrak{F}[X_0,X_1]
\quad\mbox{and}\quad
K\bigl(\mathfrak{F}[Y^{\times}_0,Y^{\times}_1]\bigr)=\mathfrak{F}[Y_0,Y_1].
\end{equation*}
Moreover, according to Theorem~\ref{th-basic}, we conclude that
\begin{equation}\label{interp-Y-star}
\mathfrak{F}[Y^{\times}_0,Y^{\times}_1]=
\mathfrak{F}[Z^{\times}_{0}(K,Y_{0}),Z^{\times}_{1}(K,Y_{1})]=
(\mathfrak{F}[Z^{\times}_{0},Z^{\times}_{1}])(K,\mathfrak{F}[Y_0,Y_1])
\end{equation}
putting $X_j=Z^{\times}_{j}$, $T:=K$, $R:=J^{\times}$, and $S:=0$ in this theorem.

Therefore, if $u=Kv$ for a certain distribution $v\in\mathfrak{F}[X^{\times}_0,X^{\times}_1]$ such that $A^{\times}v\in\mathfrak{F}[Y^{\times}_{0},Y^{\times}_{1}]$, then $u\in\mathfrak{F}[X_0,X_1]$ and $Au=KA^{\times}v\in\mathfrak{F}[Y_0,Y_1]$. Conversely, if $u\in\mathfrak{F}[X_0,X_1]$ and $Au\in\mathfrak{F}[Y_0,Y_1]$, then
$u=Kv$ for certain $v\in\mathfrak{F}[X^{\times}_0,X^{\times}_1]$,
which implies that $A^{\times}v\in\mathfrak{F}[Z^{\times}_{0},Z^{\times}_{1}]$ and $KA^{\times}v=Au\in\mathfrak{F}[Y_0,Y_1]$, i.e. $A^{\times}v\in\mathfrak{F}[Y^{\times}_0,Y^{\times}_1]$ by \eqref{interp-Y-star}. Thus,
\begin{equation}\label{restriction-right-side}
K\bigl((\mathfrak{F}[X^{\times}_0,X^{\times}_1])
(A^{\times},\mathfrak{F}[Y^{\times}_{0},Y^{\times}_{1}])\bigr)=
(\mathfrak{F}[X_0,X_1])(A,\mathfrak{F}[Y_{0},Y_{1}]).
\end{equation}

Now the required equality \eqref{interp-XY-XY-assump} follows from \eqref{interp-star-spaces}, \eqref{restriction-left-side}, and \eqref{restriction-right-side}.
\end{proof}

\begin{remark}\label{rem-X0-interp}
Theorem~\ref{main-interp-th} also describes the interpolation between $X_0$ and $X_1(A,Y_1)$ because $X_0=X_0(A,Z_0)$ up to equivalence of norms.
Thus,
\begin{equation*}
\mathfrak{F}[X_0,X_1(A,Y_1)]=
(\mathfrak{F}[X_0,X_1])(A,\mathfrak{F}[Z_{0},Y_{1}\cap Z_{1}]).
\end{equation*}
Analogously
\begin{equation*}
\mathfrak{F}[X_0(A,Y_0),X_1]=
(\mathfrak{F}[X_0,X_1])(A,\mathfrak{F}[Y_{0}\cap Z_{0},Z_{1}]).
\end{equation*}
\end{remark}

Let us apply Theorem~\ref{main-interp-th} to some known interpolation functors, we restricting ourselves to the important case where both $Y_0$ and $Y_1$ belong to $\mathcal{U}(\Omega)$. We first consider the real interpolation between quasi-Banach spaces and the complex interpolation between Banach spaces, the most applicable interpolation functors in the theory of function spaces (see, e.g., \cite{Amann19, Triebel95}). Then we examine the more exotic $\pm$-interpolation functor introduced by Gustavsson and Peetre \cite{GustavssonPeetre77} because it is applied to quasi-Banach spaces and retains the class of Triebel--Lizorkin spaces \cite[p.~134]{FrazierJawerth90} (in contrast to the real interpolation).

Using designations like $X(\Omega)$ and $Y(\Omega)$ for quasi-Banach spaces continuously embedded in $\mathcal{D}'(\Omega)$, we prefer to exploit more concise notation $X(A,Y,\Omega)$ for the quasi-Banach space $(X(\Omega))(A,Y(\Omega))$. Thus,
\begin{equation*}
E^{s}_{p,q}(A,G^{\alpha}_{\beta,\gamma},\Omega):=
\bigl\{u\in E^{s}_{p,q}(\Omega):
Au\in G^{\alpha}_{\beta,\gamma}(\Omega)\bigr\}
\end{equation*}
with the corresponding graph quasi-norm; here, $E$ and $G$ stand for either $B$ or $F$ in designations of Besov or Triebel--Lizorkin spaces, whereas the number parameters take on admissible values for these spaces.

We begin with the case where $\mathfrak{F}[X_0,X_1]=(X_0,X_1)_{\theta,q}$ is the functor of the real interpolation with the parameters $\theta$ and $q$ subject to $0<\theta<1$ and $0<q\leq\infty$.

\begin{theorem}\label{R-interp}
Let $E,G,V,W\in\{B,F\}$, and let $s_j,\alpha_j\in\mathbb{R}$ and  $p,\beta, q_j,\gamma_j\in(0,\infty]$ for each $j\in\{0,1\}$. Assume that $s_0\neq s_1$ and $\alpha_0\neq\alpha_1$ and that $p<\infty$ if $E=F$ or $V=F$ and that $\beta<\infty$ if $G=F$ or $W=F$. Suppose also that
\begin{equation}\label{inclusions-R-interp}
G^{\alpha_0}_{\beta,\gamma_0}(\Omega)\subset E^{s_0-2l}_{p,q_0}(\Omega)
\quad\mbox{and}\quad
W^{\alpha_1}_{\beta,\gamma_1}(\Omega)\subset V^{s_1-2l}_{p,q_1}(\Omega).
\end{equation}
Put
\begin{equation}\label{s-alpha}
s:=(1-\theta)s_0+\theta s_1\quad\mbox{and}\quad
\alpha:=(1-\theta)\alpha_0+\theta\alpha_1.
\end{equation}
Then
\begin{equation*}\label{R-interp-formula}
(E^{s_0}_{p,q_0}(A,G^{\alpha_0}_{\beta,\gamma_0},\Omega),
V^{s_1}_{p,q_1}(A,W^{\alpha_1}_{\beta,\gamma_1},\Omega))_{\theta,q}
=B^{s}_{p,q}(A,B^{\alpha}_{\beta,q},\Omega)
\end{equation*}
up to equivalence of quasi-norms.
\end{theorem}

\begin{proof}
This result follows directly from Theorem~\ref{main-interp-th} and the known theorems about the real interpolation between Besov spaces and between Triebel--Lizorkin spaces with the same integral-exponent (see \cite[Corollary~1.111(i)]{Triebel06}). This interpolation gives only Besov spaces. Note that the real interpolation between Triebel--Lizorkin spaces or between Besov space and Triebel--Lizorkin space with the same integral-exponent are reduced to the interpolation between only Besov spaces because
\begin{equation*}
B^{s}_{p,\min\{p,q\}}(\Omega)\hookrightarrow F^{s}_{p,q}(\Omega)\hookrightarrow B^{s}_{p,\max\{p,q\}}(\Omega)
\end{equation*}
due to \cite[Proposition~2.3.2/2(iii)]{Triebel83}.
\end{proof}

Note that the inclusions \eqref{inclusions-R-interp} imply the continuous embeddings of the quasi-Banach spaces involved because these spaces are continuously embedded in $\mathcal{D}'(\Omega)$. There are sufficient conditions for \eqref{inclusions-R-interp}. For example, the first inclusion follows from
\begin{equation}\label{cond-embedd}
\alpha_0-s_0+2l>\max\biggl\{0,\frac{n}{\beta}-\frac{n}{p}\biggr\}
\end{equation}
\cite[Section~2.5.1]{EdmundsTriebel96}. If $E=G$ and $\gamma_0\leq q_0$, we may replace $>$ with $\geq$ in \eqref{cond-embedd} (see also \cite[Theorem~2.85 (ii)]{Triebel20}).

The following result deals with the case where $\alpha_0=\alpha_1=:\lambda$ or $s_0=s_1=:t$.

\begin{theorem}\label{R-interp-special}
Suppose that the hypotheses formulated in the first two sentences of Theorem~$\ref{R-interp}$ are satisfied. Suppose also that $\lambda,t\in\mathbb{R}$, $\gamma,r\in(0,\infty]$,
\begin{equation*}
G^{\lambda}_{\beta,\gamma}(\Omega)\subset
E^{s_0-2l}_{p,q_0}(\Omega)\cap V^{s_1-2l}_{p,q_1}(\Omega),
\end{equation*}
and
\begin{equation*}
G^{\alpha_0}_{\beta,\gamma_0}(\Omega)\cup W^{\alpha_1}_{\beta,\gamma_1}(\Omega)\subset E^{t-2l}_{p,r}(\Omega).
\end{equation*}
Define numbers $s$ and $\alpha$ by formulas \eqref{s-alpha}. Then
\begin{equation}\label{R-interp-special-formula}
(E^{s_0}_{p,q_0}(A,G^{\lambda}_{\beta,\gamma},\Omega),
V^{s_1}_{p,q_1}(A,G^{\lambda}_{\beta,\gamma},\Omega))_{\theta,q}
=B^{s}_{p,q}(A,G^{\lambda}_{\beta,\gamma},\Omega)
\end{equation}
and
\begin{equation*}
(E^{t}_{p,r}(A,G^{\alpha_0}_{\beta,\gamma_0},\Omega),
E^{t}_{p,r}(A,W^{\alpha_1}_{\beta,\gamma_1},\Omega))_{\theta,q}
=E^{t}_{p,r}(A,B^{\alpha}_{\beta,q},\Omega)
\end{equation*}
up to equivalence of quasi-norms.
\end{theorem}

\begin{proof}
It is the same as the proof of Theorem~$\ref{R-interp}$.
\end{proof}

Consider the case where  $\mathfrak{F}[X_0,X_1]=[X_0,X_1]_{\theta}$ is the functor of the complex interpolation (between Banach spaces) with the parameter $\theta$ subject to $0<\theta<1$.

\begin{theorem}\label{C-interp}
Let $E,G\in\{B,F\}$, and let $s_j,\alpha_j\in\mathbb{R}$ and $p_j,q_j,\beta_j,\gamma_j\in[1,\infty)$ for each $j\in\{0,1\}$. Suppose that
\begin{equation}\label{inclusion-C-interp}
G^{\alpha_j}_{\beta_j,\gamma_j}(\Omega)\subset E^{s_j-2l}_{p_j,q_j}(\Omega)
\quad\mbox{whenever}\quad j\in\{0,1\}.
\end{equation}
Define the parameters $s$, $\alpha$, $p$, $q$, $\beta$, and $\gamma$ by  formulas \eqref{s-alpha} and
\begin{gather}\label{p-q}
\frac{1}{p}=\frac{1-\theta}{p_0}+\frac{\theta}{p_1},\quad
\frac{1}{q}=\frac{1-\theta}{q_0}+\frac{\theta}{q_1},\\
\frac{1}{\beta}=\frac{1-\theta}{\beta_0}+\frac{\theta}{\beta_1},\quad
\frac{1}{\gamma}=\frac{1-\theta}{\gamma_0}+\frac{\theta}{\gamma_1}.
\label{beta-gamma}
\end{gather}
Then
\begin{equation}\label{C-interp-formula}
[E^{s_0}_{p_0,q_0}(A,G^{\alpha_0}_{\beta_0,\gamma_0},\Omega),
E^{s_1}_{p_1,q_1}(A,G^{\alpha_1}_{\beta_1,\gamma_1},\Omega)]_{\theta}=
E^{s}_{p,q}(A,G^{\alpha}_{\beta,\gamma},\Omega)
\end{equation}
up to equivalence of norms.
\end{theorem}

\begin{proof}
This result follows immediately from Theorem~\ref{main-interp-th} and the known theorems about the complex interpolation between Besov spaces and between Triebel--Lizorkin spaces (see \cite[Proposition~1.114]{Triebel06} and also \cite[Theorem~1.22 (ii)]{Triebel10}).
\end{proof}

The method of complex interpolation admits an extension to analytically convex quasi-Banach spaces, to which Besov and Triebel--Lizorkin spaces belong \cite[Section~7]{KaltonMayborodaMitrea07}. However, the above-mentioned $\pm$-method of interpolation is more suitable for our purposes because it is applicable to an arbitrary interpolation pair of quasi-Banach spaces $X_{0}$ and $X_{1}$. Following \cite[Section~6]{GustavssonPeetre77}, we let $\langle X_{0},X_{1},\theta\rangle$ denote the quasi-Banach space obtained by the $\pm$-method of interpolation applied to this pair. (The designation $\langle X_{0},X_{1}\rangle_{\theta}$, which suggests itself, stands for the other interpolation method introduced by Peetre \cite[Section~1]{Peetre71} and based on Gagliardo's work \cite{Gagliardo68}. We prefer to use the $\pm$-method because it retains the class of spaces $E^{s}_{p,\infty}$, where $E\in\{B,F\}$, unlike the functor $\langle X_{0},X_{1}\rangle_{\theta}$ \cite[Theorem~2.28]{YuanSickelYang15}.)

\begin{theorem}\label{plus-minus-interp}
Let $E,G\in\{B,F\}$, and let $s_j,\alpha_j\in\mathbb{R}$ and $p_j,q_j,\beta_j,\gamma_j\in(0,\infty]$ for each $j\in\{0,1\}$. Suppose that $p_0<\infty$ and $p_1<\infty$ if $E=F$, and suppose that $\beta_0<\infty$ and $\beta_1<\infty$ if $G=F$. Assume also that \eqref{inclusion-C-interp} holds true. Define the parameters $s$, $\alpha$, $p$, $q$, $\beta$, and $\gamma$ by  formulas \eqref{s-alpha}, \eqref{p-q}, and \eqref{beta-gamma}. Then
\begin{equation*}\label{plus-minus-interp-formula}
\langle E^{s_0}_{p_0,q_0}(A,G^{\alpha_0}_{\beta_0,\gamma_0},\Omega),
E^{s_1}_{p_1,q_1}(A,G^{\alpha_1}_{\beta_1,\gamma_1},\Omega),
\theta\rangle=
E^{s}_{p,q}(A,G^{\alpha}_{\beta,\gamma},\Omega)
\end{equation*}
up to equivalence of quasi-norms.
\end{theorem}

\begin{proof}
This result follows from Theorem~\ref{main-interp-th} and \cite[Theorem~2.12 (i)]{YuanSickelYang15}. The latter theorem contains relevant interpolation formulas for the $\pm$-method applied to Besov spaces and Triebel--Lizorkin spaces over $\mathbb{R}^{n}$ (see also \cite[p.~134]{FrazierJawerth90} as to $F$-spaces). The counterparts of these  formulas for spaces over $\Omega$ are proved with the help of Rychkov's extension operator \cite[Theorem~2.2]{Rychkov99} as is shown in \cite[Theorem~1.110]{Triebel06}
for the real and complex methods of interpolation. Let us give the proof for the readers' convenience. Let $J:\mathcal{D}'(\Omega)\to\mathcal{S}'(\mathbb{R}^{n})$ be the above-mentioned extension operator. It sets bounded operators from  $E^{s_j}_{p_j,q_j}(\Omega)$ to $E^{s_j}_{p_j,q_j}(\mathbb{R}^{n})$ for each $j\in\{0,1\}$. Interpolating them, we obtain the bounded operator
\begin{equation}\label{exten-interp}
J:\bigl\langle E^{s_0}_{p_0,q_0}(\Omega),
E^{s_1}_{p_1,q_1}(\Omega),\theta\bigr\rangle\to
\bigl\langle E^{s_0}_{p_0,q_0}(\mathbb{R}^{n}),
E^{s_1}_{p_1,q_1}(\mathbb{R}^{n}),\theta\bigr\rangle=
E^{s}_{p,q}(\mathbb{R}^{n}).
\end{equation}
We also need the bounded operator
\begin{equation}\label{exten-E}
J:E^{s}_{p,q}(\Omega)\to E^{s}_{p,q}(\mathbb{R}^{n}).
\end{equation}
The mapping $J$ is right-inverse to the restriction operator $K:w\mapsto w\!\upharpoonright\!\Omega$, where $w\in\mathcal{S}'(\mathbb{R}^{n})$. The mapping $K$ sets bounded operators from  $E^{s_j}_{p_j,q_j}(\mathbb{R}^{n})$ to $E^{s_j}_{p_j,q_j}(\Omega)$ for each $j\in\{0,1\}$. Interpolating them, we obtain the bounded operator
\begin{equation}\label{restr-interp}
K:E^{s}_{p,q}(\mathbb{R}^{n})\to \bigl\langle E^{s_0}_{p_0,q_0}(\Omega),
E^{s_1}_{p_1,q_1}(\Omega),\theta\bigr\rangle.
\end{equation}
We also need the bounded operator
\begin{equation}\label{restr-E}
K:E^{s}_{p,q}(\mathbb{R}^{n})\to E^{s}_{p,q}(\Omega).
\end{equation}
Now \eqref{exten-interp} and \eqref{restr-E} imply that the identity operator $KJ$ set the continuous embedding
\begin{equation}
\bigl\langle E^{s_0}_{p_0,q_0}(\Omega),
E^{s_1}_{p_1,q_1}(\Omega),\theta\bigr\rangle\hookrightarrow
E^{s}_{p,q}(\Omega).
\end{equation}
The inverse holds true by \eqref{exten-E} and \eqref{restr-interp}.
\end{proof}

The following theorem gives sufficient conditions for the distribution space $X(A,Y)$ to be separable and for $C^{\infty}$-functions to be dense in this space.

\begin{theorem}\label{th-distr-sep-dense}
Let $X$ be a quasi-Banach space from the class $\mathcal{U}(\Omega)$, and let $Y$ be an admissible quasi-Banach space continuously embedded in $\mathcal{D}'(\Omega)$. We put $Z:=B^{s-2l}_{p,q}(\Omega)$ if $X:=B^{s}_{p,q}(\Omega)$ and put $Z:=F^{s-2l}_{p,q}(\Omega)$ if $X:=F^{s}_{p,q}(\Omega)$. Then:
\begin{itemize}
\item[(i)] If $X$ and $Y\cap Z$ are separable, then $X(A,Y)$ is also separable.
\item[(ii)] If the set $C^{\infty}(\overline{\Omega})$ is dense in the quasi-normed spaces $X$ and $Y\cap Z$, then $C^{\infty}(\overline{\Omega})$ is dense in the quasi-normed space $X(A,Y)$.
\end{itemize}
\end{theorem}

As to assertion (ii), observe that $C^{\infty}(\overline{\Omega})$ lies in $X$ and $Y\cap Z$ due to the common (i.e. for both assertions) hypotheses of this theorem. Hence, $C^{\infty}(\overline{\Omega})\subset X(A,Y)$.

\begin{proof}[Proof of Theorem $\ref{th-distr-sep-dense}$.]
\emph{Step~$1$.} Since $X(A,Y)=X(A,Y\cap Z)$ up to equivalence of quasi-norms, we suppose without loss of generality that the space $Y$ is continuously embedded in $Z$. We refer to the proof of Theorem~\ref{main-interp-th} in the case where $X_{0}=X_{1}=X$, $Y_{0}=Y_{1}=Y$, and $Z_{0}=Z_{1}=Z$.

\emph{Step~$2$.} We first consider the special case when $\partial\Omega$ is of class $C^{\infty}$. Let $R$ and $S$ be the operators introduced on Step~2 of the proof of Theorem~\ref{main-interp-th}. They satisfy the common hypotheses of Theorem~\ref{th-sep-dense} in which $T:=A$. Thus, assertion~(i) follows immediately from assertion~(i) of Theorem~\ref{th-sep-dense}.

Let us prove assertion~(ii) in the case under consideration. Suppose that the set $M:=C^{\infty}(\overline{\Omega})$ is dense in $X$ and $Y$, and use assertion~(ii) of Theorem~\ref{th-sep-dense} for $T$ standing for $A$. As to the hypotheses of this assertion, we see that $A(M)$ is dense in $Y$ if $A(M)=M$ and that
\begin{equation*}
R(A(M))=(A^{r-1}A^{r\ast}SA)(M)\subset M
\end{equation*}
because
\begin{equation*}
S(M)\subset\bigcap_{\lambda\in\mathbb{R}}B^{\lambda+4lr}_{2,2}(\Omega)=M
\end{equation*}
in view of \eqref{isom-invers}. Therefore, the required density of $M$ in the quasi-normed space $X(A,Y)$ follows directly from assertion~(ii) of Theorem~\ref{th-sep-dense} if $A(M)=M$.

However, the strict inclusion $A(M)\subsetneqq M$ is possible, which follows from the fact that certain homogeneous elliptic equations in $\mathbb{R}^{n}$, where $n\geq3$, with $C^{\infty}$-coefficients have compactly supported nontrivial solutions (see, e.g., \cite[Theorem~13.6.15]{Hermander05v2}). Assuming $A(M)\neq M$ to be valid, let us now prove that the set $A(M)$ is dense in $Y\cap A(X)$ in the quasi-norm $\|\cdot\|_{Y}$, which the second hypothesis of assertion~(ii) of Theorem~\ref{th-sep-dense} claims. We may and do consider $\Omega$ as an open subset of a closed (i.e. compact and boundaryless) Riemannian manifold $\Xi$ of class $C^{\infty}$ and dimension~$n$. We also may and do suppose that the PDO $A$ is a restriction of a certain elliptic PDO $\mathcal{A}$ on $\Xi$ with infinitely smooth coefficients and of order $2l$. Similarly to step~3 of the proof of Theorem~\ref{main-interp-th}, we put $\mathcal{X}:=B^{s}_{p,q}(\Xi)$ and $\mathcal{Z}:=B^{s-2l}_{p,q}(\Xi)$ if $X=B^{s}_{p,q}(\Omega)$ and put $\mathcal{X}:=F^{s}_{p,q}(\Xi)$ and $\mathcal{Z}:=F^{s-2l}_{p,q}(\Xi)$ if $X=F^{s}_{p,q}(\Omega)$ and introduce the quasi-Banach space $\mathcal{Y}:=\mathcal{Z}(\mathcal{K},Y)$, where $\mathcal{K}$ denotes the operator which maps every distribution $v\in\mathcal{D}'(\Xi)$ (given on $\Xi$) into its restriction $v\!\upharpoonright\!\Omega$. Recall that Besov and Triebel--Lizorkin spaces over the manifold $\Xi$ are defined with the help of local charts on $\Xi$ (see, e.g., \cite[Section~3.2.2]{Triebel83}). The mapping $\mathcal{K}$ is a bounded operator from $\mathcal{Z}$ to $Z$ and has a right-inverse bounded linear operator $\mathcal{J}:Z\to\mathcal{Z}$ such that $\mathcal{J}(M)\subset\mathcal{M}:=C^{\infty}(\Xi)$. The operator $\mathcal{J}$ is trivially built on the base of the extension operator $J$ used on step~3 of the above-mentioned proof.

Since $M$ is dense in $X$, the set $\mathcal{M}$ is dense in $\mathcal{X}$, which is trivially shown with the help of local charts on $\mathcal{X}$. This implies the density of $\mathcal{M}$ in $\mathcal{Z}$ because some elliptic operators on $\Xi$, e.g, $(1-\Delta_{\Xi})^{l}$, set isomorphisms between $\mathcal{X}$ and $\mathcal{Z}$, which follows from \cite[Theorem~6.2.2]{Triebel92}. (As usual, $\Delta_{\Xi}$ denotes the Laplace--Beltrami operator on $\Xi$.) Hence, $\mathcal{M}$ is dense in the quasi-Banach space $\mathcal{Y}=\mathcal{Z}(\mathcal{K},Y)$ according to assertion~(ii) of Theorem~\ref{th-sep-dense} in which we put $X:=\mathcal{Z}$, $T:=\mathcal{K}$, $R:=\mathcal{J}$, and $S:=0$.

To achieve our purpose, let us prove that the set $\mathcal{A}(\mathcal{M})$ is dense in $\mathcal{Y}\cap\mathcal{A}(\mathcal{X})$ in the quasi-norm $\|\cdot\|_{\mathcal{Y}}$. We choose a distribution  $f\in\mathcal{Y}\cap\mathcal{A}(\mathcal{X})$ arbitrarily and note that  $f=\mathcal{A}u$ for certain $u\in\mathcal{X}$. Since $\mathcal{M}$ is dense in $\mathcal{X}$, there exists a sequence of functions $u_k\in\mathcal{M}$ such that $u_k\to u$ in $\mathcal{X}$, as $k\to\infty$. Then $\mathcal{A}u_k\to f$ in $\mathcal{Z}$. Besides, since $\mathcal{M}$ is dense in $\mathcal{Y}$, there exists a sequence of functions $f_k\in\mathcal{M}$ such that $f_k\to f$ in $\mathcal{Y}$. As is known \cite[Theorem~2.3.8]{Agranovich94}, it follows from the ellipticity of $\mathcal{A}$ on the closed manifold $\Xi$ that the Hilbert space $L_2(\Xi)$ decomposes into the orthogonal sum
\begin{equation}\label{orth-sum}
L_2(\Xi)=\mathcal{A}(H^{2l}_{2}(\Xi))\oplus\ker\mathcal{A}^{\ast}.
\end{equation}
Here, $H^{2l}_{2}(\Xi):=F^{2l}_{2,2}(\Xi)$ is the inner product Sobolev  space of order $2l$ over $\Xi$, and $\ker\mathcal{A}^{\ast}$ is the finite-dimensional space $\{w\in\mathcal{M}:\mathcal{A}^{\ast}w=0\;\mbox{on $\Xi$}\}$, whereas   $\mathcal{A}^{\ast}$ is the formally adjoint PDO to $\mathcal{A}$ (with respect to the Riemannian metric on $\Xi$). Hence, $f_k=\mathcal{A}v_k+w_k$ for some $v_k\in H^{2l}_{2}(\Xi)$ and $w_k\in\ker\mathcal{A}^{\ast}$. Since $\mathcal{A}v_k\in\mathcal{M}$, the ellipticity of $\mathcal{A}$ implies that $v_k\in\mathcal{M}$ \cite[Theorem~2.2.6]{Agranovich94}. Note that $\mathcal{Y}\hookrightarrow\mathcal{Z}\hookrightarrow H^{\lambda}_{2}(\Xi)$ for some $\lambda<0$ and that the Sobolev space $H^{\lambda}_{2}(\Xi)$ splits into the following direct sum of its (closed) subspaces:
\begin{equation}\label{direct-sum}
H^{\lambda}_{2}(\Xi)=\mathcal{A}(H^{\lambda+2l}_{2}(\Xi))\dotplus
\ker\mathcal{A}^{\ast},
\end{equation}
this sum being an extension of \eqref{orth-sum} \cite[Theorem~2.3.12]{Agranovich94}. Since the sequence $(f_k)$ converges in $\mathcal{Y}\hookrightarrow H^{\lambda}_{2}(\Xi)$, it follows from \eqref{direct-sum} that the sequence $(w_k)$ converges to a certain function $w\in\ker\mathcal{A}^{\ast}$ in any quasi-norm, specifically in $\|\cdot\|_{\mathcal{Y}}$. Thus,
\begin{equation*}
\mathcal{A}(u_k-v_k)=\mathcal{A}u_k-f_{k}+w_{k}\to w
\quad\mbox{in}\quad \mathcal{Z}\hookrightarrow\mathcal{D}'(\Xi).
\end{equation*}
This implies in view of   \eqref{orth-sum} that
\begin{equation*}
0=(\mathcal{A}(u_k-v_k),w)_{\Xi}\to(w,w)_{\Xi},
\end{equation*}
with $(\cdot,\cdot)_{\Xi}$ denoting the inner product in $L_2(\Xi)$. Hence, $w=0$, which yields
\begin{equation*}
\mathcal{A}v_k=f_k-w_k\to f-w=f\;\;\mbox{in $\mathcal{Y}$}.
\end{equation*}
Thus, the set $\mathcal{A}(\mathcal{M})$ is dense in $\mathcal{Y}\cap\mathcal{A}(\mathcal{X})$ in the quasi-norm $\|\cdot\|_{\mathcal{Y}}$.

Hence, the set $A(M)=\mathcal{K}(\mathcal{A}(\mathcal{M}))$ is dense in $Y\cap A(X)=\mathcal{K}(\mathcal{Y}\cap\mathcal{A}(\mathcal{X}))$ in the quasi-norm $\|\cdot\|_{Y}$ because $\mathcal{K}$ is a bounded operator  from $\mathcal{Y}$ to $Y$. As to the latter equality, note that if $Ag\in Y$ for certain $g\in X$, then $Ag=(\mathcal{A}h)\!\upharpoonright\!\Omega$ for every extension $h\in\mathcal{X}$ of $g$, where $\mathcal{A}h\in\mathcal{Z}$, hence, $Ag\in\mathcal{K}(\mathcal{Y}\cap\mathcal{A}(\mathcal{X}))$. Therefore, $Y\cap A(X)\subset\mathcal{K}(\mathcal{Y}\cap\mathcal{A}(\mathcal{X}))$; the inverse inclusion is trivial. Now we see in the $A(M)\subsetneqq M$ case that all hypotheses of  assertion~(ii) of Theorem~\ref{th-sep-dense} are satisfied for $T:=A$. Thus, $M$ is dense in the quasi-normed space $X(A,Y)$ due to this assertion.

We have proved assertions (i) and (ii) in the special case when $\partial\Omega$ is of class~$C^{\infty}$.

\emph{Step~$3$.} Consider now the general case \eqref{Lipschitz-boundary} and deduce these assertions from their counterparts for the special case just treated. We use the notation introduced on Step~3 of Theorem~\ref{main-interp-th}.

Let us deduce assertion~(i). Suppose that $X$ and $Y$ are separable and recall that $X=E^{s}_{p,q}(\Omega)$ under a certain choose of $E\in\{B,F\}$ and for some $s$, $p$, and $q$ subject to \eqref{indexes-case}. It follows from the separability of $X$ that
the quasi-Banach spaces $X^{\times}=E^{s}_{p,q}(\Omega^{\times})$ and $Z^{\times}=E^{s-2l}_{p,q}(\Omega^{\times})$ are also separable. Indeed, if $E^{s}_{p,q}(\Theta_{0})$ is separable for a certain open ball $\Theta_{0}$, then $E^{s}_{p,q}(\Theta)$ is separable for every ball $\Theta$ due to the invariance of $E^{s}_{p,q}(\mathbb{R}^n)$ with respect to $C^{\infty}$-diffeomorphisms of $\mathbb{R}^n$ onto itself \cite[Theorem~4.3.2]{Triebel92}. Choosing
$\Theta_{0}\subset\Omega$ and $\Theta\supset\Omega^{\times}$, we see that the separability of $X$ implies the same property for $X^{\times}$, which yields the separability of $Z^{\times}$ because the elliptic PDO $(1-\Delta)^{l}$ is a surjective bounded operator from $X^{\times}$ onto $Z^{\times}$. Since $Z^{\times}$ and $Y$ are separable, the quasi-normed space $Y^{\times}=Z^{\times}(K,Y)$ is also separable by assertion~(i) of Theorem~\ref{th-sep-dense} in which we let
$X:=Z^{\times}$, $T:=K$, $R:=J^{\times}$, and $S:=0$. Hence, the quasi-normed space $X^{\times}(A^{\times},Y^{\times})$ is also separable according to Step~2 applicable to the domain $\Omega^{\times}$ with boundary of class $C^{\infty}$. Thus, the quasi-normed space $X(A,Y)$ is separable because $K$ is a surjective bounded operator from $X^{\times}(A^{\times},Y^{\times})$ onto $X(A,Y)$. Assertion~(i) is proved in the general case.

Let us now proved assertion~(ii) in this case. We reason similarly to the proof of assertion~(i) just given. Assume that the set $M:=C^{\infty}(\overline{\Omega})$ is dense in $X$ and $Y$. Then the set $M^{\times}:=C^{\infty}(\Omega^{\times}\cup\partial\Omega^{\times})$ is dense in $X^{\times}$ and~$Z^{\times}$. Therefore, $M^{\times}$ is also dense in the quasi-normed space $Y^{\times}=Z^{\times}(K,Y)$ due to assertion~(ii) of Theorem~\ref{th-sep-dense} with the above letting. Hence, $M^{\times}$ is dense in $X^{\times}(A^{\times},Y^{\times})$ according to Step~2. Thus, $M$ is dense in $X(A,Y)$. Assertion~(ii) is proved in the general case.
\end{proof}

Applying Theorem~\ref{th-distr-sep-dense} to the $Y\in\mathcal{U}(\Omega)$ case, we deduce the following result:

\begin{theorem}\label{B-F-sep-dense}
Let $E,G\in\{B,F\}$ and $s,\alpha\in\mathbb{R}$ and $p,q,\beta,\gamma\in(0,\infty)$. Then the quasi-Banach space $E^{s}_{p,q}(A,G^{\alpha}_{\beta,\gamma},\Omega)$ is separable, and the set $C^{\infty}(\overline{\Omega})$ is dense in this space.
\end{theorem}

\begin{proof}
As is known \cite[Theorem~2.3.3]{Triebel83}, the set $\mathcal{S}(\mathbb{R}^{n})$ is dense in $E^{s}_{p,q}(\mathbb{R}^{n})$ because $p<\infty$ and $q<\infty$. Here, as usual, $\mathcal{S}(\mathbb{R}^{n})$ is the linear topological space of all rapidly decreasing $C^{\infty}$-functions on $\mathbb{R}^{n}$. Moreover, Step~5 of \cite[Proof of Theorem~2.3.3]{Triebel83} provides an arbitrary distribution $f\in E^{s}_{p,q}(\mathbb{R}^{n})$ with an approximating (in $E^{s}_{p,q}(\mathbb{R}^{n})$) sequence
$(f_k)\subset\mathcal{S}(\mathbb{R}^{n})$ that does not depend on  $E\in\{B,F\}$, $s\in\mathbb{R}$, and $p,q\in[1,\infty)$. Hence, the set $C^{\infty}(\overline{\Omega})$ is dense in the quasi-Banach spaces $E^{s}_{p,q}(\Omega)$ and $G^{\alpha}_{\beta,\gamma}(\Omega)\cap E^{s-2l}_{p,q}(\Omega)$. This implies that the continuous embeddings of
the separable Sobolev space $H^{\lambda}_{r}(\Omega)=F^{\lambda}_{r,2}(\Omega)$ in these spaces are dense; here $\lambda>\max\{s,\alpha\}$ and $r>\max\{1,p,\beta\}$. Thus, these spaces are also separable. Now the conclusion of Theorem~\ref{B-F-sep-dense} follows from Theorem~\ref{th-distr-sep-dense}.
\end{proof}

\section{Applications to elliptic problems with rough boundary data}\label{sec4}

Let $\Omega$ be a bounded open domain in $\mathbb{R}^{n}$ with boundary $\Gamma:=\partial\Omega$ of class $C^{\infty}$, with $n\geq2$ (the definition of such a domain is given, e.g., in \cite[Definition~3.2.1/2]{Triebel95}). We consider an elliptic boundary-value problem of the form
\begin{gather}\label{el-PDE}
Au=f\quad\mbox{in}\quad\Omega,\\
B_{j}u=g_{j}\quad\mbox{on}\quad\Gamma,
\quad j=1,\ldots,l. \label{bound-cond}
\end{gather}
Here, as above, $A$ is an elliptic linear PDO on $\overline{\Omega}=\Omega\cup\Gamma$ of an even order $2l\geq2$, and each $B_{j}$ is a linear boundary PDO on $\Gamma$ of an arbitrary order $m_{j}\geq0$. All coefficients of $A$ and  $B_{j}$ are supposed to be complex-valued infinitely smooth functions on $\overline{\Omega}$ and $\Gamma$, resp. Put $B:=(B_{1},\ldots,B_{l})$, $g:=(g_{1},\ldots,g_{l})$, and $m:=\max\{m_{1},\ldots,m_{l}\}$, with the case $m\geq2l$ being possible. The ellipticity of this problem means that $A$ is properly elliptic in $\overline{\Omega}$ and that the collection $B$ of boundary operators satisfy the Lopatinskii condition with respect to $A$ on $\Gamma$ (see, e.g., \cite[Section~1.2]{Agranovich97}). Recall that the elipticity of $A$ implies its proper ellipticity in the $n\geq3$ case.

It is known that the mapping $u\mapsto(Au,Bu)$ sets the Fredholm bounded linear operators
\begin{equation}\label{Fredholm-B-positive}
(A,B):B^{s}_{p,q}(\Omega)\to B^{s-2l}_{p,q}(\Omega)\times
\prod_{j=1}^{l}B^{s-m_j-1/p}_{p,q}(\Gamma)
\end{equation}
and
\begin{equation}\label{Fredholm-F-positive}
(A,B):F^{s}_{p,q}(\Omega)\to F^{s-2l}_{p,q}(\Omega)\times
\prod_{j=1}^{l}B^{s-m_j-1/p}_{p,p}(\Gamma)
\end{equation}
whenever $s$, $p$, and $q$ satisfy \eqref{indexes-case} and $s>m+\pi(p,n)$, with $\pi(p,n)$ being defined by \eqref{trace-th-parameter} and $Bu$ being understood in the sense of the trace theorem \cite[Theorem~3.3.3]{Triebel83}. These operators has a common kernel $N\subset C^{\infty}(\overline{\Omega})$ and index $\varkappa$, which do not depend on $s$, $p$, and $q$. Moreover, there exists a finite-dimensional space $M\subset C^{\infty}(\overline{\Omega})\times(C^{\infty}(\Gamma))^{l}$ that do not depend on $s$, $p$, and $q$ and that the target spaces of the operators \eqref{Fredholm-B-positive} and
\eqref{Fredholm-F-positive} split into the following direct sums:
\begin{equation}\label{direct-sum-positive-B}
B^{s-2l}_{p,q}(\Omega)\times\prod_{j=1}^{l}B^{s-m_j-1/p}_{p,q}(\Gamma)=
M\dotplus(A,B)(B^{s}_{p,q}(\Omega))
\end{equation}
and
\begin{equation}\label{direct-sum-positive-F}
F^{s-2l}_{p,q}(\Omega)\times\prod_{j=1}^{l}B^{s-m_j-1/p}_{p,p}(\Gamma)=
M\dotplus(A,B)(F^{s}_{p,q}(\Omega)).
\end{equation}
This result is contained in \cite[Theorem 5.2]{Johnsen96}. (It was also proved in \cite[Theorem~15]{FrankeRunst95} in the case where the elliptic problem is regular. A more specific case, when the regular elliptic problem sets isomorphisms \eqref{Fredholm-B-positive} and \eqref{Fredholm-F-positive}, was considered earlier in \cite[Theorem~4.3.3]{Triebel83} under additional restrictions on $s$, $p$, and $q$.)

Of course, $B^{\sigma}_{p,q}(\Gamma)$, with $\sigma\in\mathbb{R}$, denotes the Besov space over $\Gamma$ defined on the base of $B^{\sigma}_{p,q}(\mathbb{R}^{n-1})$ by the localization of $\Gamma$ (see, e.g., \cite[Definition~3.2.2/2(i)]{Triebel83}). The space $B^{\sigma}_{p,q}(\Gamma)$ is quasi-Banach one (is Banach one in the  $p,q\geq1$ case) and does not depend up to equivalence of quasi-norms on the localization procedure \cite[Proposition~3.2.3(ii)]{Triebel83}. (The same applies to the Triebel--Lizorkin space $F^{\sigma}_{p,q}(\Gamma)$.)

Note that the operators \eqref{Fredholm-B-positive} and \eqref{Fredholm-F-positive} coincide if $p=q$. The spaces on $\Gamma$ in  \eqref{Fredholm-F-positive} do not depend on~$q$. Since $B^{\sigma}_{p,p}(\Gamma)=F^{\sigma}_{p,p}(\Gamma)$, the operator \eqref{Fredholm-F-positive} can be considered as a mapping between Triebel--Lizorkin spaces whatever~$q$.

As to the notion of a Fredholm mapping for a bounded linear operator on a pair of quasi-Banach spaces, we exploit the same  definition as that used in the case of Banach spaces (see, e.g., \cite[Definition~19.1.2]{Hermander07v2}) and apply it to dual reach quasi-Banach spaces, as is done in \cite[Section~6.1]{FrankeRunst95}. Following \cite[Definition~2.5]{GeislerRunst91}, we say that a quasi-Banach space $X$ is dual reach if for every nonzero vector $u\in X$ there exists a continuous linear functional $\omega$ on $X$ subject to $\omega(u)=1$. Briefly saying, $X$ is dual reach if and only if it has a  separating dual.  Thus, a bounded linear operator $T:X\to Y$ acting between dual reach quasi-Banach spaces $X$ and $Y$ is called a Fredholm one if its kernel  $\ker T$ and co-kernel $\mathrm{coker}\,T:=Y/T(X)$ are finite-dimensional. If $T$ is a Fredholm operator, then its range $T(X)$ is closed in $Y$ and its index $\mathrm{ind}\,T:=\dim\ker T-\dim\mathrm{coker}\,T$ is well defined and finite. The closedness of $T(X)$ is proved in the same way as that used for Banach spaces (see, e.g., \cite[Lemma~19.1.1]{Hermander07v2}); the proof is based on Open Mapping Theorem, which applies to quasi-Banach spaces \cite[Section~2]{Kalton03}.

It is clear that every quasi-Banach space continuously embedded in a certain normed space is dual reach. Hence, the sum of such quasi-Banach spaces is also dual reach provided that they form an interpolation pair (we will use this property when applying Lemma~\ref{prop-interp-Fredholm}). All quasi-Banach spaces being used in this section will obey such embeddings. Thus, all the source and target spaces in \eqref{Fredholm-B-positive} and \eqref{Fredholm-F-positive} are dual reach because of the following continuous embeddings \cite[Theorem~3.3.1(i),(ii) and Remark 3.3.1/3]{Triebel83}:
\begin{equation}\label{embedding-Omega}
E^{s_1}_{p_1,q_1}(\Omega)\hookrightarrow G^{s_0}_{p_0,q_0}(\Omega)
\quad\mbox{whenever}\quad s_1-s_0>\max\biggl\{0,n\biggl(\frac{1}{p_1}-\frac{1}{p_0}\biggr)\biggr\}
\end{equation}
and
\begin{equation}\label{embedding-Gamma}
E^{s_1}_{p_1,q_1}(\Gamma)\hookrightarrow G^{s_0}_{p_0,q_0}(\Gamma)
\quad\mbox{whenever}\quad s_1-s_0>
\max\biggl\{0,(n-1)\biggl(\frac{1}{p_1}-\frac{1}{p_0}\biggr)\biggr\};
\end{equation}
here, $E,G\in\{B,F\}$, $s_0,s_1\in\mathbb{R}$, $0<p_0,p_1,q_0,q_1\leq\infty$, with $p_0\neq\infty$ in the $G=F$ case
and with $p_1\neq\infty$ in the $E=F$ case. Of course, whatever $s_1$, $p_1$, and $q_1$ we can always choose $s_{0}$
so that these embeddings hold true for $p_1=\max\{1,p_0\}$ and $q_1=\max\{1,q_0\}$, i.e. for some normed spaces on the right of the embeddings. Besides, the quasi-Banach distribution spaces over $\Omega$ or $\Gamma$ used in this paper are dual reach because they are continuously embedded in the linear topological spaces of all distributions on $\Omega$ or $\Gamma$ (see \cite[Proof of Proposition~3.1]{GeislerRunst91}).

\begin{remark}\label{rem-complement-space}
Assume $M$ to be a subspace of $C^{\infty}(\overline{\Omega})\times(C^{\infty}(\Gamma))^{l}$ such that either \eqref{direct-sum-positive-B} or \eqref{direct-sum-positive-F} holds true for certain $s$, $p$, and $q$ subject to conditions \eqref{indexes-case}  and $s>m+\pi(p,n)$. Then $M$ is finite-dimensional ($\dim M=-\varkappa+\dim N$), and both \eqref{direct-sum-positive-B} and \eqref{direct-sum-positive-F} are valid for all $s$, $p$, and $q$ just indicated \cite[Theorem 5.2 with the reference to assertion 3$^{\circ}$ of Theorem~1.3]{Johnsen96}. If $m\leq2q-1$, then the kernel of a certain adjoint boundary-value problem to problem \eqref{el-PDE}, \eqref{bound-cond} can be taken as such a space $M$. We will discuss this in Appendix~A.
\end{remark}

\begin{remark}\label{adjoint-kernel}
The kernels of the adjoint operators to the Fredholm bounded operators \eqref{Fredholm-B-positive} and \eqref{Fredholm-F-positive} are equal to a certain space $W$ which do not depend on $s$, $p$, and $q$, with $\dim W=-\varkappa+\dim N$. Indeed, choosing any two operators among \eqref{Fredholm-B-positive} and \eqref{Fredholm-F-positive}, we note that they are extensions of a certain Fredholm bounded operator \eqref{Fredholm-B-positive} with $s\gg1$. This follows from the continuous embedding \eqref{embedding-Omega}, which is dense if $p_{0}$ and $q_{0}$ are finite. Hence, the adjoint of the last operator is an extension of the adjoint of any choosing operator if these choosing operators act between spaces with finite subscripts. In this case, the kernels of the choosing operators are equal to a certain space $W$ because the kernels of these three operators are of the same dimension $-\varkappa+\dim N$. Moreover, since any choosing operator among \eqref{Fredholm-B-positive} and \eqref{Fredholm-F-positive} is a restriction of a certain operator \eqref{Fredholm-B-positive} acting between spaces with finite subscripts, the kernel of the adjoint of the choosing operator coincides with $W$ even in the case where $p=\infty$ and/or $q=\infty$. Note that a vector $(f,g)$ belongs to the range of the Fredholm operator \eqref{Fredholm-B-positive} (resp.,  \eqref{Fredholm-F-positive}) if and only if $(f,g)$ belongs to the target space of \eqref{Fredholm-B-positive} (resp., \eqref{Fredholm-F-positive}) and satisfies $\omega(f,g)=0$ for every functional $\omega\in W$ (see Proposition~\ref{prop-range} given in Appendix~A). In Appendix~A, we will discuss a representation of $W$ in terms $C^{\infty}$-functions belonging to the kernel of a certain adjoint boundary-value problem to problem \eqref{el-PDE}, \eqref{bound-cond}.
\end{remark}

The operators \eqref{Fredholm-B-positive} and \eqref{Fredholm-F-positive} can not be well-defined in the case where $s<m+\pi(p,n)$
\cite[Remark 2.7.2/4]{Triebel83}. Let us show that in this case we may take the spaces $B^{s}_{p,q}(A,B^{\alpha}_{p,q},\Omega)$ and $F^{s}_{p,q}(A,F^{\alpha}_{p,q},\Omega)$ with $\alpha>m-2l+\pi(p,n)$ as domains of the operator $(A,B)$ to obtain its boundedness and Fredholm property on relevant distribution spaces at least if $p$ and $q$ are finite.

\begin{theorem}\label{th-Fredholm}
Let $p,q\in(0,\infty)$ and
\begin{equation}\label{Fredholm-parameters}
s\leq m+\pi(p,n)\quad\mbox{and}\quad \alpha>m-2l+\pi(p,n).
\end{equation}
Then the mapping $u\mapsto(Au,Bu)$, with $u\in C^{\infty}(\overline{\Omega})$, extends uniquely (by continuity) to  bounded linear operators
\begin{equation}\label{Fredholm-B}
(A,B):B^{s}_{p,q}(A,B^{\alpha}_{p,q},\Omega)\to B^{\alpha}_{p,q}(\Omega)\times
\prod_{j=1}^{l}B^{s-m_j-1/p}_{p,q}(\Gamma)
\end{equation}
and
\begin{equation}\label{Fredholm-F}
(A,B):F^{s}_{p,q}(A,F^{\alpha}_{p,q},\Omega)\to F^{\alpha}_{p,q}(\Omega)\times
\prod_{j=1}^{l}B^{s-m_j-1/p}_{p,p}(\Gamma).
\end{equation}
They are Fredholm operators with the kernel $N$ and index $\varkappa$. Moreover, the target spaces of these operators split into the direct sums
\begin{equation}\label{direct-sum-negative-B}
B^{\alpha}_{p,q}(\Omega)\times\prod_{j=1}^{l}B^{s-m_j-1/p}_{p,q}(\Gamma)=
M\dotplus(A,B)(B^{s}_{p,q}(A,B^{\alpha}_{p,q},\Omega))
\end{equation}
and
\begin{equation}\label{direct-sum-negative-F}
F^{\alpha}_{p,q}(\Omega)\times\prod_{j=1}^{l}B^{s-m_j-1/p}_{p,p}(\Gamma)=
M\dotplus(A,B)(F^{s}_{p,q}(A,F^{\alpha}_{p,q},\Omega))
\end{equation}
for an arbitrary space $M$ satisfying the assumption made in Remark~$\ref{rem-complement-space}$.
\end{theorem}

Thus, this theorem is applicable to the elliptic problem \eqref{el-PDE} and \eqref{bound-cond} with boundary data of arbitrarily low (specifically, negative) regularity. Such data are said to be rough. Remark that, if $g_{1},\ldots, g_{l}$ are arbitrarily chosen distributions on $\Gamma$ and if $p,q\in(0,\infty)$, then there exists a small enough real number $s$ such that $g_{j}\in B^{s-m_j-1/p}_{p,g}(\Gamma)$ for each $j\in\{1,\ldots,l\}$. Hence, we may apply Theorem~\ref{th-Fredholm} to the elliptic problem with such boundary data if $f$ is sufficiently regular. Note that $f$ is allowed to have a certain negative regularity if $m\leq2l-1$ and $p\geq1$.

We will prove this theorem by a chain of interpolations of certain Fredholm operators \eqref{Fredholm-F-positive} and \eqref{Fredholm-F}. The Fredholm property of \eqref{Fredholm-F} is known in the special case where $p=q=2$ and additionally $\alpha>-1/2$ \cite[Theorem~3]{KasirenkoMikhailetsMurach19}. To this end we need the following abstract result concerning interpolation of Fredholm operators:

\begin{lemma}\label{prop-interp-Fredholm}
Let $[X_0,X_1]$ and $[Y_0,Y_1]$ be interpolation pairs of quasi-Banach spaces (resp., Banach spaces), and let the quasi-Banach space $X_{0}+X_{1}$ be dual reach. Suppose that certain Fredholm bounded linear operators $T_0:X_0\to Y_0$ and $T_1:X_1\to Y_1$ satisfy the following conditions:
\begin{itemize}
\item[(i)] $\ker T_0=\ker T_1=:\mathcal{N}$;
\item[(ii)] $\mathrm{ind}\,T_0=\mathrm{ind}\,T_1=:\kappa$;
\item[(iii)] $T_{0}u=T_{1}u$ whenever $u\in X_0\cap X_1$;
\item[(iv)] $Y_0=\mathcal{M}\dotplus T_0(X_0)$ and $Y_1=\mathcal{M}\dotplus T_1(X_1)$ for a certain  finite-dimensional space $\mathcal{M}\subset Y_0\cap Y_1$;
\item[(v)] $Y_1\cap T_{0}(X_0)\subset T_{0}(X_0\cap X_1)$ or $Y_0\cap T_{1}(X_1)\subset T_{1}(X_0\cap X_1)$.
\end{itemize}
Define a linear operator $T:X_0+X_1\to Y_0+Y_1$ by the formula $T(u_0+u_1):=T_0u_0+T_1u_1$ for all $u_0\in X_0$ and $u_1\in X_1$ (this  operator is well defined in view of \rm(iii)\it). Let $\mathfrak{F}$ be an arbitrary interpolation functor defined on the category of all interpolation pairs of quasi-Banach spaces (resp., Banach spaces). Then the restriction of $T$ to the space $\mathfrak{F}[X_0,X_1]$ is a Fredholm bounded operator $T_{\mathfrak{F}}:=T:\mathfrak{F}[X_0,X_1]\to\mathfrak{F}[Y_0,Y_1]$ such that
\begin{itemize}
\item[(a)] $\ker T_{\mathfrak{F}}=\mathcal{N}$;
\item[(b)] $\mathrm{ind}\,T_{\mathfrak{F}}=\kappa$;
\item[(c)] $T(\mathfrak{F}[X_0,X_1])=
    (T_0(X_0)+T_1(X_1))\cap\mathfrak{F}[Y_0,Y_1]$;
\item[(d)] $\mathfrak{F}[Y_0,Y_1]=\mathcal{M}\dotplus T(\mathfrak{F}[X_0,X_1])$.
\end{itemize}
\end{lemma}

A version of Lemma~\ref{prop-interp-Fredholm} is obtained in \cite[Proposition~5.2]{Geymonat65} in the case of Banach spaces and the dense continuous embedding $Y_{1}\hookrightarrow Y_{0}$. Various questions about the interpolation of Fredholm operators are considered, e.g., in \cite{AsekritovaKruglyakMastylo} for Banach spaces and in \cite{KaltonMitrea98} for quasi-Banach spaces. The proof of Lemma~\ref{prop-interp-Fredholm} is rather technical. We will give it in Appendix~B and then substantiate the following result:

\begin{corollary}\label{cor-interp-Fredholm}
Let $[X_0,X_1]$ and $[Y_0,Y_1]$ be interpolation pairs of quasi-Banach spaces (resp., Banach spaces) such that the dense continuous embeddings $X_1\hookrightarrow X_0$ and $Y_1\hookrightarrow Y_0$ are valid and that the quasi-Banach space $X_{0}$ is dual reach. Suppose that certain Fredholm bounded linear operators $T_0:X_0\to Y_0$ and $T_1:X_1\to Y_1$ satisfy hypotheses $(\mathrm{i})$--$(\mathrm{iii})$ of Lemma~$\ref{prop-interp-Fredholm}$. Then its conclusion hold true, with $T_0(X_0)+T_1(X_1)=T_0(X_0)$ and with $\mathcal{M}$ being a complement of $T_1(X_1)$ in $Y_1$ (such a complement exists).
\end{corollary}

\begin{proof}[Proof of Theorem $\ref{th-Fredholm}$.] We divide it in six steps, on each of which a certain pair of Fredholm operators will be interpolated. We will write equalities of quasi-Banach spaces up to equivalence of quasi-norms.

\emph{Step~$1$.} We first prove Theorem~\ref{th-Fredholm} in the $p=q=2$ case of Hilbert spaces. The Besov space $B^{\sigma}_{2,2}(\cdot)$ and Triebel--Lizorkin space $F^{\sigma}_{2,2}(\cdot)$ coincide with the Sobolev space $H^{\sigma}_{2}(\cdot)$ up to equivalence norms whenever $\sigma\in\mathbb{R}$. In this case, the operators \eqref{Fredholm-B} and \eqref{Fredholm-F} coincide and are considered on the pair of Hilbert spaces
\begin{equation}\label{Fredholm-H}
(A,B):H^{s}_{2}(A,H^{\alpha}_{2},\Omega)\to H^{\alpha}_{2}(\Omega)\times
\prod_{j=1}^{l}H^{s-m_j-1/2}_{2}(\Gamma)
\end{equation}
whenever $s\leq m+1/2$ and $\alpha>m-2l+1/2$; note that $\pi(2,n)=1/2$. If in addition $\alpha>-1/2$, then the required properties of \eqref{Fredholm-H} is due to \cite[Theorem~3]{KasirenkoMikhailetsMurach19} excepting the representation
\begin{equation}\label{direct-sum-H}
H^{\alpha}_{2}(\Omega)\times\prod_{j=1}^{l}H^{s-m_j-1/2}_{2}(\Gamma)=
M\dotplus(A,B)(H^{s}_{2}(A,H^{\alpha}_{2},\Omega))
\end{equation}
(the $m\geq2l$ case is covered by \cite[Theorem~1]{AnopKasirenkoMurach18UMJ3}, whereas the $m\leq2l-1$ case is treated analogously to \cite[Section~4.4.3]{MikhailetsMurach14}, where regular elliptic boundary-value problems are studied).

Let us substantiate \eqref{direct-sum-H} in the case of $\alpha>-1/2$.
As indicated, $M$ and co-kernel of \eqref{Fredholm-H} are of the same dimension $-\varkappa+\dim N$. Hence, \eqref{direct-sum-H} holds true if
\begin{equation}\label{M-range-0}
M\cap(A,B)(H^{s}_{2}(A,H^{\alpha}_{2},\Omega))=\{0\}.
\end{equation}
To show this equality we choose $s_{1}>\max\{m+1/2,\alpha+2l\}$ and consider
the Fredholm bounded operator
\begin{equation}\label{Fredholm-H-positive}
(A,B):H^{s_1}_{2}(\Omega)\to
H^{s_{1}-2l}_{2}(\Omega)\times\prod_{j=1}^{l}H^{s_1-m_j-1/2}_{2}(\Gamma),
\end{equation}
which is a restriction of the Fredholm operator \eqref{Fredholm-H}. Their co-kernels have the same dimension $-\varkappa+\dim N$. Then the kernels of the adjoint operators to \eqref{Fredholm-H} and \eqref{Fredholm-H-positive} coincide with the space $W$ indicated in Remark~\ref{adjoint-kernel}. Hence, if a vector $(f,g)\in M$ belongs to the range of \eqref{Fredholm-H}, then $\omega(f,g)=0$ for every functional $\omega\in W$, which implies that this vector belongs to the range of \eqref{Fredholm-H-positive}; i.e.,
\begin{equation*}
(f,g)\in M\cap(A,B)(H^{s}_{2}(A,H^{\alpha}_{2},\Omega))\;\Longrightarrow\;
(f,g)\in M\cap(A,B)(H^{s_1}_{2}(\Omega))=\{0\},
\end{equation*}
with the equality to $\{0\}$ being a special case of  \eqref{direct-sum-positive-F}. Thus, \eqref{M-range-0} holds true (which implies the required property \eqref{direct-sum-H}).

It remains to show that the conclusion of Theorem~\ref{th-Fredholm} is true in the $p=q=2$ case without the assumption $\alpha>-1/2$. This assumption imposes a supplementary restriction on $\alpha$ only in the $m\leq2l-2$ case (specifically for the Dirichlet boundary conditions). Thus, we consider this case and subject $\alpha$ only to the condition $\alpha>m-2l+1/2$. Choose numbers $s_1>m+1/2$ and $\alpha_0>-1/2$ so that $s_1-2l<\alpha<\alpha_0$. Then
\begin{equation}\label{mean-alpha}
\alpha=(1-\theta)\alpha_{0}+\theta(s_1-2l)
\quad\mbox{for a unique number}\quad\theta\in(0,1),
\end{equation}
and
\begin{equation}\label{mean-s}
s_0:=\frac{s-\theta s_1}{1-\theta}<s,\quad\mbox{with}\quad s=(1-\theta)s_0+\theta s_1
\end{equation}
(the last inequality is due to $s_1>m+1/2\geq s$).

Consider two Fredholm bounded operators
\eqref{Fredholm-H-positive} and
\begin{equation}\label{Fredholm-H-negative}
(A,B):H^{s_0}_{2}(A,H^{\alpha_0}_{2},\Omega)\to H^{\alpha_0}_{2}(\Omega)\times\prod_{j=1}^{l}H^{s_0-m_j-1/2}_{2}(\Gamma)
\end{equation}
(recall that $s_0<s\leq m+1/2$ and $\alpha_0>-1/2$). As was indicated, they have the common kernel $N$ and index $\varkappa$, and the space $M$ is a complement of the range of each operator in its target space. Thus, the operators \eqref{Fredholm-H-negative} and \eqref{Fredholm-H-positive} satisfy hypotheses (i), (ii), and (iv) of Lemma~\ref{prop-interp-Fredholm} provided that we write these operators in the form $T_{0}:X_{0}\to Y_{0}$ and $T_{1}:X_{1}\to Y_{1}$, resp. (Since $X_{0}\not\subset X_{1}$ and $X_{1}\not\subset X_{0}$, we cannot apply Corollary~\ref{cor-interp-Fredholm} to these operators.) Moreover, they also satisfy hypothesis (iii), i.e. coincide on the Banach space
\begin{equation}\label{intersection-domains}
X_{0}\cap X_{1}=H^{s_0}_{2}(A,H^{\alpha_0}_{2},\Omega)\cap H^{s_1}_{2}(\Omega)=
H^{s_1}_{2}(A,H^{\alpha_0}_{2},\Omega).
\end{equation}
This is true because $C^{\infty}(\overline{\Omega})$ is dense in $H^{s_1}_{2}(A,H^{\alpha_0}_{2},\Omega)$ by Theorem~\ref{B-F-sep-dense}
and since these operators coincide on $C^{\infty}(\overline{\Omega})$. The operators also satisfy hypothesis (v). Indeed, if a vector $(f,g)$ belongs to $Y_{1}\cap T_{0}(X_{0})$, then $\omega(f,g)=0$ for every functional $\omega\in W$, which implies that
$(f,g)\in T_{1}(X_{1})$. Recall that $W$ is the common kernel of the adjoint operators to \eqref{Fredholm-H-negative} and \eqref{Fredholm-H-positive}. Thus, $(A,B)u=(f,g)$ for a certain distribution $u\in H^{s_1}_{2}(\Omega)$, with $Au=f\in H^{\alpha_0}_{2}(\Omega)$. We conclude by \eqref{intersection-domains} that $u\in X_{0}\cap X_{1}$, which yields the inclusion $(f,g)\in T_{0}(X_{0}\cap X_{1})$, as hypothesis~(v) claims.

Defining the operator $T:X_0+X_1\to Y_0+Y_1$ as in Lemma~\ref{prop-interp-Fredholm} and interpolating the Fredholm bounded operators \eqref{Fredholm-H-negative} and \eqref{Fredholm-H-positive} by means of the complex interpolation functor $\mathfrak{F}[X_0,X_1]=[X_0,X_1]_{\theta}$, we conclude by this lemma that the restriction of $T$ to the space $[X_0,X_1]_{\theta}$ is a Fredholm bounded operator from this space to $[Y_0,Y_1]_{\theta}$. Here,
\begin{align*}
[X_0,X_1]_{\theta}&=
[H^{s_0}_{2}(A,H^{\alpha_0}_{2},\Omega),H^{s_1}_{2}(\Omega)]_{\theta}=
[H^{s_0}_{2}(A,H^{\alpha_0}_{2},\Omega),
H^{s_1}_{2}(A,H^{s_1-2l}_{2}\Omega)]_{\theta}\\
&=H^{s}_{2}(A,H^{\alpha}_{2},\Omega)
\end{align*}
by Theorem~\ref{C-interp} in view of \eqref{mean-alpha} and \eqref{mean-s}. (Note that $H^{\alpha_0}_{2}(\Omega)\subset H^{s_0-2l}_{2}(\Omega)$, which follows from $s_0-2l<s_1-2l<\alpha<\alpha_0$ and is required in Theorem~\ref{C-interp}). As is well known,
\begin{align*}
[Y_0,Y_1]_{\theta}&=
[H^{\alpha_0}_{2}(\Omega),H^{s_1-2l}_{2}(\Omega)]_{\theta}
\times\prod_{j=1}^{l}
[H^{s_0-m_j-1/2}_{2}(\Gamma),H^{s_1-m_j-1/2}_{2}(\Gamma)]_{\theta}\\
&=H^{\alpha}_{2}(\Omega)\times
\prod_{j=1}^{l}H^{s-m_j-1/2}_{2}(\Gamma).
\end{align*}
Since $C^{\infty}(\overline{\Omega})$ is dense in $H^{s}_{2}(A,H^{\alpha}_{2},\Omega)$ by Theorem~\ref{B-F-sep-dense}, the last Fredholm bounded operator is a unique extension by continuity of the mapping $u\mapsto(Au,Bu)$, with $u\in C^{\infty}(\overline{\Omega})$. Owing to Lemma~\ref{prop-interp-Fredholm}, this operator has all properties indicated in Theorem~\ref{th-Fredholm}. The case $p=q=2$ has been  treated.

\emph{Step~$2$.} This step and the next two are devoted to the proof of  the Fredholm property for operator \eqref{Fredholm-F} in the case where $p\neq2$ and $q=2$. Put $p_{1}:=p/2$ if $p<2$, and let $p_{1}:=2p$ if $p>2$. Since $1/p$ is located between $1/2$ and $1/p_{1}$, we conclude that
\begin{equation}\label{mean-p}
\frac{1}{p}=\frac{1-\theta}{2}+\frac{\theta}{p_1}
\quad\mbox{for a unique number}\quad\theta\in(0,1).
\end{equation}
Suppose on this step that
\begin{equation}\label{step2-cond-alpha}
\alpha>m-2l+\max\biggl\{\pi(p,n),\pi(p_1,n),\frac{1}{2}\biggr\}
\end{equation}
and
\begin{equation}\label{step2-cond-s}
s<-\max\biggl\{0,n\biggl(\frac{1}{p_1}-\frac{1}{2}\biggr)\biggr\}.
\end{equation}
Put
\begin{equation}\label{step2-mean-s}
s_0:=\frac{s-\theta(\alpha+2l)}{1-\theta}<s,\quad\mbox{with}\quad s=(1-\theta)s_0+\theta(\alpha+2l).
\end{equation}
We consider the Fredholm bounded operators
\begin{equation}\label{step2-Fredholm-H-negative}
(A,B):F^{s_0}_{2,2}(A,F^{\alpha}_{2,2},\Omega)\to F^{\alpha}_{2,2}(\Omega)\times
\prod_{j=1}^{l}B^{s_0-m_j-1/2}_{2,2}(\Gamma)
\end{equation}
and
\begin{equation}\label{step2-Fredholm-F-positive}
(A,B):F^{\alpha+2l}_{p_1,2}(\Omega)\to F^{\alpha}_{p_1,2}(\Omega)\times
\prod_{j=1}^{l}B^{\alpha+2l-m_j-1/p_1}_{p_1,p_1}(\Gamma).
\end{equation}
They have the common kernel $N$ and index $\varkappa$, and the space $M$ is a complement of the range of each operator in its target space. Recall that these properties were proved for \eqref{step2-Fredholm-H-negative} on Step~1, and note that \eqref{step2-cond-alpha} allows these operators to be well defined.

Thus, the operators \eqref{step2-Fredholm-H-negative} and \eqref{step2-Fredholm-F-positive} satisfy hypotheses (i), (ii), and (iv) of Lemma~\ref{prop-interp-Fredholm} provided that we write the operators in the form $T_{0}:X_{0}\to Y_{0}$ and $T_{1}:X_{1}\to Y_{1}$, resp. They satisfy hypothesis (iii), i.e. coincide on the quasi-Banach space
\begin{equation}\label{step2-intersection-domains}
X_{0}\cap X_{1}=F^{s_0}_{2,2}(A,F^{\alpha}_{2,2},\Omega)\cap F^{\alpha+2l}_{p_1,2}(\Omega)=
F^{\alpha+2l}_{p_1,2}(A,F^{\alpha}_{2,2},\Omega).
\end{equation}
Here, the last equality is true because  $F^{\alpha+2l}_{p_1,2}(\Omega)\hookrightarrow F^{s_0}_{2,2}(\Omega)$ due to \eqref{step2-cond-s} and \eqref{embedding-Omega}. We also take into account that  $C^{\infty}(\overline{\Omega})$ is dense in $F^{\alpha+2l}_{p_1,2}(A,F^{\alpha}_{2,2},\Omega)$ by Theorem~\ref{B-F-sep-dense} and that these operators coincide on $C^{\infty}(\overline{\Omega})$. The operators \eqref{step2-Fredholm-H-negative} and \eqref{step2-Fredholm-F-positive} satisfy hypothesis (v), which is proved in a similar way to that used on Step~1. Indeed, if a vector $(f,g)$ belongs to $Y_{1}\cap T_{0}(X_{0})$, then $\omega(f,g)=0$ whenever $\omega\in W$, which implies
$(f,g)\in T_{1}(X_{1})$. Recall that $W$ (indicated in Remark~\ref{adjoint-kernel}) is the common kernel of the adjoint operators to \eqref{step2-Fredholm-H-negative} and \eqref{step2-Fredholm-F-positive}, as was shown on Step~1. Thus, $(A,B)u=(f,g)$ for a certain distribution $u\in F^{\alpha+2l}_{p_1,2}(\Omega)$, with $Au=f\in F^{\alpha}_{2,2}(\Omega)$. Therefore, $(f,g)\in T_{0}(X_{0}\cap X_{1})$
by \eqref{step2-intersection-domains}. (Remark that $X_{0}\not\subset X_{1}$ and that $p_{1}<2$ implies $X_{1}\not\subset X_{0}$; hence we cannot resort to Corollary~\ref{cor-interp-Fredholm} in the $p<2$ case, instead of Lemma~\ref{prop-interp-Fredholm}.)

Defining the operator $T:X_0+X_1\to Y_0+Y_1$ as in Lemma~\ref{prop-interp-Fredholm} and interpolating the Fredholm bounded operators \eqref{step2-Fredholm-H-negative} and \eqref{step2-Fredholm-F-positive} by the interpolation functor
$\mathfrak{F}[X_0,X_1]=\langle X_{0},X_{1},\theta\rangle$ (used in Theorem~\ref{plus-minus-interp}), we conclude by this lemma that the restriction of $T$ to the space $\langle X_{0},X_{1},\theta\rangle$ is a Fredholm bounded operator from this space to $\langle Y_{0},Y_{1},\theta\rangle$. Here,
\begin{align*}
\langle X_{0},X_{1},\theta\rangle&=
\langle F^{s_0}_{2,2}(A,F^{\alpha}_{2,2},\Omega), F^{\alpha+2l}_{p_1,2}(\Omega),\theta\rangle=
\langle F^{s_0}_{2,2}(A,F^{\alpha}_{2,2},\Omega), F^{\alpha+2l}_{p_1,2}(A,F^{\alpha}_{p_1,2},\Omega),\theta\rangle\\
&=F^{s}_{p,2}(A,F^{\alpha}_{p,2},\Omega)
\end{align*}
by Theorem~\ref{plus-minus-interp} in view of \eqref{mean-p} and \eqref{step2-mean-s}. (Note that $F^{\alpha}_{2,2}(\Omega)\subset F^{s_0-2l}_{2,2}(\Omega)$, which results from $\alpha>0>s_0-2l$ and is required by the hypothesis \eqref{inclusion-C-interp} of this theorem.) Furthermore,
\begin{align*}
\langle Y_{0},Y_{1},\theta\rangle&=
\langle F^{\alpha}_{2,2}(\Omega),F^{\alpha}_{p_1,2}(\Omega),\theta\rangle
\times\prod_{j=1}^{l}
\langle B^{s_0-m_j-1/2}_{2,2}(\Gamma),
B^{\alpha+2l-m_j-1/p_1}_{p_1,p_1}(\Gamma),\theta\rangle\\
&=F^{\alpha}_{p,2}(\Omega)\times
\prod_{j=1}^{l}B^{s-m_j-1/p}_{p,p}(\Gamma),
\end{align*}
which follows from the interpolation formulas for counterparts of these spaces for $\mathbb{R}^{n}$ (see \cite[p.~134]{FrazierJawerth90} or \cite[Theorem~2.12 (i)]{YuanSickelYang15}; the transition from $\mathbb{R}^{n-1}$ to $\Gamma$ in these formulas is a standard procedure, cf. \cite[Proof of Theorem~2.2]{MikhailetsMurach14}). Since $C^{\infty}(\overline{\Omega})$ is dense in $F^{s}_{p,2}(A,F^{\alpha}_{p,2},\Omega)$ by Theorem~\ref{B-F-sep-dense}, the last Fredholm bounded operator is a unique extension by continuity of the mapping $u\mapsto(Au,Bu)$, with $u\in C^{\infty}(\overline{\Omega})$. Owing to Lemma~\ref{prop-interp-Fredholm}, the operator \eqref{Fredholm-F} is well defined and has all properties indicated in Theorem~\ref{th-Fredholm} under the conditions $p\neq2$, $q=2$, \eqref{step2-cond-alpha}, and \eqref{step2-cond-s}.

\emph{Step~$3$.} We continue to consider the case where $p\neq2$ and $q=2$ and suppose that $\alpha$ satisfies condition \eqref{step2-cond-alpha}, with $p_{1}$ being defined as above. However, we do note subject $s$ to the supplementary condition \eqref{step2-cond-s}. Choose a number $s_{0}$ so that
\begin{equation*}
s_{0}<-\max\biggl\{0,n\biggl(\frac{1}{p_1}-\frac{1}{2}\biggr)\biggr\}
\quad\mbox{and}\quad s_{0}<s.
\end{equation*}
Since $s_{0}<s<\alpha+2l$, we have $s=(1-\theta)s_0+\theta(\alpha+2l)$ for a unique number $\theta\in(0,1)$.

Consider the Fredholm bounded operators
\begin{equation}\label{step3-Fredholm-negative}
(A,B):F^{s_0}_{p,2}(A,F^{\alpha}_{p,2},\Omega)\to F^{\alpha}_{p,2}(\Omega)\times
\prod_{j=1}^{l}B^{s_0-m_j-1/p}_{p,p}(\Gamma)
\end{equation}
and
\begin{equation}\label{step3-Fredholm-positive}
(A,B):F^{\alpha+2l}_{p,2}(\Omega)\to F^{\alpha}_{p,2}(\Omega)\times
\prod_{j=1}^{l}B^{\alpha+2l-m_j-1/p}_{p,p}(\Gamma).
\end{equation}
They have the common kernel $N$ and index $\varkappa$, and the space $M$ is a complement of the range of each operator in its target space. These properties were proved for \eqref{step3-Fredholm-negative} on Step~2, and \eqref{step2-cond-alpha} allows these operators to be well defined. Thus, the operators \eqref{step3-Fredholm-negative} and \eqref{step3-Fredholm-positive} satisfy the hypotheses of Corollary~\ref{cor-interp-Fredholm} provided that we denote them by  $T_{0}:X_{0}\to Y_{0}$ and $T_{1}:X_{1}\to Y_{1}$, resp. Note that the density of the imbedding $X_{1}\hookrightarrow X_{0}$ follows from Theorem~\ref{B-F-sep-dense}.

Owing to Corollary~\ref{cor-interp-Fredholm}, the restriction of the operator \eqref{step3-Fredholm-negative} to the space $\langle X_{0},X_{1},\theta\rangle$ is a Fredholm bounded operator from this space to $\langle Y_{0},Y_{1},\theta\rangle$. Here,
\begin{align*}
\langle X_{0},X_{1},\theta\rangle&=
\langle F^{s_0}_{p,2}(A,F^{\alpha}_{p,2},\Omega), F^{\alpha+2l}_{p,2}(\Omega),\theta\rangle=
\langle F^{s_0}_{p,2}(A,F^{\alpha}_{p,2},\Omega), F^{\alpha+2l}_{p,2}(A,F^{\alpha}_{p,2},\Omega),\theta\rangle\\
&=F^{s}_{p,2}(A,F^{\alpha}_{p,2},\Omega)
\end{align*}
by Theorem~\ref{plus-minus-interp}, and
\begin{align*}
\langle Y_{0},Y_{1},\theta\rangle&=
\langle F^{\alpha}_{p,2}(\Omega),F^{\alpha}_{p,2}(\Omega),\theta\rangle
\times\prod_{j=1}^{l}
\langle B^{s_0-m_j-1/p}_{p,p}(\Gamma),
B^{\alpha+2l-m_j-1/p}_{p,p}(\Gamma),\theta\rangle\\
&=F^{\alpha}_{p,2}(\Omega)\times
\prod_{j=1}^{l}B^{s-m_j-1/p}_{p,p}(\Gamma).
\end{align*}
Since $C^{\infty}(\overline{\Omega})$ is dense in $F^{s}_{p,p}(A,F^{\alpha}_{p,p},\Omega)$ by Theorem~\ref{B-F-sep-dense}, this Fredholm bounded operator is a unique extension by continuity of the mapping $u\mapsto(Au,Bu)$, with $u\in C^{\infty}(\overline{\Omega})$. Hence, according to Corollary~\ref{cor-interp-Fredholm}, the operator \eqref{Fredholm-F} is well defined and has all properties indicated in Theorem~\ref{th-Fredholm} under the supplementary conditions $p\neq2$, $q=2$, and \eqref{step2-cond-alpha}.

\emph{Step~$4$.} Considering the same case where $p\neq2$ and $q=2$, we can now prove the Fredholm property for operator \eqref{Fredholm-F} without imposing \eqref{step2-cond-alpha} on $\alpha$. Our reasoning will be somewhat similar to that given on Step~1. Choose numbers $s_1>m+\pi(p,n)$ and
\begin{equation}\label{step3-cond-alpha0}
\alpha_{0}>m-2l+\max\biggl\{\pi(p,n),\pi(p_1,n),\frac{1}{2}\biggr\}
\end{equation}
so that $s_1-2l<\alpha<\alpha_0$. Then \eqref{mean-alpha} and \eqref{mean-s} hold true.

Consider two Fredholm bounded operators
\begin{equation}\label{step4-Fredholm-negative}
(A,B):F^{s_0}_{p,2}(A,F^{\alpha_0}_{p,2},\Omega)\to F^{\alpha_0}_{p,2}(\Omega)\times
\prod_{j=1}^{l}B^{s_0-m_j-1/p}_{p,p}(\Gamma)
\end{equation}
and
\begin{equation}\label{step4-Fredholm-positive}
(A,B):F^{s_1}_{p,2}(\Omega)\to F^{s_1-2l}_{p,2}(\Omega)\times
\prod_{j=1}^{l}B^{s_1-m_j-1/p}_{p,p}(\Gamma).
\end{equation}
They have the common kernel $N$ and index $\varkappa$, and the space $M$ is a complement of the range of each operator in its target space. These properties were proved for \eqref{step4-Fredholm-negative} on Step~3 because $\alpha_0$ satisfies \eqref{step3-cond-alpha0}. Thus, the operators \eqref{step4-Fredholm-negative} and \eqref{step4-Fredholm-positive} written as $T_{0}:X_{0}\to Y_{0}$ and $T_{1}:X_{1}\to Y_{1}$ resp. satisfy hypotheses (i), (ii), and (iv) of Lemma~\ref{prop-interp-Fredholm}. (Note that $X_{0}\not\subset X_{1}$ and $X_{1}\not\subset X_{0}$ because $s_{0}<s_{1}$ and $s_1-2l<\alpha_0$.) These operators satisfy hypothesis (iii) because they coincide on the set $C^{\infty}(\overline{\Omega})$, which is dense in the quasi-Banach space
\begin{equation}\label{intersection-domains-F}
X_{0}\cap X_{1}=F^{s_0}_{p,2}(A,F^{\alpha_0}_{p,2},\Omega)\cap F^{s_1}_{p,2}(\Omega)=
F^{s_1}_{p,2}(A,F^{\alpha_0}_{p,2},\Omega)
\end{equation}
by Theorem~\ref{B-F-sep-dense}. The operators also satisfy hypothesis (v). Indeed, if $(f,g)\in Y_{1}\cap T_{0}(X_{0})$, then $\omega(f,g)=0$ whenever $\omega\in W$, which implies $(f,g)\in T_{1}(X_{1})$. Thus, $(A,B)u=(f,g)$ for certain $u\in F^{s_1}_{p,2}(\Omega)$, with $Au=f\in F^{\alpha_0}_{p,2}(\Omega)$. This yields $(f,g)\in T_{0}(X_{0}\cap X_{1})$ in view of \eqref{intersection-domains-F}.

Defining the operator $T:X_0+X_1\to Y_0+Y_1$ as in Lemma~\ref{prop-interp-Fredholm} and interpolating the Fredholm bounded operators \eqref{step4-Fredholm-negative} and \eqref{step4-Fredholm-positive}, we conclude by this lemma that the restriction of $T$ to the space $\langle X_{0},X_{1},\theta\rangle$ is a Fredholm bounded operator from this space to $\langle Y_{0},Y_{1},\theta\rangle$. Here,
\begin{align*}
\langle X_{0},X_{1},\theta\rangle&=
\langle F^{s_0}_{p,2}(A,F^{\alpha_0}_{p,2},\Omega), F^{s_1}_{p,2}(\Omega),\theta\rangle=
\langle F^{s_0}_{p,2}(A,F^{\alpha_0}_{p,2},\Omega), F^{s_1}_{p,2}(A,F^{s_1-2l}_{p,2},\Omega),\theta\rangle\\
&=F^{s}_{p,2}(A,F^{\alpha}_{p,2},\Omega)
\end{align*}
by Theorem~\ref{plus-minus-interp} in view of \eqref{mean-alpha} and \eqref{mean-s}. (Note that $F^{\alpha_0}_{p,2}(\Omega)\subset F^{s_0-2l}_{p,2}(\Omega)$ because $\alpha_{0}>s_1-2l>s_0-2l$.) Furthermore,
\begin{align*}
\langle Y_{0},Y_{1},\theta\rangle&=
\langle F^{\alpha_0}_{p,2}(\Omega),F^{s_1-2l}_{p,2}(\Omega),\theta\rangle
\times\prod_{j=1}^{l}
\langle B^{s_0-m_j-1/p}_{p,p}(\Gamma),
B^{s_1-m_j-1/p}_{p,p}(\Gamma),\theta\rangle\\
&=F^{\alpha}_{p,2}(\Omega)\times
\prod_{j=1}^{l}B^{s-m_j-1/p}_{p,p}(\Gamma),
\end{align*}
By Lemma~\ref{prop-interp-Fredholm}, the operator \eqref{Fredholm-F} has all properties indicated in Theorem~\ref{th-Fredholm} in the case under consideration.

\emph{Step~$5$.} As to the operator \eqref{Fredholm-F}, it remains to prove its Fredholm property in the case where $q\neq2$. Considering this case, we put $q_{1}:=q/2$ if $q<2$, and let $q_{1}:=2q$ if $q>2$. Hence,
\begin{equation}\label{mean-q}
\frac{1}{q}=\frac{1-\theta}{2}+\frac{\theta}{q_1}
\quad\mbox{for a unique number}\quad\theta\in(0,1).
\end{equation}
We choose a number $\alpha_{0}>m-2l+\pi(p,n)$ so that $\alpha_{0}<\alpha$ and put
\begin{equation}\label{step5-mean-alpha}
\alpha_{1}:=\frac{\alpha-(1-\theta)\alpha_{0}}{\theta}>\alpha,
\quad\mbox{with}\quad \alpha=(1-\theta)\alpha_{0}+\theta\alpha_{1}.
\end{equation}
We then let
\begin{equation}\label{step5-mean-s}
s_0:=\frac{s-\theta(\alpha_1+2l)}{1-\theta}<s,\quad\mbox{with}\quad s=(1-\theta)s_0+\theta(\alpha_1+2l)
\end{equation}
(the last inequality holds true because $\alpha_1+2l>m+\pi(p,n)\geq s$).

Consider the Fredholm bounded operators \eqref{step4-Fredholm-negative} and
\begin{equation}\label{step5-Fredholm-positive}
(A,B):F^{\alpha_1+2l}_{p,q_1}(\Omega)\to F^{\alpha_1}_{p,q_1}(\Omega)\times
\prod_{j=1}^{l}B^{\alpha_1+2l-m_j-1/p}_{p,p}(\Gamma).
\end{equation}
They have the common kernel $N$ and index $\varkappa$, with $M$ being a complement of the range of each operator. Thus, these operators denoted by $T_{0}:X_{0}\to Y_{0}$ and $T_{1}:X_{1}\to Y_{1}$ satisfy the hypotheses of Corollary~\ref{cor-interp-Fredholm}. (The dense continuous embeddings $X_{1}\hookrightarrow X_{0}$ and $Y_{1}\hookrightarrow Y_{0}$ follow from $s_{0}<\alpha_1+2l$, $\alpha_0<\alpha_1$, and Theorem~\ref{B-F-sep-dense}.) Owing to Corollary~\ref{cor-interp-Fredholm}, the restriction of the operator \eqref{step4-Fredholm-negative} to the space $\langle X_{0},X_{1},\theta\rangle$ is a Fredholm bounded operator from
\begin{align*}
\langle X_{0},X_{1},\theta\rangle&=
\langle F^{s_0}_{p,2}(A,F^{\alpha_0}_{p,2},\Omega), F^{\alpha_{1}+2l}_{p,q_1}(\Omega),\theta\rangle=
\langle F^{s_0}_{p,2}(A,F^{\alpha_0}_{p,2},\Omega), F^{\alpha_{1}+2l}_{p,q_1}(A,F^{\alpha_{1}}_{p,q_1},\Omega),\theta\rangle\\
&=F^{s}_{p,q}(A,F^{\alpha}_{p,q},\Omega)
\end{align*}
to
\begin{align*}
\langle Y_{0},Y_{1},\theta\rangle&=
\langle F^{\alpha_{0}}_{p,2}(\Omega),F^{\alpha_{1}}_{p,q_1}(\Omega),\theta\rangle
\times\prod_{j=1}^{l}
\langle B^{s_0-m_j-1/p}_{p,p}(\Gamma),
B^{\alpha_1+2l-m_j-1/p}_{p,p}(\Gamma),\theta\rangle\\
&=F^{\alpha}_{p,2}(\Omega)\times
\prod_{j=1}^{l}B^{s-m_j-1/p}_{p,p}(\Gamma).
\end{align*}
The last equalities in these formulas are due to  \eqref{mean-q}--\eqref{step5-mean-s} and Theorem~\ref{plus-minus-interp}. According to Corollary~\ref{cor-interp-Fredholm}, the operator \eqref{Fredholm-F} has all the required properties.

\emph{Step~$6$.} We will finish our proof by demonstrating the Fredholm property for the operator \eqref{Fredholm-B}. We reason similarly to Step~5 but use the real interpolation functor $\mathfrak{F}[X_0,X_1]=(X_0,X_1)_{\theta,q}$. Put, e.g., $\theta:=1/2$; choose a number $\alpha_{0}>m-2l+\pi(p,n)$ so that $\alpha_{0}<\alpha$, and define numbers $\alpha_{1}$ and $s_{0}$ by formulas \eqref{step5-mean-alpha} and \eqref{step5-mean-s}, resp. Consider the Fredholm bounded operators \eqref{step4-Fredholm-negative} and \eqref{step5-Fredholm-positive} for $q_{1}=2$, used on the previous step and denoted by $T_{0}:X_{0}\to Y_{0}$ and $T_{1}:X_{1}\to Y_{1}$. Owing to Corollary~\ref{cor-interp-Fredholm}, the restriction of the operator \eqref{step4-Fredholm-negative} to the space $(X_0,X_1)_{\theta,q}$ is a Fredholm bounded operator from
\begin{align*}
(X_0,X_1)_{\theta,q}&=
(F^{s_0}_{p,2}(A,F^{\alpha_0}_{p,2},\Omega), F^{\alpha_{1}+2l}_{p,2}(\Omega))_{\theta,q}=
(F^{s_0}_{p,2}(A,F^{\alpha_0}_{p,2},\Omega), F^{\alpha_{1}+2l}_{p,2}(A,F^{\alpha_{1}}_{p,2},\Omega))_{\theta,q}\\
&=B^{s}_{p,q}(A,B^{\alpha}_{p,q},\Omega)
\end{align*}
to
\begin{align*}
(Y_{0},Y_{1})_{\theta,q}&=
(F^{\alpha_{0}}_{p,2}(\Omega),F^{\alpha_{1}}_{p,2}(\Omega))_{\theta,q}
\times\prod_{j=1}^{l}
(B^{s_0-m_j-1/p}_{p,p}(\Gamma),
B^{\alpha_1+2l-m_j-1/p}_{p,p}(\Gamma))_{\theta,q}\\
&=B^{\alpha}_{p,q}(\Omega)\times
\prod_{j=1}^{l}B^{s-m_j-1/p}_{p,q}(\Gamma).
\end{align*}
The last equalities in these formulas hold true by Theorem~\ref{R-interp} and relations \eqref{step5-mean-alpha} and \eqref{step5-mean-s} (note that $s_{0}<\alpha_1+2l$ and $\alpha_0<\alpha_1$). According to Corollary~\ref{cor-interp-Fredholm}, the operator \eqref{Fredholm-F} has all properties stated in Theorem~\ref{th-Fredholm}.
\end{proof}

Using the real interpolation functor $(X_0,X_1)_{\theta,\infty}$, we can  deduce a version of Theorem~\ref{th-Fredholm} for the spaces
$B^{s}_{p,\infty}(\Omega)$, which involve the Nikolskii spaces. Formulating this version, we must take into account that $C^{\infty}(\overline{\Omega})$ is not dense in $B^{s}_{p,\infty}(\Omega)$. We note in this connection that the bounded operators \eqref{Fredholm-B} (or \eqref{Fredholm-F}) from Theorem~\ref{th-Fredholm} naturally set a linear mapping
\begin{equation}\label{general-map-B}
(A,B):\bigcup_{\substack{s< m+\pi(p,n),\\\alpha>m-2l+\pi(p,n)}} B^{s}_{p,p}(A,B^{\alpha}_{p,p},\Omega)\to
\mathcal{D}'(\Omega)\times(\mathcal{D}'(\Gamma))^{l}
\end{equation}
whenever $0<p<\infty$. As usual, $\mathcal{D}'(\Omega)$ and  $\mathcal{D}'(\Gamma)$ stand for the linear topological spaces of all distributions on $\Omega$ and $\Gamma$, resp. Recall that
\begin{equation*}
B^{s}_{p,q}(\Omega)\hookrightarrow B^{s}_{p,\infty}(\Omega)\hookrightarrow B^{s-\varepsilon}_{p,q}(\Omega)
\quad\mbox{if}\;\;\varepsilon>0\;\;\mbox{and}\;\;0<q<\infty
\end{equation*}
\cite[Proposition 3.2.4(i) and Theorem 3.3.1(i)]{Triebel83}.

\begin{theorem}\label{th-Fredholm-Nikolskii}
Let numbers $p$, $s$, and $\alpha$ satisfy conditions $0<p<\infty$ and \eqref{Fredholm-parameters}. Then the restriction of mapping \eqref{general-map-B} to  $B^{s}_{p,\infty}(A,B^{\alpha}_{p,\infty},\Omega)$ is a Fredholm bounded operator on the pair of spaces \eqref{Fredholm-B} where $q:=\infty$. This operator has the kernel $N$ and index $\varkappa$; its target space admits representation \eqref{direct-sum-negative-B}, with $q=\infty$ and with $M$ satisfying the assumption made in Remark~$\ref{rem-complement-space}$.
\end{theorem}

\begin{proof}
It is the same as Step~6 of the previous proof provided that we put $q:=\infty$.
\end{proof}

\begin{remark}\label{adjoint-kernel-negative}
The kernels of the adjoint operators to the Fredholm bounded operators \eqref{Fredholm-B} and \eqref{Fredholm-F} from Theorems \ref{th-Fredholm}
and \ref{th-Fredholm-Nikolskii} have the dimension $-\varkappa+\dim N$ and hence equal the space $W$ indicated in Remark~\ref{adjoint-kernel}. This conclusion is true for the Fredholm operators from Theorem~\ref{th-Fredholm} as they are extensions of operators \eqref{Fredholm-B-positive} and \eqref{Fredholm-F-positive} considered for $s=\alpha+2l$ and because the relevant continuous embeddings of the source and target spaces are dense. The conclusion is also true for the Fredholm operator from Theorem~\ref{th-Fredholm-Nikolskii} as this operator is a restriction of the Fredholm operator
\begin{equation*}
(A,B):B^{s-\varepsilon}_{p,p}(A,B^{\alpha-\varepsilon}_{p,p},\Omega)\to B^{\alpha-\varepsilon}_{p,p}(\Omega)\times
\prod_{j=1}^{l}B^{s-\varepsilon-m_j-1/p}_{p,p}(\Gamma)
\end{equation*}
from Theorem~\ref{th-Fredholm} whenever $\varepsilon>0$ and  $\alpha-\varepsilon>m-2l+\pi(p,n)$ and because of the above-mentioned density of the relevant embeddings (see Theorem~\ref{B-F-sep-dense}). Thus, a vector $(f,g)$ belongs to the range of the Fredholm operator \eqref{Fredholm-B} (resp.,  \eqref{Fredholm-F}) if and only if $(f,g)$ belongs to the target space of this operator and satisfies $\omega(f,g)=0$ whenever $\omega\in W$.
\end{remark}

\begin{remark}\label{rem-Fredholm-Y}
Let $0<p<\infty$ and $0<q\leq\infty$ with $q\neq\infty$ for Triebel--Lizorkin spaces, and suppose that $s$ and $\alpha$ satisfy \eqref{Fredholm-parameters}. The Fredholm property keeps for operators \eqref{Fredholm-B} and \eqref{Fredholm-F} if we replace $B^{\alpha}_{p,q}(\Omega)$ and $F^{\alpha}_{p,q}(\Omega)$ with any quasi-Banach distribution space $Y(\Omega)$ continuously embedded in $B^{\alpha}_{p,q}(\Omega)$ or $F^{\alpha}_{p,q}(\Omega)$. Namely,
we obtain the bounded linear operators
\begin{equation}\label{Fredholm-B-Y}
(A,B):B^{s}_{p,q}(A,Y,\Omega)\to Y(\Omega)\times
\prod_{j=1}^{l}B^{s-m_j-1/p}_{p,q}(\Gamma)
\end{equation}
and
\begin{equation}\label{Fredholm-F-Y}
(A,B):F^{s}_{p,q}(A,Y,\Omega)\to Y(\Omega)\times
\prod_{j=1}^{l}B^{s-m_j-1/p}_{p,p}(\Gamma).
\end{equation}
They are Fredholm because their kernels coincide with $N$ and since their co-kernels are embedded in the finite-dimensional co-kernels of operators \eqref{Fredholm-B} and \eqref{Fredholm-F}, resp. The last property results from the evident fact that the range of operator \eqref{Fredholm-B-Y} (or \eqref{Fredholm-F-Y}) is the intersection of the following two sets: the target space of this operator and the range of  \eqref{Fredholm-B} (or \eqref{Fredholm-F}). Note that these embeddings may be strict at least in the $Y(\Omega)=\{0\}$ case, which follows from  \cite[Theorem~13.6.15]{Hermander05v2}. The Fredholm operators \eqref{Fredholm-B-Y} and \eqref{Fredholm-F-Y} in this case allow studying boundary properties of solutions to homogeneous elliptic equations (cf. \cite[Section~7]{AnopDenkMurach21CPAA1}, where the study was performed by means of inner product distribution spaces). If $C^{\infty}(\overline{\Omega})\subset Y(\Omega)$, then the set $M$ from Remark~\ref{rem-complement-space} is a complement of the range of operator \eqref{Fredholm-B-Y} (resp., \eqref{Fredholm-F-Y}) in its target space. In this case, the index of this operator equals $\varkappa$. Of course, this remark remains true in the case where $s>m+\pi(p,n)$ and when $Y(\Omega)$ is an arbitrary quasi-Banach space continuously embedded in $B^{s-2l}_{p,q}(\Omega)$ or $F^{s-2l}_{p,q}(\Omega)$.
\end{remark}

Note that the target space of the Fredholm bounded operator \eqref{Fredholm-F-Y} does not depend on $q$. This suggests the following result:

\begin{theorem}\label{source-space-F}
Let $s\in\mathbb{R}$ and $0<p<\infty$. Suppose that a quasi-Banach space $Y(\Omega)$ is continuously embedded in $F^{\alpha}_{p,q}(\Omega)$ for certain numbers
\begin{equation*}
\alpha>\max\{-l-1+\pi(p,n),s-2l\}
\end{equation*}
and $q\in(0,\infty)$. Then the space $F^{s}_{p,q}(A,Y,\Omega)$ does not depend on $q\in(0,\infty)$ up to equivalence of norms.
\end{theorem}

\begin{proof}
Given $0<q<q_1<\infty$, we arrive at the continuous embedding
$F^{s}_{p,q}(A,Y,\Omega)\hookrightarrow F^{s}_{p,q_{1}}(A,Y,\Omega)$ due to \cite[Proposition 3.2.4(i)]{Triebel83}. It remains to prove the inverse embedding. Consider the elliptic boundary-value problem \eqref{el-PDE}, \eqref{bound-cond} in the case where $B_{j}=\partial_{\nu}^{j-1}$ for each $j\in\{1,\ldots,l\}$ (as above, $\partial_{\nu}$ is the operator of differentiation along the inward normal to~$\Gamma$); hence,  $m=l-1$, and $\alpha$ satisfies \eqref{Fredholm-parameters}. Assume first that $Y(\Omega)=F^{\alpha}_{p,q}(\Omega)$, and note that Remark~\ref{rem-Fredholm-Y} is applicable to such $Y(\Omega)$ whatever $s\in\mathbb{R}$. Choosing a distribution $u\in F^{s}_{p,q_{1}}(A,F^{\alpha}_{p,q},\Omega)$ arbitrarily, we conclude that the vector $(f,g):=(A,B)u$ belongs to the range of the Fredholm operator \eqref{Fredholm-F-Y}, for otherwise $(f,g)=(f',g')+(f'',g'')$ for a certain vector $(f',g')$ from this range and a certain nonzero vector $(f'',g'')\in M$, which would contradict the equality $(f,g)=(A,B)u$. Hence, there exists a distribution $v\in F^{s}_{p,q}(A,F^{\alpha}_{p,q},\Omega)$ such that $(A,B)v=(f,g)$. Then $(A,B)(u-v)=0$, with $u-v\in F^{s}_{p,q_{1}}(A,F^{\alpha}_{p,q},\Omega)$. This means the inclusion $w:=u-v\in N$, as mentioned in Remark~\ref{rem-Fredholm-Y}. Thus, $u=v+w\in F^{s}_{p,q}(A,F^{\alpha}_{p,q},\Omega)$, and we have proved the inverse embedding $F^{s}_{p,q_{1}}(A,F^{\alpha}_{p,q},\Omega)\subset F^{s}_{p,q}(A,F^{\alpha}_{p,q},\Omega)$. It evidently implies the  embedding $F^{s}_{p,q_{1}}(A,Y,\Omega)\subset F^{s}_{p,q}(A,Y,\Omega)$ for an arbitrary quasi-Banach space $Y(\Omega)\subset F^{\alpha}_{p,q}(\Omega)$.
\end{proof}

Thus, if a quasi-Banach space $Y(\Omega)$ satisfies the hypotheses of Theorem~\ref{source-space-F}, then operator \eqref{Fredholm-F-Y} does not depend on $q$ and hence coincides with operator \eqref{Fredholm-B-Y} considered in the $p=q$ case.

According to this theorem, the space $\{u\in F^{s}_{p,q}(\Omega):Au=0\;\mbox{in}\;\Omega\}$ does not depend on $q\in(0,\infty)$. This property is proved in \cite[Theorem~1.6]{KaltonMayborodaMitrea07} under the additional assumption that $A$ is a constant-coefficient homogeneous PDO but in the more general case where $\Omega$ is a Lipschitz bounded domain.

A special case of Theorem~\ref{source-space-F}, where
$m\leq2l-1$, $p,q\in(1,\infty)$, and $\alpha>\max\{-1+1/p,s-2l\}$, is contained in \cite[Theorem~6]{ChepurukhinaMurach21Dop6} (given without a proof).

Ending this section, we give an application of Theorem~\ref{th-Fredholm-Nikolskii} to certain elliptic problems with rough boundary data induced by a Gaussian white noise. We previously recall some relevant notions. Let $(\Theta,\mathcal A,\mathbb P)$ be a probability space.

A generalized random variable $\xi:\Theta\to\mathcal{D}'(\Gamma)$ is said to be a Gaussian white noise on $\Gamma$ if $\xi$ satisfies the following two conditions:
\begin{itemize}
\item[($\star$)] the numerical random variable $\xi(v):\Theta\to\mathbb{C}$ is Gaussian for every function $v\in C^{\infty}(\Gamma)$;
\item[($\star\star$)] there exists a number $C>0$ such that
\begin{equation*}
\mathbb{E}[\xi(v_1)\overline{\xi(v_2)}]=C(v_1,v_2)_{\Gamma}
\quad\mbox{for arbitrary}\quad v_1,v_2\in C^{\infty}(\Gamma)
\end{equation*}
\end{itemize}
(see, e.g., \cite[Section~3]{Veraar11}). Here, as usual,  $\mathbb{E}$ is the expectation with respect to $\mathbb P$, and $\xi(v)$ denotes the value of $\xi$ at the test function $v\in C^{\infty}(\Gamma)$.

The regularity of white noise on the scale of Besov spaces (specifically, Nikolskii spaces) has recently been studied (see, e.g., \cite{FageotFallahUnser17, Veraar11}). For Gaussian white noise $\xi$ on the torus $\Gamma=\mathbb{T}^d$ of dimension $d\geq1$, it was shown  \cite[Theorem~3.4~(1)]{Veraar11} that
\begin{equation}\label{white-noise-in-B}
\mathbb{P}(\xi\in B^{-d/2}_{p,\infty}(\mathbb{T}^d))=1
\quad\mbox{for each}\;\;p\in[1,\infty).
\end{equation}
The greater $p$ is in this formula, the sharper result will be. We cannot replace the Nikolskii space in \eqref{white-noise-in-B} with a narrower Besov space $B^{-d/2}_{p,q}(\mathbb{T}^d)$, where  $p,q\in[1,\infty)$ or $p=q=\infty$, because $\mathbb{P}\{\xi\in B^{-d/2}_{p,q}(\mathbb{T}^d)\}=0$ for this  space \cite[Theorem~3.4 (4),(6)]{Veraar11}. Hence, $\mathbb{P}\{\xi\in B^{\sigma}_{p,\infty}(\mathbb{T}^d)\}=0$ whenever $\sigma>-d/2$. Thus, the values $\sigma=-d/2$ and $q=\infty$ are exact in formula \eqref{white-noise-in-B}, with the condition $p<\infty$ being essential. It is natural to assume that this formula will remain correct if we replace $\mathbb{T}^d$ with an arbitrary closed compact $C^{\infty}$-manifold $\Gamma$ of dimension $d\geq1$, but this has probably not been proved yet.

Assume that $\Omega$ is bounded by the torus $\mathbb{T}^{d}$, with $n=d+1\geq2$, and consider the Dirichlet boundary-value problem for the Poisson equation
\begin{equation}\label{Poisson-Dirichlet-noise}
\Delta u=f\quad\mbox{in}\;\,\Omega,\qquad u=\xi\quad\mbox{on}\;\,\mathbb{T}^{d}.
\end{equation}
Here, $f\in B^{\alpha}_{p,\infty}(\Omega)$ for certain $p\in[1,\infty)$ and $\alpha>-2+1/p$, and $\xi$ is a Gaussian white noise on $\mathbb{T}^{d}$. This problem is elliptic in $\Omega$, with $N=\{0\}$ and $M=\{0\}$ for it.

\begin{theorem}\label{th-white-noise}
For $\mathbb{P}$-almost all $\omega\in\Theta$, there exists a unique solution $u(\omega,\cdot)$ to \eqref{Poisson-Dirichlet-noise} that belongs to $B^{1/p-d/2}_{p,\infty}(\Omega)$. This solution satisfies the estimate
\begin{equation}\label{bound-white-noise}
\|u(\omega,\cdot),B^{1/p-d/2}_{p,\infty}(\Omega)\|\leq c\bigl(\|f,B^{\alpha}_{p,\infty}(\Omega)\|+
\|\xi(\omega),B^{-d/2}_{p,\infty}(\mathbb{T}^{d})\|\bigr),
\end{equation}
with the number $c=c(d,p,\alpha)>0$ being independent of $f$, $\xi$, and $\omega$.
\end{theorem}

This is a consequence of Theorem~\ref{th-Fredholm-Nikolskii} and property \eqref{white-noise-in-B}. Here, of course, $\|\cdot,X\|$ denotes the quasi-norm in a space $X$.

This result is formulated in \cite[Theorem~4]{MurachChepurukhina21Dop3} under the stronger assumption that $f\in L_{p}(\Omega)$. It follows from the above that the result is exact with respect to the smoothness index of the Nikolskii space over $\Omega$. The use of the Besov spaces $B^{\sigma}_{p,q}(\Omega)$ with $q<\infty$ does not allow us to attain the limit value $\sigma=1/p-d/2$.
We only assert that $\mathbb{P}(u\in B^{\sigma}_{p,q}(\Omega))=1$ whenever $\sigma<1/p-d/2$.

\section{Historical comments on some results}\label{sec5}

Theorem~\ref{th-Fredholm} has a background in classical results by Lions and Magenes on the solvability of elliptic boundary-value problems in Sobolev and Besov normed spaces of low order \cite{LionsMagenes61II, LionsMagenes61III, LionsMagenes62V, LionsMagenes63VI}. Their paper \cite[Theorems 7.1, 8.1, and 12.1]{LionsMagenes61II} deals with inner product Sobolev spaces (the $p=q=2$ case), studies the Dirichlet, Neumann, and oblique derivative elliptic problems, and proves versions of Theorem~\ref{th-Fredholm} for $s=0$ and $\alpha=-l$ (for the Dirichlet problem) or $\alpha=0$ (for other problems under study). The intermediate case where $0<s<l$ or $0<s<2l$, resp., is investigated by the quadratic interpolation between Hilbert spaces, however the domain of the problem operator is not identified with $H^{s}_{2}(A,H^{\alpha}_{2},\Omega)$ \cite[Theorems 10.1, 11.1, and 12.2]{LionsMagenes61II}. The next paper \cite[Theorem~12.1]{LionsMagenes61III} proves a version of Theorem~\ref{th-Fredholm} for the Dirichlet elliptic problem and parameters $0\leq s\leq2l$ and $\alpha=-l$ in the following two cases:
\begin{itemize}
  \item[($\ast$)] for operator \eqref{Fredholm-B} under assumption that $1<p=q<\infty$ and $s\not\in\mathbb{Z}$, with  $s-1/p\not\in\mathbb{Z}$ if $p\neq2$;
  \item[($\ast\ast$)] for operator \eqref{Fredholm-F} under assumption that $1<p<\infty$, $q=2$, and $s\in\mathbb{Z}$
\end{itemize}
(these operators are isomorphisms for this problem). The other papers \cite{LionsMagenes62V, LionsMagenes63VI} studies the general elliptic problem \eqref{el-PDE}, \eqref{bound-cond} in the case where the system $B$ of boundary PDOs is normal (the so-called regular elliptic problem). Versions of Theorem~\ref{th-Fredholm} are proved for such a problem and $\alpha=0$ in the above cases provided that $0\leq s\leq2l$, $N=\{0\}$, and $\varkappa=0$ \cite[Theorem~5.4]{LionsMagenes62V} or that $s\leq0$ \cite[Theorem~8.1]{LionsMagenes63VI}. Thus, Theorems \ref{th-Fredholm} and \ref{th-Fredholm-Nikolskii} are considered to be of the Lions--Magenes type.

Regarding the above-mentioned results, we note Magenes' survey \cite[Chapter~II]{Magenes65}, which gives applications of interpolation spaces to the Dirichlet elliptic problem, specifically in Sobolev and Besov spaces of low order. Versions of Theorem~\ref{th-Fredholm} are contained in \cite[Theorems 6.14, 6.16, 6.19, 6.20, and 6.25]{Magenes65} for such a problem provided that $s\leq2l$ and $\alpha=0$ and that $1<p=q<\infty$ for operator \eqref{Fredholm-B} or $1<p<\infty$ and $q=2$ for operator \eqref{Fredholm-F}, with $s-1/p\not\in\mathbb{Z}$ for both operators. We also mention Geymonat's paper \cite[Theorem 4.1]{Geymonat62}, which improves \cite[Theorem~12.1]{LionsMagenes61III} by using a broader source space for the operator $(A,B)$ in the Dirichlet problem case.

The indicated results are obtained by the transposition (i.e. the transition to the adjoint of the operator $(A,B)$ acting between distribution spaces of sufficiently large order) and then by the interpolation of relevant operators. The fundamental monograph by Lions and Magenes \cite[Chapter~2, Sections 6 and~7]{LionsMagenes72} expounds this method in the case of inner product Sobolev spaces, with weighted spaces of negative order being used for the right-hand sides of the elliptic equation \eqref{el-PDE}.

Another approach is based on theorems on solvability of elliptic boundary-value problems in Roitberg--Sobolev spaces \cite{Roitberg64, Roitberg65, Roitberg96}. It is proposed by Roitberg \cite[Theorem~4]{Roitberg68} and is developed in \cite[Theorem~4]{KostarchukRoitberg73}, \cite{Murach09MFAT2} for Sobolev spaces of low order (see also \cite[Section~7.9~c]{Agranovich94} and \cite[Section~4.4]{MikhailetsMurach14}). This approach allows obtaining a version of Theorem~\ref{th-Fredholm} for operator \eqref{Fredholm-F} under assumption that $m\leq2l-1$, $1<p<\infty$, $q=2$, $s<2l$, and $\alpha>-1+1/p$. It can be extended to Besov and Triebel--Lizorkin spaces whose order and exponents satisfy $s\geq0$, $1<p<\infty$ and $1<q<\infty$, as stated in \cite[Theorem~2 and Remark~1]{Murach94Dop12} without proofs, the assumption $\alpha>\max\{s-2l,m-2l+1/p,-1+1/p\}$ ensuring the Fredholm property for operators \eqref{Fredholm-B} and \eqref{Fredholm-F}.

If $m\leq2l-1$, $1<p<\infty$, $1<q<\infty$, $s<2l-1+1/p$, and $\alpha>-1+1/p$, then Theorem~\ref{th-Fredholm} is contained in \cite[Theorems 1 and~2]{ChepurukhinaMurach21Dop6} provided that $s\neq-k+1/p$ whenever $1\leq k\in\mathbb{Z}$. The paper \cite[Theorems 3 and~4]{ChepurukhinaMurach21Dop6} also involves weighted Besov--Triebel--Lizorkin spaces of negative order as the spaces of right-hand sides of the elliptic equation (in spirit of \cite[Chapter~2, Sections 6 and~7]{LionsMagenes72} and \cite[Section~4.4.4]{MikhailetsMurach14}). The indicated results are given without proofs in~\cite{ChepurukhinaMurach21Dop6}.

Thus, Theorem \ref{th-Fredholm} is new at least in the case where $0<p\leq1$ or/and $0<q\leq1$ and hence involves quasi-Banach spaces of low order in similar results for the first time. If $m\geq2l$, then this theorem is new for all $p,q\in(0,\infty)$ in the case where $s<0$ (the $0\leq s\leq m+1/p$ case is covered by \cite[Theorem~2 and Remark~1]{Murach94Dop12} provided that $p,q\in(1,\infty)$) excepting the $p=q=2$ case involved in \cite[Theorem~3]{KasirenkoMikhailetsMurach19}). If $m\leq2l-2$, then Theorem~\ref{th-Fredholm} is new even when $p,q\in(1,\infty)$ (specifically, in the $p=q=2$ case of inner product Sobolev spaces) because it allows $\alpha>m-2l+1/p$ instead of the stronger assumption $\alpha>-1+1/p$ made in \cite[Theorems 1 and~2]{ChepurukhinaMurach21Dop6}. (A similar diminution of the value of $\alpha$ was earlier attained only for the Dirichlet elliptic problem provided that $0\leq s\leq l$; see, e.g., the above-mentioned paper \cite[Theorem~12.1]{LionsMagenes61III}, when the value $\alpha=-l$ was taken instead of $\alpha=0$ for other boundary conditions.)

Theorem~\ref{th-Fredholm-Nikolskii} is new excepting the case where $m\leq2l-1$, $1<p<\infty$, and $\alpha=0$. In this case,
a version of this theorem was formulated in \cite[Theorem~1]{MurachChepurukhina21Dop3} (without a proof, only its method was indicated).

Note that versions of Theorem \ref{th-Fredholm} were proved in \cite[Theorem~4.2]{AnopDenkMurach21CPAA1}, \cite[Theorem~1]{AnopKasirenkoMurach18UMJ3}, and \cite[Theorem~3.1.1]{MikhailetsMurach14} for inner product Sobolev spaces of generalized low smoothness (see also surveys \cite[Theorem~6.2]{MikhailetsMurach09OperatorTheory191} and \cite[Theorem~12.7]{MikhailetsMurach12BJMA2}). These versions deal with the $p=q=2$ case and are contained in Theorem~\ref{th-Fredholm} for  Sobolev spaces. The important case of homogeneous elliptic equation \eqref{el-PDE} was studied in \cite{Anop19Dop2, AnopMurach18Dop3, MikhailetsMurach06UMJ11, MurachChepurukhina20Dop8}, \cite[Section~7]{AnopDenkMurach21CPAA1}, and \cite[Section~3.3]{MikhailetsMurach14} for such spaces of generalized smoothness and various boundary conditions
(see also \cite[Section~4.2]{MikhailetsMurach09OperatorTheory191} and \cite[Section~10.1]{MikhailetsMurach12BJMA2}).

It is worthwhile to compare Theorems \ref{th-Fredholm} and \ref{th-Fredholm-Nikolskii} with Hummel's recent result \cite[Theorem~6.3]{Hummel21JEE} on conditions for unique solvability of some constant-coefficient parameter-elliptic problems in certain normed anisotropic distribution spaces over the half-space. These spaces are built on the base of Besov or Triebel--Lizorkin spaces of an arbitrary real order, and the right-hand side of the elliptic equation is assumed to have a nonnegative integer-valued Sobolev regularity in the normal direction. Unlike Theorems \ref{th-Fredholm} and \ref{th-Fredholm-Nikolskii}, this result does not allow deriving the maximal regularity of solutions from the obtained conditions \cite[p.~1949]{Hummel21JEE} and applies only to Banach spaces.
The conditions for the maximal regularity of solutions to parameter-elliptic problems in $\Omega$ is obtained in the newest paper \cite[Theorem~4.9]{DenkPlossRauSeiler23} by Denk, Plo{\ss}, Rau, and  Seiler for the solution space $H^{s+\sigma}_{p}(A,H^{\alpha}_{p},\Omega)$ provided that $1<p<\infty$, $s>m+1/p$, $-1<\sigma\leq0$, $-1+1/p<\alpha\leq s+\sigma$, and $\alpha>s-2m$, the Euclidean domain $\Omega$ being assumed to be bounded or exterior with sufficiently smooth compact boundary.

Interpolation formulas \eqref{R-interp-special-formula} and \eqref{C-interp-formula} are known in some very special cases involving only normed Sobolev and Besov spaces. Formula~\eqref{R-interp-special-formula} is proved in \cite[Propositions 5.5 and~5.6]{LionsMagenes62V} for the case where $E=V=F$, $1<p=\beta=q<\infty$, $q_0=q_1=\gamma=2$, $\lambda\in\{-l,0\}$, $s_0\in\mathbb{Z}$, $1\leq s_0\leq2l+\lambda$, $s_1=s_0-1$, and $1-\theta\neq1/p$. Formula \eqref{C-interp-formula} is proved in
\cite[Propositions 5.3 and~5.4]{LionsMagenes62V} for the case where $E=F$, $1<p_0=p_1=\beta_0=\beta_1<\infty$, $q_0=q_1=\gamma_0=\gamma_1=2$, $\alpha_0=\alpha_1\in\{-l,0\}$, $s_0=2l+\alpha_0$, $s_1=0$, and $\theta s_0\in\mathbb{Z}$. These results were deduced from a theorem on isomorphisms set by the Dirichlet elliptic boundary-value problem on relevant pairs of Banach distribution spaces (see \cite[Theorem~12.1]{LionsMagenes61III}). In the case of Hilbert spaces (where all subscripts equal $2$), formula \eqref{C-interp-formula} is proved in \cite[Theorem~2]{KasirenkoMikhailetsMurach19} on the base of Theorem~\ref{th-basic}, established in Lions and Magenes' monograph \cite[Chapter~1, Section~14.3]{LionsMagenes72} for the complex method of interpolation. This monograph \cite[Chapter~2, Theorem~7.2]{LionsMagenes72} gives an example of general interpolation formula \eqref{interp-XY-XY} in the case where $X_0=H^{2l}_{2}(\Omega)$, $Y_0=X_1=H^{0}_{2}(\Omega)$, and $Y_1=\{\varrho^{2l}f:f\in H^{-2l}_{2}(\Omega)\}$, with the weight function $\varrho\in C^{\infty}(\overline{\Omega})$ satisfying $\varrho>0$ in $\Omega$ and
$\varrho(x)=\mathrm{dist}(x,\Gamma)$ in a neighbourhood of $\Gamma$. This formula is used to prove a classical result \cite[Chapter~2, Theorem~7.4]{LionsMagenes72} on isomorphisms induced by a regular elliptic boundary-value problem and involving weighted inner product Sobolev spaces of negative order. These specific interpolation formulas were obtained in the above-mentioned papers for case when $\partial\Omega$ is of class~$C^{\infty}$.

Of interest is a margin version of interpolation formula \eqref{interp-XY-XY} where $Y_0=Y_1=\{0\}$. This version is not covered by Theorem~\ref{main-interp-th} (because the null space is not admissible) and deals with spaces formed by solutions to the homogeneous elliptic equation $Au=0$ in $\Omega$. The version is proved by Kalton, Mayboroda, and Mitrea \cite[Theorem~1.5]{KaltonMayborodaMitrea07} for the real and complex interpolation methods in the case where all coefficients of $A$ are constant. The case of variable coefficients can be treated with the help of the Fredholm operators \eqref{Fredholm-B-Y} and \eqref{Fredholm-F-Y} considered for the Dirichlet problem and $Y(\Omega)=\{0\}$ on the assumption that $\Gamma$ and all coefficients of $A$ are of class $C^{\infty}$. This was done by Anop, Denk, and Murach \cite[Theorem~7.8]{AnopDenkMurach21CPAA1} for the method of the quadratic interpolation with function parameter between Hilbert spaces.

A part of Theorem~\ref{B-F-sep-dense} concerning the density of $C^{\infty}(\overline{\Omega})$ in the space $E^{s}_{p,q}(A,G^{\alpha}_{\beta,\gamma},\Omega)$, where $E,G\in\{B,F\}$, is crucial for theorems of the Lions--Magenes type. Indeed, the bounded operators \eqref{Fredholm-B} and \eqref{Fredholm-F} are defined to be  extensions by continuity of the mapping $u\mapsto(Au,Bu)$ where $u$ ranges over the set $C^{\infty}(\overline{\Omega})$. The density of this set in the Hilbert space $H^{0}_{2}(A,H^{0}_{2},\Omega)$ (see \cite[Remark]{Hermander58}) plays an important role in the spectral theory of elliptic operators. The density of $C^{\infty}(\overline{\Omega})$ in the space $H^{0}_{p}(A,H^{-l}_{p},\Omega)$ used for the Dirichlet boundary-value problem was proved in \cite[Lemma~2.2]{LionsMagenes61II} for $p=2$ and in
\cite[Proposition~9.2]{LionsMagenes61III} for every $p\in(1,\infty)$. This density in $H^{s}_{p}(A,H^{0}_{p},\Omega)$ was established in \cite[Proposition~3.1]{LionsMagenes62V} if $0\leq s\leq 2l$ and in
\cite[Lemma~7.1]{LionsMagenes63VI} if $s<0$ and $s-1/p\not\in\mathbb{Z}$, provided that $p\in(1,\infty)$. Theorem~\ref{B-F-sep-dense} in the $p=q=2$ case of Hilbert spaces was proved in \cite[Theorem~1, (i) and (iii)]{KasirenkoMikhailetsMurach19}. These results were obtained in the case where $\partial\Omega$ is of class $C^{\infty}$ (excepting those on $H^{0}_{2}(A,H^{0}_{2},\Omega)$).

\section{Appendix A}\label{sec6}

Let us discuss the structure of the spaces $M$ and $W$ indicated in Remarks \ref{rem-complement-space}, \ref{adjoint-kernel}, and  \ref{adjoint-kernel-negative}.

Suppose first that $m\leq 2l-1$. Then the kernel of a certain elliptic boundary-value problem can be taken as the space $M$ satisfying the assumption made in Remark~\ref{rem-complement-space} and can be interpreted as the space $W$. This problem is adjoint to problem \eqref{el-PDE}, \eqref{bound-cond} with respect to the following Green formula (see, e.g., \cite[Theorem~3.1.1]{KozlovMazyaRossmann97}):
\begin{equation*}
(Au,w)_{\Omega}+\sum_{j=1}^{l}(B_{j}\,u,h_{j})_{\Gamma}=
(u,A^{+}w)_{\Omega}+\sum_{k=1}^{2l}\biggl(D_{\nu}^{k-1}u,K_{k}\,w+
\sum_{j=1}^{l}Q_{j,k}^{+}\,h_{j}\biggr)_{\Gamma},
\end{equation*}
where the functions $u,w\in C^{\infty}(\overline{\Omega})$ and  $h_{1},\ldots,h_{l}\in C^{\infty}(\Gamma)$ are arbitrary, whereas $(\cdot,\cdot)_{\Omega}$ and $(\cdot,\cdot)_{\Gamma}$ denote the inner products in $L_{2}(\Omega)$ and $L_{2}(\Gamma)$, resp., and also denote their extensions by continuity. (As usual, $L_{2}(\Omega)$ and $L_{2}(\Gamma)$ stand for  the Hilbert spaces  of functions square integrable over $\Omega$ and $\Gamma$, resp.) Here, $A^{+}$ is the formally adjoint PDO to $A$ with respect to $(\cdot,\cdot)_{\Omega}$.
All boundary linear PDOs $K_{k}:=K_{k}(x,D)$ are uniquely determined by the following formula:
\begin{equation*}
(Au,w)_{\Omega}=(u,A^{+}w)_{\Omega}+
\sum_{k=1}^{2l}(D_{\nu}^{k-1}u,K_{k}\,w)_{\Gamma}
\end{equation*}
whenever $u,w\in C^{\infty}(\overline{\Omega})$, with $\mathrm{ord}\,K_{k}\leq2l-k$ and $D_{\nu}:=i\partial_{\nu}$ (recall that $\partial_{\nu}$ is the operator of differentiation along the inward normal to $\Gamma$). Finally, each $Q_{j,k}^{+}$ is the formally adjoint tangent PDO to $Q_{j,k}$ with respect to $(\cdot,\cdot)_{\Gamma}$, whereas the tangent linear PDO $Q_{j,k}$ is taken from the representation of the boundary PDO $B_{j}$ in the form
\begin{equation}\label{bound-PDO-representation}
B_{j}(x,D)=
\sum_{k=1}^{2l}Q_{j,k}(x,D_{\tau})D_{\nu}^{k-1}
\end{equation}
(of course, $Q_{j,k}=0=Q_{j,k}^{+}$ if $k\geq m_{j}+2$). The coefficients of the indicated PDOs are infinitely smooth. Thus, the homogeneous adjoint problem takes the form
\begin{gather}\label{adjoint-PDE}
A^{+}w=0\quad\mbox{in}\quad\Omega,\\
K_{k}\,w+\sum_{j=1}^{l}Q_{j,k}^{+}h_{j}=0\quad
\mbox{on}\quad\Gamma,\quad k=1,...,2l.
\label{adjoint-bound-cond}
\end{gather}
This problem contains $l$ additional unknown functions $h_{1},\ldots,h_{l}$ on the boundary $\Gamma$ and is elliptic  \cite[Theorem~3.1.2]{KozlovMazyaRossmann97}. Remark that such a Green formula and adjoint boundary-value problem were introduced by {\L}awruk \cite[Section~4]{Lawruk63a}.

The space
\begin{equation}\label{M-description}
M:=\{(w,h_{1},\ldots,h_{l})\in C^{\infty}(\overline{\Omega})\times (C^{\infty}(\Gamma))^{l}:\mbox{\eqref{adjoint-PDE} and \eqref{adjoint-bound-cond} hold true}\}
\end{equation}
satisfies the assumption made in Remark~\ref{rem-complement-space} and
is hence relevant in Theorems \ref{th-Fredholm} and~\ref{th-Fredholm-Nikolskii}. Moreover, $W=M$ in the natural sense that $W$ consists of the continuous linear functionals $(f,w)_{\Omega}+(g_{1},h_{1})_{\Gamma}+\cdots+
(g_{l},h_{l})_{\Gamma}$ induced by arbitrary vectors $(w,h_{1},\ldots,h_{l})\in M$ and defined on the target space of any operator among \eqref{Fredholm-B-positive}, \eqref{Fredholm-F-positive},
\eqref{Fredholm-B}, and \eqref{Fredholm-F}, with $(f,g_{1},\ldots,g_{l})$ running through this target space. (Recall that the $q=\infty$ case is admissible for \eqref{Fredholm-B} due to Theorem~\ref{th-Fredholm-Nikolskii}.) This follows directly from the fact that $M$ is the kernel of the adjoint operator to \eqref{Fredholm-F-positive} in the case where $s=2l$ and $p=q=2$ (see, e.g., \cite[Theorem~3.4.2]{KozlovMazyaRossmann97}). Hence, the vector $(f,g)$ belongs to the range of any Fredholm operator among \eqref{Fredholm-B-positive}, \eqref{Fredholm-F-positive}, \eqref{Fredholm-B}, and \eqref{Fredholm-F} if and only if $(f,g)$ belongs to the target space of this operator and satisfies the following orthogonality condition:
\begin{equation*}
(f,w)_{\Omega}+\sum_{j=1}^{l}(g_{j},h_{j})_{\Gamma}=0
\quad\mbox{for every}\quad (w,h_{1},\ldots,h_{l})\in M.
\end{equation*}

In this regard, we note the following: if $(w,h_{1},\ldots,h_{l})\in M$, then $D_{\nu}^{k-1}w=0$ on $\Gamma$ whenever $1\leq k<2l-m$ (cf. \cite[Corollary~1.4]{Johnsen96}). Indeed, since the space $M=W$ lies, specifically, in the dual of
\begin{equation*}
H^{s-2l}_{2}(\Omega)\times\prod_{j=1}^{l}H^{s-m_j-1/2}_{2}(\Gamma)
\end{equation*}
for $s=m+1$, the first component of any vector $(w,h_{1},\ldots,h_{l})\in M$ belongs to the closure of $C^{\infty}_{0}(\Omega)$ in the space $H^{2l-m-1}_{2}(\Omega)$ \cite[Theorems 4.3.2/1(c) and 4.8.2(c)]{Triebel95}, which implies the required property of $w$. (As usual, $C^{\infty}_{0}(\Omega)$ is the space of all functions $u\in C^{\infty}(\overline{\Omega})$ such that $\mathrm{supp}\,u\subset\Omega$.) Note also that $(g_{j},h_{j})_{\Gamma}$ coincides with the value of the distribution $g_{j}$ at the test function $h_{j}$ (provided that we interpret distributions as antilinear functionals on a relevant space of test functions).

It is worthwhile to mention the frequently encountered case where the system $B$ of boundary PDOs is normal \cite[Chapter~2, Definition~1.4]{LionsMagenes72}. Then the boundary-value problem \eqref{el-PDE}, \eqref{bound-cond} also satisfies the classical Green formula
\begin{equation*}
(Au,w)_{\Omega} + \sum_{j=1}^{l}(B_{j}u,C_{l+j}w)_{\Gamma} =
(u,A^{+}w)_{\Omega} + \sum_{j=1}^{l}(B_{l+j}u,C_{j}w)_{\Gamma}
\end{equation*}
for arbitrary $u,w\in C^{\infty}(\overline{\Omega})$. Here,
$B_{l+j}$, $C_{j}$, and $C_{l+j}$ are certain boundary linear PDOs (with infinitely smooth coefficients on $\Gamma$) that form the Dirichlet systems $\{B_{1},\ldots,B_{2l}\}$ and $\{C_{1},\ldots,C_{2l}\}$ of order $2l$ on $\Gamma$; the second system is uniquely determined by the first one for any choice of $\{B_{l+1},\ldots,B_{2l}\}$ that complements $\{B_{1},\ldots,B_{l}\}$ to the the Dirichlet system (see, e.g., \cite[Section~3.1.4]{KozlovMazyaRossmann97} or \cite[Chapter~2, Theorem~2.1]{LionsMagenes72}). This formula leads to the homogeneous elliptic adjoint boundary-value problem
\begin{equation}\label{adjoint-problem-classical}
A^{+}w=0\quad\mbox{in}\quad\Omega,\qquad
C_{j}w=0\quad\mbox{on}\quad\Gamma,
\quad j=1,\ldots l.
\end{equation}
Then the right-hand side of \eqref{M-description} becomes
\begin{equation*}
M=\{(w,C_{l+1}w,\ldots,C_{2l}w):w\in
C^{\infty}(\overline{\Omega})\;\mbox{satisfies
\eqref{adjoint-problem-classical}}\}
\end{equation*}
\cite[Lemma~3.1.1]{KozlovMazyaRossmann97}. An important example of a normal system of boundary PDOs is given by    $B_{j}:=D_{\nu}^{\lambda+j-1}+\mbox{\{lower order terms\}}$, where $j=1,\ldots,l$; here the integer $\lambda\in\{0,\ldots,l\}$ is arbitrarily chosen. This system satisfies the Lopatinskii condition with respect to any properly elliptic PDO $A$ of order $2l$.

Suppose now that $m\geq 2l$. Then the following modified Green formula holds true for the boundary-value problem \eqref{el-PDE}, \eqref{bound-cond}:
\begin{align*}
&(Au,w)_{\Omega}+\sum_{j=1}^{m+1-2l}(D_{\nu}^{j-1}Au,w_{j})_{\Gamma}+
\sum_{j=1}^{l}(B_{j}\,u,h_{j})_{\Gamma}\\
&=(u,A^{+}w)_{\Omega}+
\sum_{k=1}^{m+1}\biggl(D_{\nu}^{k-1}u,K_{k}w+
\sum_{j=1}^{m+1-2l}R_{j,k}^{+}w_{j}+
\sum_{j=1}^{l}Q_{j,k}^{+}h_{j}\biggr)_{\Gamma},
\end{align*}
where the functions $u,w\in C^{\infty}(\overline{\Omega})$ and  $w_{1},\ldots,w_{m+1-2l},h_{1},\ldots,h_{l}\in C^{\infty}(\Gamma)$ are arbitrary (see, e.g., \cite[formula (4.1.10)]{KozlovMazyaRossmann97}). Here, all PDOs $A^{+}$, $K_{k}$, and $Q_{j,k}^{+}$ are the same as above provided that $k$ ranges from $1$ to $m+1$ in  \eqref{bound-PDO-representation}. Furthermore, each $R_{j,k}^{+}$ is the formally adjoint tangent PDO to $R_{j,k}$ with respect to $(\cdot,\cdot)_{\Gamma}$, whereas the tangent linear PDO $R_{j,k}$ is taken from the representation
\begin{equation*}
D_{\nu}^{j-1}A(x,D)=\sum_{k=1}^{m+1}R_{j,k}(x,D_{\tau})D_{\nu}^{k-1}
\end{equation*}
(of course, $R_{j,k}=0=R_{j,k}^{+}$ if $k\geq 2l+j+1$). The corresponding homogeneous adjoint problem consists of equation \eqref{adjoint-PDE} and the boundary conditions
\begin{equation}\label{adjoint-bound-cond-m}
K_{k}\,w+\sum_{j=1}^{m+1-2l}R_{j,k}^{+}w_{j}+
\sum_{j=1}^{l}Q_{j,k}^{+}h_{j}=0\quad
\mbox{on}\quad\Gamma,\quad k=1,...,m+1.
\end{equation}
This problem contains $m+1-l$ additional unknown functions $w_{j}$ and $h_{j}$ on the boundary $\Gamma$ and is elliptic \cite[Theorem~4.1.1]{KozlovMazyaRossmann97}. Since $m+1-l\geq l+1$, the kernel
\begin{align*}
N^{+}:=\{(w,w_{1},\ldots,w_{m+1-2l},h_{1},\ldots,h_{l})\in
C^{\infty}(\overline{\Omega})\times(C^{\infty}(\Gamma))^{m+1-l}\\
:\mbox{\eqref{adjoint-PDE} and \eqref{adjoint-bound-cond-m} are true}\}
\end{align*}
of this problem cannot be taken as $M$ in the $m\geq 2l$ case.

However, the space $W$ indicated in Remarks \ref{adjoint-kernel} and \ref{adjoint-kernel-negative} is described in terms of $N^{+}$ as follows: $W$ consists of the continuous linear functionals induced by arbitrary vectors
\begin{equation}\label{N+}
(w,w_{1},\ldots,w_{m+1-2l},h_{1},\ldots,h_{l})\in N^{+}
\end{equation}
according to the formula
\begin{equation}\label{Q-m}
\mathcal{L}(f,g):=
(f,w)_{\Omega}+\sum_{j=1}^{m+1-2l}(D_{\nu}^{j-1}f,w_{j})_{\Gamma}+
\sum_{j=1}^{l}(g_{j},h_{j})_{\Gamma},
\end{equation}
where $(f,g_{1},\ldots,g_{l})$ runs through the target space of the corresponding Fredholm operator. This results directly from the fact that such a description holds true for the kernel of the adjoint operator to \eqref{Fredholm-F-positive} if $s=m+1$ and $p=q=2$ (see, e.g., \cite[Theorem~4.1.5]{KozlovMazyaRossmann97}). The distributions  $D_{\nu}^{j-1}f$ on $\Gamma$ are well defined in \eqref{Q-m} by the trace theorem \cite[Theorem~3.3.3]{Triebel83} because $f\in B^{\alpha}_{p,q}(\Omega)$ or $f\in F^{\alpha}_{p,q}(\Omega)$ for certain  $\alpha>m-2l+\pi(p,n)$.

Thus, the vector $(f,g)$ belongs to the range of any Fredholm operator among \eqref{Fredholm-B-positive}, \eqref{Fredholm-F-positive},
\eqref{Fredholm-B}, and \eqref{Fredholm-F} in the $m\geq 2l$ case if and only if $(f,g)$ belongs to the target space of this operator and satisfies the following orthogonality condition:
\begin{gather*}
\mathcal{L}(f,g)=0\quad\mbox{whenever \eqref{N+} holds true}.
\end{gather*}
Apparently, the presence of $(D_{\nu}^{j-1}f,w_{j})_{\Gamma}$
in the orthogonality condition was first indicated in \cite[\S~4, formula~(76)]{VainbergGrushin67b} as a remark to \cite[Theorem~10.6.1]{Hermander63}.

Applying the space $W$ to Fredholm operators induced by the elliptic boundary-value problem, we use the following fact:

\begin{proposition}\label{prop-range}
Let $E$ and $G$ be quasi-Banach spaces, and let $T:E\to G$ be a bounded linear operator. Suppose that $G$ is dual reach and that the factor space $G/T(E)$ is finite-dimensional. Then
\begin{equation}\label{prop-range-f}
T(E)=\{y\in G:\omega(y)=0\;\,\mbox{for every}\;\,\omega\in\ker T'\},
\end{equation}
where $T':G'\to E'$ is the adjoint operator to $T$.
\end{proposition}

Apparently, this proposition is known. Just in case, we give its proof.

\begin{proof}
The inclusion $\subset$ in \eqref{prop-range-f} is trivial. Let us prove the inverse inclusion by contradiction. Since $\dim G/T(E)<\infty$, we conclude that $T(E)$ is a complemented closed subspace of $G$. Let $\mathcal{M}$ be a complement of $T(E)$ in $G$, and let $P$ be a projector of $G$ on $\mathcal{M}$ parallel to $T(E)$. This projector is a bounded operator on $G$ due to Closed Graph Theorem. Suppose the contrary, namely that a certain vector $y_0\in G\setminus T(E)$ satisfies $\omega(y_0)=0$ for every functional $\omega\in\ker T'$. It follows from the representation $y_0=Tx+z$ for certain vectors $x\in E$ and $z\in\mathcal{M}$ with $z\neq0$ that
\begin{equation}\label{omega(z)=0}
\omega(z)=0\quad\mbox{whenever}\quad \omega\in\ker T'.
\end{equation}
Since $G$ is dual reach, there exists a functional $\omega_1\in G'$ such that $\omega_1(z)=1$. Define a functional $\omega\in G'$ by the formula $\omega(y):=\omega_1(Py)$ for every $y\in G$. We have the inclusion $\omega\in\ker T'$ because
\begin{equation*}
(T'\omega)x=\omega(Tx)=\omega_1(PTx)=\omega_1(0)=0
\quad\mbox{whenever}\quad x\in E.
\end{equation*}
However, $\omega(z)=\omega_1(Pz)=\omega_1(z)=1$, which contradicts \eqref{omega(z)=0}. This proves the inclusion $\supset$ in \eqref{prop-range-f}.
\end{proof}

\section{Appendix B}\label{sec7}

We will prove Lemma~\ref{prop-interp-Fredholm} and Corollary~\ref{cor-interp-Fredholm}, which give sufficient conditions for interpolation of Fredholm bounded operators to yield a Fredholm operator again. Recall that we used these results to substantiate Theorems \ref{th-Fredholm} and \ref{th-Fredholm-Nikolskii}. The proof of Lemma~\ref{prop-interp-Fredholm} is rather technical, however we could not find it in the literature.

\begin{proof}[Proof of Lemma~$\ref{prop-interp-Fredholm}$.]
\emph{Step~$1$.}  The operator $T:X_0+X_1\to Y_0+Y_1$ is well posed by condition~(iii). Indeed, if a vector $u\in X_0+X_1$ is represented as $u=u_0+u_1$ and $u=v_0+v_1$ for certain $u_0,v_0\in X_0$ and $u_1,v_1\in X_1$, then $u_0-v_0=v_1-u_1\in X_0\cap X_1$, which implies that $T_0(u_0-v_0)=T_1(v_1-u_1)$ by (iii), i.e. $T_0u_0+T_1u_1=T_0v_0+T_1v_1:=Tu$. This operator induces an isomorphism
\begin{equation}\label{l42f1}
T:(X_0+X_1)/\mathcal{N}\leftrightarrow T_0(X_0)+T_1(X_1),
\end{equation}
we considering $T_0(X_0)$ and $T_1(X_1)$ as quasi-Banach spaces endowed with quasi-norms $\|\cdot\|_{Y_0}$ and $\|\cdot\|_{Y_1}$, resp. To prove this we only need to show that $\ker T=\mathcal{N}$ and apply Open Mapping Theorem.

Let $u\in X_0+X_1$ and $Tu=0$. Then $u=u_0+u_1$ for certain $u_0\in X_0$ and $u_1\in X_1$, and $T_0u_0+T_1u_1=0$. Hence,
\begin{equation*}
f:=T_0u_0=-T_1u_1\in T_0(X_0)\cap T_1(X_1)=T_0(X_0\cap X_1)=T_1(X_0\cap X_1)
\end{equation*}
by condition~(v). Therefore, $f=T_0v=T_1v$ for some $v \in X_0\cap X_1$.
Thus,  $T_0u_0=T_0v$ and $-T_1u_1=T_1v$, i.e. $T_0(u_0-v)=0$ and $T_1(v+u_1)=0$. Hence $u_0-v\in\mathcal{N}$ and $v+u_1\in\mathcal{N}$ by condition~(i), which gives $u_0+u_1=(u_0-v)+(v+u_1)\in\mathcal{N}$. We have proved the inclusion  $\ker T\subset\mathcal{N}$. The inverse is obvious.

\emph{Step~$2$.} Note that $\ker T_{\mathfrak{F}}=\ker T$ because $\mathcal{N}\subset X_0\cap X_1\subset\mathfrak{F}[X_0,X_1]$. Hence, $\ker T_{\mathfrak{F}}=\mathcal{N}$ according to Step~1, i.e. $T_{\mathfrak{F}}$ satisfies property~(a). To substantiate other properties of $T_{\mathfrak{F}}$, we consider the induced isomorphisms
\begin{equation*}
T_0:X_0/\mathcal{N}\leftrightarrow T_0(X_0)\quad\mbox{and}\quad
T_1:X_1/\mathcal{N}\leftrightarrow T_1(X_1)
\end{equation*}
and note that the pairs $[X_0/\mathcal{N},X_1/\mathcal{N}]$ and $[T_0(X_0),T_1(X_1)]$ of quasi-Banach spaces are interpolation pairs because of the continuous embeddings $X_0/\mathcal{N},X_1/\mathcal{N}\hookrightarrow(X_0+X_1)/\mathcal{N}$ and $T_0(X_0),T_1(X_1)\hookrightarrow Y_0+Y_1$. These isomorphisms are  restrictions of isomorphism \eqref{l42f1}; hence, interpolating them, we get another isomorphism
\begin{equation}\label{l42intoperator1}
T:\mathfrak{F}[X_0/\mathcal{N}, X_1/\mathcal{N}]\leftrightarrow
\mathfrak{F}[T_0(X_0), T_1(X_1)].
\end{equation}
We will describe the domain and range of \eqref{l42intoperator1} via the interpolation spaces $\mathfrak{F}[X_0,X_1]$ and $\mathfrak{F}[Y_0,Y_1]$, resp.

\emph{Step~$3$.} Let us show that the domain coincides with the factor space of $\mathfrak{F}[X_0,X_1]$ by $\mathcal{N}$. Let $\{w_{1},\ldots,w_{\lambda}\}$ be a basis of $\mathcal{N}$, with $\lambda:=\dim\mathcal{N}<\infty$. Since the quasi-Banach space $X_0+X_1$ is dual reach, there exist continuous linear functionals $\varphi_{1},\ldots,\varphi_{\lambda}$ on $X_0+X_1$ such that $\varphi_{j}(w_{k})=\delta_{j,k}$ (Kronecker's symbol) whenever $j,k\in\{1,\ldots,\lambda\}$ (see, e.g., \cite[Lemma~2.2]{GeislerRunst91}). Then the linear mapping
\begin{equation}\label{def-P}
P:u\mapsto\sum_{j=1}^{\lambda}\varphi_{j}(u)w_{j},
\quad\mbox{where}\quad u\in X_0+X_1,
\end{equation}
is a bounded projector of $X_0+X_1$ on $\mathcal{N}$. Moreover, the restriction of $P$ on each $X_j$, with $j\in\{0,1\}$, is a bounded projector of $X_j$ onto $\mathcal{N}$. This implies that
\begin{equation}\label{l42intequiv}
\mathfrak{F}[X_0/\mathcal{N},X_1/\mathcal{N}]=
(\mathfrak{F}[X_0,X_1])/\mathcal{N}
\end{equation}
up to equivalence of quasi-norms. This fact is contained in \cite[Theorem~1.17.2]{Triebel95} for Banach spaces. Apparently, this theorem remains valid for quasi-Banach spaces. Nevertheless, we prefer to substantiate  \eqref{l42intequiv} for the readers convenience.

To this end, we consider the linear mappings
\begin{equation*}
R:X_0+X_1\ni u\mapsto \{u+w:w\in \mathcal{N}\}\in (X_0+X_1)/\mathcal{N}
\end{equation*}
and
\begin{equation*}\label{l42Qmapping2}
Q:(X_0+X_1)/\mathcal{N} \ni\{u+w:w\in \mathcal{N}\}\mapsto u-Pu\in X_0+X_1.
\end{equation*}
We see that $Q$ is well defined and that $RQ=I$ is the identity mapping on $(X_0+X_1)/\mathcal{N}$.
Restrictions of $R$ and $Q$ set the following bounded operators:
\begin{gather}\label{l42Rrestroper1}
R:X_j\rightarrow X_j/\mathcal{N}\quad\mbox{whenever}\quad j\in\{0,1\},\\
R:\mathfrak{F}[X_0, X_1]\rightarrow (\mathfrak{F}[X_0, X_1])/\mathcal{N}, \label{l42Rrestroper2}\\
Q:X_j/\mathcal{N}\rightarrow X_j\quad\mbox{whenever}\quad j\in\{0,1\},
\label{l42Qrestroper1}\\
Q:(\mathfrak{F}[X_0, X_1])/\mathcal{N}\rightarrow \mathfrak{F}[X_0, X_1]. \label{l42Qrestroper2}
\end{gather}
Note that the boundedness of the last operator follows from the boundedness of the operator $P:\mathfrak{F}[X_0, X_1]\to\mathcal{N}$. The latter results from the boundedness of the operators $P:X_0\to\mathcal{N}$ and $P:X_1\to\mathcal{N}$ (or from \eqref{def-P} and the continuous embedding  $\mathfrak{F}[X_0, X_1]\hookrightarrow X_0+X_1$).
Interpolating \eqref{l42Rrestroper1} and \eqref{l42Qrestroper1}, we get the bounded operators
\begin{equation}\label{l42intRoperator}
R:\mathfrak{F}[X_0, X_1]\rightarrow
\mathfrak{F}[X_0/\mathcal{N},X_1/\mathcal{N}]
\end{equation}
and
\begin{equation}\label{l42intQoperator}
Q:\mathfrak{F}[X_0/\mathcal{N},X_1/\mathcal{N}]\rightarrow \mathfrak{F}[X_0, X_1].
\end{equation}
Owing to \eqref{l42intQoperator} and \eqref{l42Rrestroper2} we have the continuous embedding
\begin{equation}\label{l42RQoperator}
I=RQ:\mathfrak{F}[X_0/\mathcal{N}, X_1/\mathcal{N}]\hookrightarrow
(\mathfrak{F}[X_0, X_1])/\mathcal{N}.
\end{equation}
The inverse follows from \eqref{l42Qrestroper2} and \eqref{l42intRoperator}. We have proved \eqref{l42intequiv}.

\emph{Step~$4$.} Let us show that the range of isomorphism \eqref{l42intoperator1} is the intersection of $\mathfrak{F}[Y_0,Y_1]$ and $T(X_0+X_1)$. Given $j\in\{0,1\}$, we let $\mathcal{P}_j$ denote the projector of $Y_j$ on $T_j(Y_j)$ parallel to $\mathcal{M}$. This projector is well defined owing to condition~(iv) and is a bounded operator on $Y_j$ due to Closed Graph Theorem. Note that
\begin{equation}\label{l42f2}
\mathcal{P}_{0}f=\mathcal{P}_{1}f\quad\mbox{whenever}\quad
f\in Y_0\cap Y_1.
\end{equation}
Indeed, any vector $f\in Y_0\cap Y_1$ admits the representation
\begin{equation}\label{l42f3}
h_0+T_0u_0=f=h_1+T_1u_1
\end{equation}
for certain $h_0,h_1\in\mathcal{M}$, $u_0\in X_0$, and $u_1\in X_1$. Hence,
\begin{equation*}
T_0u_0=h_1-h_0+T_1u_1 \in T_0(X_0)\cap Y_1\subset T_1(X_0\cap X_1)
\end{equation*}
or
\begin{equation*}
T_1u_1=h_0-h_1+T_0u_0 \in T_1(X_1)\cap Y_0\subset T_0(X_0\cap X_1)
\end{equation*}
according to condition~(v). Then $h_1-h_0\in\mathcal{M}\cap T_1(X_1)$ or $h_0-h_1\in\mathcal{M}\cap T_0(X_0)$, which implies by condition~(iv) that $h_1-h_0=0$. Thus, $\mathcal{P}_0f=T_0u_0=T_1u_1=\mathcal{P}_1f$ in view of \eqref{l42f3}.

Consider the linear mapping
\begin{equation}\label{P+}
\mathcal{P}:f_0+f_1\mapsto\mathcal{P}_0f_0+\mathcal{P}_1f_1,
\quad\mbox{where}\quad f_0 \in Y_0\quad\mbox{and}\quad
f_1 \in Y_1.
\end{equation}
It is well defined by \eqref{l42f2} and is a bounded operator on $Y_0+Y_1$ that projects $Y_0+Y_1$ on $T(X_0+X_1)$. Hence, $T(X_0+X_1)$ is a closed subspace of $Y_0+Y_1$ (and the quasi-Banach spaces $T(X_0+X_1)$ and $T_0(X_0)+T_1(X_1)$ are equal up to equivalence of quasi-norms by Open Mapping Theorem.) The projectors $\mathcal{P}_0$ and $\mathcal{P}_1$ are restrictions of $\mathcal{P}$. We will show that this implies
\begin{equation}\label{l42rangeinterpol}
\mathfrak{F}[T_0(X_0),T_1(X_1)]=\mathfrak{F}[Y_0,Y_1]\cap T(X_0+X_1)
\end{equation}
up to equivalence of quasi-norms. Here, the right-hand side of the equality is a closed subspace of $\mathfrak{F}[Y_0,Y_1]$ due to the continuous embedding $\mathfrak{F}[Y_0,Y_1]\hookrightarrow Y_0+Y_1$.

Let us substantiate \eqref{l42rangeinterpol}. (In the case of Banach spaces, this interpolation formula follows from \cite[Theorem~1.17.1/1]{Triebel95} in view of condition~(v).) Let $\mathcal{Z}$ denote the the right-hand side of \eqref{l42rangeinterpol} considered as a subspace of $\mathfrak{F}[Y_0,Y_1]$. Interpolating the bounded operators $\mathcal{P}_0:Y_0\to T_0(X_0)$ and $\mathcal{P}_1:Y_1\to T_1(X_1)$, we conclude that a restriction of $\mathcal{P}$ to $\mathfrak{F}[Y_0, Y_1]$ is a bounded operator
\begin{equation*}
\mathcal{P}:\mathfrak{F}[Y_0, Y_1]\rightarrow
\mathfrak{F}[T_0(X_0), T_1(X_1)],
\end{equation*}
with $\mathcal{P}f=f$ for every $f\in\mathcal{Z}$. This yields the continuous embedding
\begin{equation}\label{l42embedding1}
\mathcal{Z}\hookrightarrow  \mathfrak{F}[T_0(X_0), T_1(X_1)].
\end{equation}
On the other hand, interpolating the operators of the continuous embeddings $T_0(X_0)\hookrightarrow Y_0$ and $T_1(X_1)\hookrightarrow Y_1$, we get another continuous embedding
\begin{equation*}
\mathfrak{F}[T_0(X_0),T_1(X_1)]\hookrightarrow \mathfrak{F}[Y_0, Y_1].
\end{equation*}
However,
\begin{equation*}
\mathfrak{F}[T_0(X_0), T_1(X_1)]\subset T_0(X_0)+T_1(X_1)=T(X_0+X_1).
\end{equation*}
Hence,
\begin{equation}\label{l42embedding3}
\mathfrak{F}[T_0(X_0), T_1(X_1)]\subset\mathcal{Z}.
\end{equation}
Now \eqref{l42embedding1} and \eqref{l42embedding3} give \eqref{l42rangeinterpol} in view of Open Mapping Theorem.

\emph{Step~$5$.} Owing to \eqref{l42intequiv} and \eqref{l42rangeinterpol}, the isomorphism \eqref{l42intoperator1} acts between the spaces
\begin{equation}\label{l42IsomorphismInterpol}
T:(\mathfrak{F}[X_0,X_1])/N\leftrightarrow\mathfrak{F}[Y_0,Y_1]\cap T(X_0+X_1).
\end{equation}
This yields property~(c). It follows from \eqref{P+} that $\mathcal{I}-\mathcal{P}$ is the projector of $Y_0+Y_1$ on $\mathcal{M}$, with $\mathcal{I}$ denoting the identity operator on $Y_0+Y_1$. Hence, $\mathcal{I}-\mathcal{P}$ and $\mathcal{P}$ induces the splitting
\begin{equation}\label{T-range-M}
Y_0+Y_1=\mathcal{M}\dotplus T(X_0+X_1),
\end{equation}
which implies property~(d) by~(c).
According to (a) and (d), the operator $T_{\mathfrak{F}}$ is Fredholm and satisfies property (b), as
\begin{equation*}
\mathrm{ind}\,T_{\mathfrak{F}}=\dim\mathcal{N}-\dim\mathcal{M}=\kappa
\end{equation*}
in view of conditions (i) (ii), and (iv).
\end{proof}

\begin{remark}\label{rem-Fredholm-T}
The operator $T:X_0+X_1\to Y_0+Y_1$ in Lemma~\ref{prop-interp-Fredholm} is obtained by the interpolation of the operators $T_{0}$ and $T_{1}$ by the trivial interpolation functor $\mathfrak{F}:[E_{0},E_{1}]\mapsto E_{0}+E_{1}$ defined on the category of all interpolation pairs $[E_{0},E_{1}]$ of quasi-Banach spaces. Hence the bounded operator $T$ is Fredholm and possesses the properties $\ker T=\mathcal{N}$, $\mathrm{ind}\,T=\kappa$, and \eqref{T-range-M}. Note that we have previously proved them on Steps 1 and~4.
\end{remark}

\begin{remark}
As to condition (v) in Lemma~\ref{prop-interp-Fredholm}, it is worthwhile to note that
\begin{equation}\label{equivalence-v}
Y_1\cap T_{0}(X_0)\subset T_{0}(X_0\cap X_1)\,\Longleftrightarrow\,
Y_0\cap T_{1}(X_1)\subset T_{1}(X_0\cap X_1)
\end{equation}
under conditions (iii) and (iv). Indeed, assuming, e.g., that the left-hand side of \eqref{equivalence-v} holds true and choosing a vector $f\in Y_0\cap T_{1}(X_1)$ arbitrarily, we represent it in the form  $f=h+T_{0}u_{0}=T_{1}u_{1}$ for certain $h\in\mathcal{M}$, $u_{0}\in X_0$, and $u_{1}\in X_1$ and conclude that
\begin{equation*}
T_{1}u_{1}-h=T_{0}u_{0}\in Y_1\cap T_{0}(X_0)\subset T_{0}(X_0\cap X_1)=
T_{1}(X_0\cap X_1)
\end{equation*}
in view of (iii) and (iv). Hence, $h\in \mathcal{M}\cap T_{1}(X_{1})=\{0\}$ and $f=T_{1}u_{1}\in T_{1}(X_0\cap X_1)$.
\end{remark}

\begin{remark}
Although condition (v) may seem superfluous in Lemma~\ref{prop-interp-Fredholm}, it is essential and cannot be removed or replaced with the weaker assumption that $Y_1\cap T_{0}(X_0)\subset T_{1}(X_1)$ and $Y_0\cap T_{1}(X_1)\subset T_{0}(X_0)$. This is justified by the following simple example: let $E$ be a quasi-Banach space,  $X_0=\{(x_0,0):x_0\in E\}\subset E^{2}$, $X_1=\{(0,x_{1}):x_{1}\in E\}\subset E^{2}$, $Y_0=Y_{1}=E$, $T_{0}:(x_0,0)\mapsto x_0$, $T_{1}:(0,x_{1})\mapsto x_{1}$. In this case, we have the isomorphisms $T_0:X_0\leftrightarrow Y_0$ and $T_1:X_1\leftrightarrow Y_1$, which satisfy conditions (i)--(iv) and the above assumption. (These operators do not satisfy condition~(v) because $X_0\cap X_1=\{0\}$.) However, the operator $T$ is defined by the formula $T:(x_0,x_1)\mapsto x_0+x_1$ and, hence, is not an isomorphism between $X_{0}+X_{1}=E^{2}$ and $Y_0+Y_1=E$, which contradicts Remark~\ref{rem-Fredholm-T}. If $\dim E=\infty$, this operator is not even Fredholm.
\end{remark}

\begin{proof}[Proof of Corollary~$\ref{cor-interp-Fredholm}$.]
We only need to show that the operators $T_0$ and $T_1$ satisfy hypothesis (iv) of Lemma~$\ref{prop-interp-Fredholm}$. (Hypothesis (v) is satisfied because the second inclusion in it is trivial.) Consider the adjoint operators $T_0'$ and $T_1'$ to $T_0$ and $T_1$, resp, and note that $\ker T_0'\subset \ker T_1'$ due to hypothesis (iii) and the dense continuous embeddings $X_1\hookrightarrow X_0$ and $Y_1\hookrightarrow Y_0$. Moreover,
\begin{equation*}
d:=\dim\ker T_0'=\dim\ker T_1'=\dim\mathcal{N}-\kappa
\end{equation*}
owing to  hypotheses (i) and (ii). Hence, $\ker T_0'=\ker T_1'$. This implies that
\begin{equation}\label{proof-corr-inclusion}
Y_1\cap T_{0}(X_0)\subset T_{1}(X_1).
\end{equation}
Indeed, if $f\in Y_1\cap T_{0}(X_0)$, then $\omega(f)=0$ for every functional $\omega\in \ker T_0'=\ker T_1'$, which yields $f\in T_{1}(X_1)$. Let $\mathcal{M}$ be a complement of $T_1(X_1)$ in $Y_1$. Such a complement exists; we may put $\mathcal{M}:=\mathrm{span}\{f_{1},\ldots,f_{d}\}$, where the cosets
$\{f_{j}+h:h\in T_1(X_1)\}$, with $j=1,\ldots,d$, form the basis of $Y_1/T(X_1)$. Then $Y_0=\mathcal{M}\dotplus T_0(X_0)$ because $\dim M=d=\dim(Y_0/T_0(X_0))$ and since
\begin{equation*}
\mathcal{M}\cap T_0(X_0)=\mathcal{M}\cap T_{1}(X_1)=\{0\}
\end{equation*}
by \eqref{proof-corr-inclusion}. Thus, $T_0$ and $T_1$ satisfy hypothesis (iv).
\end{proof}

\subsection*{Acknowledgments.} This work was funded by the National Academy of Sciences of Ukraine. The authors were supported by the European Union's Horizon 2020 research and innovation programme under the Marie Sk{\l}odowska-Curie grant agreement No~873071 (SOMPATY: Spectral Optimization: From Mathematics to Physics and Advanced Technology). The first named author was also supported by a grant from the Simons Foundation (1030291, I.S.C., and 1290607, I.S.C.) and by Universities-for-Ukraine Non-residential Fellowship with a Johns Hopkins University.

\end{document}